\documentclass[12pt]{amsart}

\usepackage{amsmath,amssymb,amsthm,verbatim, enumitem}
\usepackage{hyperref}
\usepackage[letterpaper,margin=1.1in]{geometry}

\newtheorem{theorem}[subsection]{Theorem}
\newtheorem{lemma}[subsection]{Lemma}

\newtheorem{question}[subsection]{Question}

\theoremstyle{definition}
\newtheorem{definition}[subsubsection]{Definition}
\newtheorem{remark}[subsection]{Remark}

\newcommand{\dist}{\mathrm{dist}}

\newcommand{\haus}{\mathcal{H}}

\newcommand{\spt}{\mathrm{spt}}

\newcommand{\Ric}{\mathrm{Ric}}
\newcommand{\reg}{\mathrm{reg}}

\newcommand{\graph}{\mathrm{graph}}

\newcommand{\eps}{\epsilon}

\newcommand{\sing}{\mathrm{sing}}
\newcommand{\cF}{\mathcal{F}}
\newcommand{\R}{\mathbb{R}}
\newcommand{\bC}{\mathbf{C}}

\newcommand{\del}{\partial}

\newcommand{\cM}{\mathcal{M}}
\newcommand{\bY}{\mathbf{Y}}

\newcommand{\cS}{\mathcal{S}}

\newcommand{\cI}{\mathcal{I}}
\newcommand{\cV}{\mathcal{V}}
\newcommand{\cA}{\mathcal{A}}
\newcommand{\cC}{\mathcal{C}}
\newcommand{\cISV}{\mathcal{ISV}}

\newcommand{\N}{\mathbb{N}}

\newcommand{\cE}{\mathcal{E}}
\newcommand{\cR}{\mathcal{R}}

\newcommand{\Lip}{\mathrm{Lip}}

\newcommand{\cG}{\mathcal{G}}
\newcommand{\mindex}{\mathrm{index}}
\newcommand{\geucl}{g_{\mathrm{eucl}}}

\title[Degeneration of $7$-dimensional minimal surfaces]{Degeneration of $7$-dimensional minimal hypersurfaces which are stable or have bounded index}
\author{Nick Edelen}

\address{Department of Mathematics, University of Notre Dame, Notre Dame, IN, 46556}
\email{nedelen@nd.edu}

\begin{document}

\begin{abstract}
A $7$-dimensional area-minimizing embedded hypersurface $M^7$ will in general have a discrete singular set, and the same is true if $M$ is locally stable provided $\haus^6(\sing M) = 0$.  We show that if $M_i^7$ is a sequence of 7D minimal hypersurfaces which are minimizing, stable, or have bounded index, then $M_i \to M$ can limit to a singular $M^7$ with only very controlled geometry, topology, and singular set.  We show one can always ``parameterize'' a subsequence $i'$ with controlled bi-Lipschitz maps $\phi_{i'}$ taking $\phi_{i'}(M_{1'}) = M_{i'}$.  As a consequence, we prove the space of smooth, closed, embedded minimal hypersurfaces $M$ in a closed Riemannian $8$-manifold $(N^8, g)$ with a priori bounds $\haus^7(M) \leq \Lambda$ and $\mindex(M) \leq I$ divides into finitely-many diffeomorphism types, and this finiteness continues to hold if one allows the metric $g$ to vary, or $M$ to be singular.
\end{abstract}

\maketitle

\section{Introduction}

We are interested in the space of $7$-dimensional embedded minimal hypersurfaces which are area-minimizing, stable, or have bounded index.

An $n$-dimensional area-minimizing hypersurface $M^n$ will in general have a $(n-7)$-dimensional singular set $\overline{M} \setminus M$, where by $\overline{M}$ we denote set-theoretic closure.  The same regularity holds if $M$ is locally stable for the area-functional \cite{ScSi}, \cite{neshan}, provided one knows a priori that $\haus^{n-1}(\overline{M} \setminus M) = 0$ (that is, $M$ has no junctions modeled on half-planes meeting along an edge).  A significant goal in minimal surface theory is to understand the nature of the singular set, and the structure of $M$ nearby.

Deep work of \cite{Simon2}, \cite{NaVa} have shown that the singular set of $M^n$ as above is in fact countably-$(n-7)$-rectifiable, with locally-finite $(n-7)$-dimensional area.  Rectifiability is the best one can hope for, as \cite{simon:sing} has constructed remarkable examples of stable, embedded minimal $(n \geq 8)$-dimensional hypersurfaces in smooth manifolds $(\R^{n+1}, g)$ with singular set being any arbitrary preassigned closed subset of $\{0^8\} \times \R^{n-7}$.  Moreover, at scale $1$ these minimal surfaces can be made to look arbitrarily varifold-close to $(\text{Simons' cone}) \times \R^{n-7}$, and the metric $g$ made to look arbitrarily smoothly close to Euclidean.

The examples of \cite{simon:sing} suggest there is little hope, for smooth metrics at least, of understanding the fine-scale structure of an area-minimizing hypersurface $M^n$ from only the knowledge of its tangent cones or its behavior at large scales.  Said differently, if $M_i$ is a sequence of area-minimizing $(n \geq 8)$-dimensional hypersurfaces varifold-converging to some $M$, then one cannot hope to understand the regular or singular structure of $M_i$ based on $M$.

On the other hand, we shall show in this paper that these questions have satisfactory answers in the ``edge case'' when $n = 7$.  (If $n \leq 6$ then area-minimizing $M_i, M$ are all regular, and convergence $M_i \to M$ is always smooth.)  In dimension $n=7$, the singular set is discrete, and any tangent cone has a smooth link.  Profound works of \cite{AllAlm}, \cite{simon} imply that near any such isolated singular point, $M$ is a $C^1$ perturbation of its tangent cone, giving a very precise fine-scale picture of $M$ at every point.

The Hardt-Simon foliation \cite{HaSi}, \cite{BDG} gives examples of smooth, entire, area-minimizing hypersurfaces in $\R^8$ which are asymptotic to singular cones.  So, somewhat analogous to the $n \geq 8$ case, two area-minimizing $7$-hypersurfaces that look very close at scale $1$ can have very different regular or singular structure at scale $\eps$.  However, in sharp contrast to the higher-dimensional case, when $M^7$ resembles a Simons cone $\bC^{3,3}$ or $\bC^{2,4}$ one can still give a precise description of $M$: \cite{me-luca} (see also \cite{SiSo}) has shown that, besides the cone itself, leaves of the foliation are the only other possible pictures for $M$.  In other words, if $M_i^7$ limit to a singular $M^7$, then near any singularity of $M$ modeled on a minimizing quadratic cone, the $M_i$ are locally $C^{1,\alpha}$ perturbations of either the cone, or a leaf of the foliation.

Our goal is to prove a similar result for arbitrary stable $7$-dimensional singularities.  Precisely, we answer the following question:
\begin{question}\label{q:main}
Let $g_i \to g$ be a sequence of metrics on $B_1(0) \subset \R^8$, and $M_i^7$ be a sequence of stable embedded minimal hypersurfaces in $(B_1, g_i)$ with discrete singular set $\overline{M_i} \setminus M_i$, which converge as varifolds in $B_1$ to some minimal surface $M$ in $(B_1, g)$.  How can the geometry and topology of the $M_i$ degenerate near a singular point of $M$?
\end{question}

Near any regular point of $M$, the sheeting theorem of \cite{ScSi} implies the $M_i$ converge smoothly with possible multiplicity, and the a priori bound of \cite{NaVa} implies that only finitely-many singular points of $M_i$ can collapse into a singularity of $M$.  So the interesting behavior is in the regular structure of the $M_i$ as it collapses into a singular point of $M$ (we mention that our results do not rely on the estimates of \cite{NaVa}, but rather obtain \cite{NaVa}'s bound in dimension $7$ as a corollary of our main theorem).

Aside from being interesting in its own right, understanding how a sequence of minimal surfaces $M_i$ can behave as they approach a possibly singular limit is important for several reasons.  If one wants to answer quantitative or global questions about moduli spaces of minimal surfaces (e.g. how many diffeomorphism types are there, which geometric quantitaties are bounded, what sort of compactness theory holds), one needs to know how the geometry of the $M_i$ can degenerate along a limit.  If one wants to answer questions about the singular set (e.g. bounding the dimension or measure of certain strata), one needs to know how these singularities can behave when passing to limits.

Even if one is attempting to understand the behavior of a \emph{fixed} minimal surface $M^n$ near some singular point, if a tangent cone $\bC$ at that point itself has more than one singularity, then one is confronted with variants of Question \ref{q:main}.  For example, if $n = 8$ and for some sequence of dilations $r_i \to 0$, the rescaled $M_i := (M - p)/r_i$ converge to $\bC = \bC^{3,3}\times \R$, then near any point in $\bC$ away from the origin, the $M_i$ will resemble a largely arbitrary sequence of minimal surfaces converging to $\bC$.

\vspace{3mm}

We answer Question \ref{q:main} in Theorem \ref{thm:main3}, where we show that the geometry, topology, and singular set of the $M_i$ is precisely controlled near any singular limit point of $M$.  In fact we answer Question \ref{q:main} assuming the $M_i$ are only locally stable, but with a uniform bound on the index of the stability operator.  It turns out the techniques required to deal with stable singularities extend with relatively little effort to deal with points of index concentration.

We prove Theorem \ref{thm:main3} by demonstrating a ``cone-decomposition'' for the $M_i$, analogous to the bubble-tree decompositions for harmonic maps \cite{SaUhl}, \cite{parker}, or ``neck-decompositions'' for GH-limits of spaces with non-negative Ricci \cite{cheeger-naber}, \cite{naber-jiang}: we show there is a finite, predetermined pool of ``smooth models'' (entire minimal surfaces in $\R^8$ which are minimizing, stable, or having bounded index), so that every $M_i$ breaks up into a controlled number of annular ``cone regions'' wherein $M_i$ is a multi-graph over some fixed stable cone, and a controlled number of ``smooth regions'' wherein $M_i$ is a (scaled-down) multi-graph over one of the smooth models.

Unlike the finiteness theorems of \cite{ckm}, \cite{cheeger-naber}, we have no classification of tangent cone beyond being stable and with smooth cross section, and in particular we cannot assume our cones are ``integrable through rotations.''  A key technical issue we must overcome therefore is demonstrating that the cone regions do not ``rotate'' through families of possibly pathological cones.  We do this in Sections \ref{sec:cone-region}, \ref{sec:decay}, using the Lojasiewicz-Simon inequality of \cite{simon}.

Additionally, we show that, up to passing to a subsequence, we can find bi-Lipschitz maps $\phi_i$ on the ambient space (e.g. the unit ball $B_1$), which ``parameterize'' the $M_i$ in the sense that $\phi_i(M_1) = M_i$, and which have uniformly bounded Lipschitz constant.  Due to the presense of multiplicity, the construction of these maps (Section \ref{sec:param}) turns out to be fairly technical.

A consequence of our theorem is the following finitness result for minimal hypersurface in an $8$-dimensional closed manifold. See Theorem \ref{thm:main2} for a more precise version, which additionally allows the metric to change and gives quantitative bounds on the diffeomorphisms/bi-Lipschitz maps.
\begin{theorem}\label{thm:main1}
Let $(N^8, g)$ be a closed, $8$-dimensional Riemannian manifold with $C^3$ metric $g$, and take $\Lambda, I \geq 0$.  There is a number $K(N, g, I, \Lambda)$ so that every $C^2$, closed, embedded minimal hypersurface $M \subset (N, g)$ having $\haus^7(M) \leq \Lambda$ and $\mindex(M) \leq I$ fits into one of at most $K$ diffeomorphism classes.  If $M$ is not closed but $\overline{M} \setminus M$ is discrete, then $\overline{M}$ fits into one of at most $K$ bi-Lipschitz equivalence classes.
\end{theorem}

The only known area-minimizing hypercones in $\R^8$ are the Simons' cones $\bC^{2,4}$ and $\bC^{3,3}$, so for area-minimizing $M_i^7$ it is conceivable that Question \ref{q:main} can be entirely addressed using the methods of \cite{me-luca}.  On the other hand, the Simons' cone $\bC^{1,5} \subset \R^8$ is known to be stable but not area-minimizing (it is minimizing on only one side), and more generally \cite{me-luca} cannot handle higher-multiplicity singularity models.  To compare and contrast: in \cite{me-luca} we address only multiplicity-one area-minimizing quadratic cones, but we do so in any dimension, and we can give a priori $C^{1,\alpha}$ bounds. In this paper, we can deal with any stable hypercone with possible multiplicity, but only for $n = 7$, and we can only give a priori $C^1$ bounds on our maps.

In general there is not much progress on variants of Question \ref{q:main} for singular surfaces, in large part because (as illustrated by \cite{simon:sing}'s examples) the singular or regular nature of $M_i$ can drastically change in the limit.  A notable exception is when singularities of $M$ are modeled on a union of three half-planes.  This kind of singularity does not typically arise from the ``usual'' minimization via integral currents or sets of locally-finite-perimeter, but does arise naturally in so-called $(\mathbf{M}, \eps, \delta)$-minimizers.  \cite{Simon1} has shown that if $M_i \to M$, and every singularity of $M$ has this type, then the $M_i$  are $C^{1,\alpha}$ perturbations of $M$.  These ``$\bY$-type'' singularities are very special because they enforce the same singular structure on $M_i$.  In a similar vein, \cite{CoEdSp} showed the same regularity holds if $M$ has ``tetrahedral-type'' singularities, and the $M_i$ secretly have an orientation structure attached to them.

All the above mentioned results (like those in this paper) are local, allowing for $M_i$ to be minimal in only a unit ball $(B_1, g_i)$, for some sequence of metric $g_i$ close to Euclidean.  Let us mention that \cite{wang1} has proved a \emph{global} isolation theorem for certain $M$ in a (fixed) closed $(N^{n+1}, g)$.  Analogous to \cite{Wh2}'s results for strictly stable smooth minimal surfaces, \cite{wang1} showed that if $M$ is strictly stable and has only isolated singularities modelled on strictly-minimizing smooth cones, then $M$ is area-minimizing in some neighborhood $U$, and in fact is the unique multiplicity-one minimal surface in $(U, g)$, so if $M_i \to M$ are all minimal in $(N, g)$ then one must have $M_i = M$ for $i >> 1$.  (If $M$ is not strictly stable, or some singularity is not strictly minimizing, then this isolation can fail).

\vspace{3mm}

The past several years has seen significant study of the behavior of $(n\leq 6)$-dimensional embedded minimal hypersurfaces with bounded index (see e.g. \cite{Sh}, \cite{ckm}, \cite{BuSh}, \cite{carlotto}, \cite{li:index}, \cite{song}), motivated largely by various min-max constructions.  In this case, if $M_i \to M$ and the $M_i$ have uniformly bounded index $I$, then $M$ is entirely smooth, but convergence may fail to be smooth at a collection of at most $I$ points of ``index concentration.''  Question \ref{q:main} in this setting has been very satisfactorily answered by \cite{ckm}, \cite{BuSh}, using variants of a bubble-tree type decomposition.  \cite{ckm}'s approach is well-adapted for $M_i$ without area bounds, while \cite{BuSh} gives additional information on total curvature and curvature concentration.  Most of the methods in these works rely very heavily on the planar nature of singularities and smoothness of the limit.  Our theorems specialize to the $n \leq 6$ case, and give an alternate method of proof to several of the results in \cite{ckm}, \cite{BuSh}, but we emphasize that the real interest in this work is for singular $M$.

Singular $7$-dimensional minimal hypersurfaces with bounded index arise naturally in both Almgren-Pitts (\cite{li1}) and Allen-Cahn (\cite{hiesmayr}, \cite{gaspar}) min-max constructions.  By carefully understanding the behavior of the min-max surface and sweepout near its singular points, \cite{lov} and \cite{li-wang} have recently shown that for a generic metric on a closed $N^8$ there is a smooth, closed, embedded minimal hypersurface.  In a similar vein, there are numerous perturbative existence results, which state that given a suitable minimizing or stable $M^7 \subset (N^8, g)$, then by slightly perturbing the boundary of $M$ (\cite{HaSi}) or the ambient metric $g$ (\cite{Smale}, \cite{wang1}), one can obtain a smooth $7$-dimensional minimizing or stable hypersurface.

\cite{White_b2} has shown that for a closed manifold $N$, the space of ordered pairs $(g, M)$ consisting of $C^k$ metrics $g$ on $N$, and $C^{j,\alpha}$ minimal surfaces $M$ in $(N, g)$, can be given a Banach manifold structure, and moreover that the projection $(g, M) \to g$ is Fredholm with index $0$.  \cite{White_b2}'s theorem implies some remarkable genericity and finiteness results, for example: for a generic ``bumpy'' $C^3$ metric $g$ on $N^7$ (in the sense of Baire) there exist only finitely-many stable hypersurfaces $M^6 \subset (N,g)$ with $\haus^6(M) \leq \Lambda$.  White's ``bumpy'' metric theorem is a major ingredient in many generic min-max results.

A natural question garnering recent interest is whether a similar manifold structure holds on the spaces of pairs of metrics and \emph{possibly singular} minimal surfaces.  \cite{wang1} has made some headway on this question, by developing a robust linear theory for Jacobi fields on minimal hypersurfaces with isolated singularities, and using this to show that for certain $M^n \subset (N^{n+1}, g)$ (having bounded index, only strongly-isolated, strictly-minimizing singularities, and a few other technical conditions), every $C^4$ metric peturbation $g'$ admits a minimal surface $M' \subset (N, g')$ which is close to $M$.

One can view our Theorem \ref{thm:main2} as a step towards the manifold-structure problem from the opposite perspective, by showing that nearby metric-surface pairs $(g', M'^7)$ with controlled index in a closed $N^8$ can be ``parameterized'' using finitely-many model surfaces.  It would be very interesting if bumpiness results similar to those of \cite{White_b2} held in the singular setting also, e.g. if for a generic $C^3$ metric $g$ on a closed $N^8$, there are only finitely-many area-minimizing or stable hypersurfaces $M^7 \subset (N,g)$ satisfying $\haus^7(M) \leq \Lambda$ and $\dim(\sing M) = 0$.

\vspace{3mm}

Finally, let us mention that the basic ideas that go into proving the main Theorem \ref{thm:diffeo} are fairly general, relying only on good compactness and partial regularity theorems, a monotonicity formula, and the existence of a Lojasiewicz-Simon-type inequality for singular models.  One could likely adapt the method of this paper to prove (for example) that there are only finitely-many isotopy classes of smooth energy-minimizing maps $(M^3, g_M) \to (N^n, g_N)$ having energy $ \leq \Lambda$, where $M^3$ is a closed $3$-dimensional manifold with smooth metric $g_M$, and $N^n$ a closed $n$-dimensional manifold with analytic metric $g_N$.

\vspace{3mm}

\textbf{Acknowledgements:}  I am indebted to Otis Chodosh, Luca Spolaor, Gabor Szekelyhidi, and Brian White for many useful discussions, and to Andrew Putman for patiently answering my naive questions about mapping class groups.  Part of this work was conceived while visiting Stanford University and Big Basin State Park.  I thank them both for their hospitality, and wish Big Basin a speedy recovery.

\section{Main theorems} \label{sec:thm}


To more easily state our results we require a little notation.  Let $(N^8, g)$ a Riemannian manifold with $C^3$ metric $g$.  Since many of our results involve changing the metric, we will typically explicitly specify the metric dependence.  Let us write $\haus^7_g$ for the $7$-dimensional Hausdorff measure in $(N, g)$, $\Ric_g$ for the Ricci curvature, $B_r^g(a)$ for the open geodesic ball centered at $a$, and $d_{H, g}$ for the Hausdorff distance.  If $\phi : N \to N$, we write $\Lip_g(\phi)$ for the Lipschitz constant of $\phi$ w.r.t. $g$.  If $U \subset N$, then $\overline{U}$ denotes set-theoretic closure.

If $N = U \subset \R^8$, all measurements, derivatives, norms, etc. will be taken w.r.t. the Euclidean metric $\geucl$ unless explicitly stated otherwise.  So, if $N \subset \R^8$, $\haus^7$ will always denote the Euclidean Hausdorff measure, $B_r(a)$ the Euclidean open $r$-ball centered at $a$, $A_{R, r}(a) = B_R(a) \setminus \overline{B_r(a)}$ for the open annulus, $d_H$ the Euclidean Hausdorff distance, and $\Lip(\phi)$ for the Lipschitz constant w.r.t. $\geucl$.  Similarly, if $M$ is a $C^1$ embedded submanifold of $\R^8$, we write $T^\perp M$ for the Euclidean normal bundle of $M$, $u : M \to M^\perp$ for a section of this normal bundle, and $\nabla$ the connection induced by the Euclidean connection on $T^\perp M$.

If $M$ is a $C^1$ embedded minimal hypersurface of $(N, g)$, define $\reg M$ to be the set of points $x \in \overline{M}$ having the property that in some open neighborhood $U$ of $x$, $\overline{M} \cap U$ is a closed (as sets) embedded submanifold of $U$.  Define $\sing M = \overline{M} \setminus \reg M$.  Of course $\sing M \subset \overline{M} \setminus M$, but there could be points in $\overline{M} \setminus M$ that ``secretly'' are regular points, and we want to make the notion of singular set canonical.  There is no loss in always assuming that $M = \reg M$.

Let us write $T^{\perp_g} M$ for the normal bundle of $M$ with respect to $g$, and $u : M \to M^{\perp_g}$ for a section of this normal bundle, and $\nabla^g$ the connection induced by $g$ on the normal bundle.  For $M$ being $C^2$, write $A_{M,g}$ for the second fundamental form of $M \subset (N, g)$.  We define $\mindex(M, N, g)$ to be the maximal dimension of subspace $\subset C^1_c(M, M^{\perp_g})$ on which the quadratic form
\begin{equation}\label{eqn:Q}
Q_{M, g}(u, u) = \int_M |\nabla^g u|^2 - |A_{M, g}|^2 |u|^2 - \Ric_g(u, u) d\haus^7_g 
\end{equation}
is strictly negative definite.  If $M$ is stable for the area functional in $(N, g)$, then $Q_{M, g} \geq 0$.

Define $\cM_7(N^8, g)$ to be the space of $C^2$ embedded minimal hypersurfaces in $(N, g)$ with the property that $\overline{M} \setminus M$ consists of a discrete set of points.  It follows by work of \cite{ScSi}, \cite{neshan} that if $M$ is any $C^2$ embedded minimal hypersurface in $(N, g)$ satisfying $\mindex(M, N, g) < \infty$ and $\haus^6(\overline{M} \setminus M) = 0$, then $M \in \cM_7(N, g)$.  See Section \ref{sec:prelim} for more details.  We define $\cC$ to be the set of smooth (away from $0$), stable minimal hypercones in $\R^8$.

If $M$ is a $C^1$ $n$-dimensional embedded hypersurface in $\R^{n+1}$, we define $\theta_M(a, r) = \omega_n^{-1} r^{-n} \haus^n(M \cap B_r(a))$ for the density ratio of $M$.  More generally, if $V$ is an $n$-varifold with mass measure $\mu_V$, then $\theta_V(a, r) = \omega_n^{-1} r^{-n} \mu_V(B_r(a))$.  If $M = \bC$ is a cone, then we write $\theta_\bC(0) = \theta_\bC(0, 1)$.  If $M$ is a $C^1$ embedded hypersurface in $(N^8, g)$, write $[M]_g$ for the $7$-varifold obtained by integrating over $M$, i.e. with mass measure $\mu_{[M]_g} = \haus^7_g \llcorner M$.

\vspace{3mm}

Our main result is local, and has two parts: first, Theorem \ref{thm:diffeo}, which builds a ``cone decomposition'' for a minimal surface $M$ with bounded index and area sufficiently nearby singular cone; and second, Theorem \ref{thm:param}, which ``parameterizes'' minimal surfaces admitting cone decompositions.  Together, they say:

\begin{theorem}\label{thm:main3}
Given $\Lambda, I \geq 0$, and $\sigma, \beta \in (0, \frac{1}{10^3 (I+1)})$, then there is a finite collection of  ``$(\Lambda, \beta, \sigma)$-smooth models'' $\cS$ (Definition \ref{def:model}), and constants $K$, $\delta > 0$, all depending only on $(\Lambda, I, \sigma, \beta)$, so that the following holds.

Take $\bC \in \cC$ so that $\theta_\bC(0) \leq \Lambda$.  Let $g_i, g$ be $C^3$ metrics on $B_1(0)$ such that $|g_i  - \geucl|_{C^3(B_1)} \leq \delta$, $g_i \to g$ in $C^2(B_1)$, and let $M_i \in \cM_7(B_1, g_i)$ be a sequence of minimal surfaces satisfying
\[
d_H( M_i \cap B_1 , \bC \cap B_1) \leq \delta, \quad \haus^7_{g_i}(M_i \cap B_1) \leq \Lambda, \quad \mindex(M_i, B_1, g_i) \leq I .
\]
Then, after passing to a subsequence, we can find a radius $r \in (1-200 \sigma (I+1), 1)$, so that every $M_i \cap (B_r(0), g_i)$ admits a ``$(\Lambda, \beta, \cS, K)$-cone decomposition'' (Definition \ref{def:decomp}).

Moreover, passing to a further subsequence, we can find bi-Lipschitz maps $\phi_i : B_r \to B_r$ and a constant $C$ so that:
\begin{enumerate}
\item \label{item:main3-1} $\phi_i( \overline{M_1} \cap B_r) =  \overline{M_i} \cap B_r$, and $\phi_i(\sing M_1 \cap B_r) = \sing M_i \cap B_r = \sing M_i \cap B_{(1-\sigma)r}$;
\item \label{item:main3-2} $\phi_i$ restricts to a $C^2$ diffeomorphism $B_r \setminus \sing M_1 \to B_r \setminus \sing M_i$;
\item \label{item:main3-3} $\Lip(\phi_i) \leq C$;
\item \label{item:main3-4} there is a stationary integral varifold $V$ in $(B_1, g)$ of the form $V = m_1 [M_1']_g + \ldots + m_k [M_k']_g$, for $M_i' \in \cM_7(B_1, g)$ and $m_i \in \N$, so that $[M_i]_{g_i} \to V$ as varifolds in $B_1$, and in $C^2$ on compact subsets of $B_1 \setminus (\sing V \cup \cI)$ for some set $\cI$ of at most $I$ points;
\item \label{item:main3-5} there is a Lipschitz $\phi_\infty : B_r \to B_r$ such that $\phi_\infty(\overline{M_1} \cap B_r) = \spt V \cap B_r$ and $\phi_i \to \phi_\infty$ in $C^\alpha(B_r)$ for all $\alpha \in (0, 1)$.
\end{enumerate}
\end{theorem}

\begin{remark}
In certain circumstances (when singularities of the $M_i$ are ``integrable through rotations'') then one can arrange each $\phi_i$ to be globally $C^{1,\alpha}$ for some $\alpha > 0$.  Ditto for Theorems \ref{thm:main2}, \ref{thm:main4}.  See Remarks \ref{rem:c1alpha}, \ref{rem:cone-c1alpha}.
\end{remark}

\begin{remark}
Theorem \ref{thm:main3} (and Theorems \ref{thm:main2}, \ref{thm:main4}) also hold for smooth, $n$-dimensional minimal hypersurfaces, when $2 \leq n \leq 6$.  In this case we additionally obtain the curvature bounds
\begin{equation}\label{eqn:rem-tot-A}
\int_{M_i \cap B_{r}} |A_{M_i, g_i}|^n d\haus^n_{g_i} \leq C
\end{equation}
for $C$ depending only on $(\Lambda, I, \beta, \sigma)$ (recovering the total curvature bound of \cite{BuSh}).  \eqref{eqn:rem-tot-A} is a consequence of a more general Dini-type estimate, see Remark \ref{rem:dini}
\end{remark}

There are two key ideas that go into proving Theorem \ref{thm:main3} (or, rather, Theorem \ref{thm:diffeo}), both heavily indebted to the Lojasiewicz-Simon inequality \cite{simon}.  The first is that the set of densities for smooth, embedded, stable minimal cones in $\R^8$ forms a discrete set.  This allows us to prove Theorem \ref{thm:diffeo} by induction on $\Lambda$, reducing to one of the base cases of either $\bC$ being a multiplicity-one plane (in which case the result follows by Allard's theorem \cite{All}), or $\bC$ being a plane-with-multiplicity and the $M_i$ being stable (in which case the result follows by Schoen-Simon \cite{ScSi}).

The second key idea is captured in Theorem \ref{thm:scone}, where we show that if a minimal surface $M$ in $B_R(a)$ has very small density drop in some annulus $B_R(a) \setminus B_\rho(a)$, then in this annulus $M$ is graphical over some fixed cone, \emph{irrespective} of what $M$ looks like inside $B_\rho(a)$.  We need this estimate to rule out the cones ``rotating'' in annular regions obtained from Theorem \ref{thm:diffeo}.  To prove Theorem \ref{thm:scone} we adapt the decay/growth estimates of \cite{simon}, which loosely say that a non-conical minimal graph over a cone must quantitatively decay or growth either polynomially or logarithmically from one scale to the next.

The third, let's say key half-idea, involved in Theorem \ref{thm:main3} is in constructing the parameterizations $\phi_i$ (Theorem \ref{thm:param}).  The construction turns out to be fairly technical, because the presence of multiplicity means that the $\phi_i$ will only be $C^0$ perturbations of the identity, rather than $C^1$ perturbations.  So one cannot use ``naive'' gluing methods to piece together the maps, but must instead use very strongly the exact structure of the $\phi_i$.


\vspace{3mm}

Theorem \ref{thm:main3} implies the following global finiteness theorem in a closed Riemannian $8$-manifold.
\begin{theorem}\label{thm:main2}
Let $(N^{8}, g)$ be a closed $8$-dimensional Riemannian manifold with $C^3$ metric $g$.  Given $\Lambda, I \geq 0$, there is a number $\delta(N, g, \Lambda, I) > 0$ so that the following holds.

Let $\cG$ be any set of $C^3$ metrics on $N$ satisfying $|g - g'|_{C^3(N, g)} \leq \delta$ for every $g' \in \cG$.  Then we can find a constant $C(N, g, \Lambda, I, \cG)$, and \emph{finite} collections of metrics $\{g_v\}_v \subset \cG$, and minimal surfaces $\{M_v \in \cM_7(N, g_v) \}_v$, with the properties:
\[
\haus^7_{g_v}(M_v) \leq \Lambda, \quad \mindex(M_v, N, g_v) \leq I,
\]
so that if $M \in \cM_7(N, g')$ for some $g' \in \cG$ satisfies $\haus^7_{g'}(M) \leq \Lambda$, $\mindex(M, N, g') \leq I$, then there is a $v$ and a bi-Lipschitz mapping $\phi : (N, g) \to (N, g)$ such that:
\begin{enumerate}
\item \label{item:main2-1} $\phi(\overline{M_v}) = \overline{M}$, $\phi(\sing M_v) = \sing M$;
\item \label{item:main2-2} $\phi$ restricts to a $C^2$ diffeomorphism $N \setminus \sing M_v \to N\setminus \sing M$;
\item \label{item:main2-3} $\Lip_g(\phi) \leq C$.
\end{enumerate}
If $M$ is area-minimizing, then one can assume $M_v$ is area-minimizing too.
\end{theorem}

\begin{remark}\label{rem:main2}
\cite{wang1} has shown that if $(N^8, g)$ is closed and $g$ is $C^4$, then given any $\eps > 0$ and any $M \in \cM_7(N, g)$ having bounded index which is orientable, ``non-degenerate,'' and with the property that every singular point is modelled on a strictly stable and strictly minimizing cone (e.g. a minimizing Simons' cone), then there is a $C^4$ neighborhood of metrics $\cG \ni g$ so that for each $g' \in \cG$, one can find an $M' \in \cM_7(N, g')$ with bounded index satisfying $d_{H, g}(M, M') \leq \eps$.
\end{remark}

We also have a corresponding finiteness result in Euclidean space.
\begin{theorem}\label{thm:main4}
Given $\Lambda, I \geq 0$, we can find a finite collection $\{M_v\}_v \subset \cM_7(\R^8, \geucl)$, and an increasing function $C_{\Lambda, I} : \R \to \R$, with the properties
\begin{equation}\label{eqn:main4-hyp}
\theta_{M_v}(0, \infty) \leq \Lambda, \quad \mindex(M_v, \R^8, \geucl) \leq I,
\end{equation}
so that given any other $M \in \cM_7(\R^8, \geucl)$ satisfying \eqref{eqn:main4-hyp}, we can find a number $\lambda > 0$, a locally bi-Lipschitz map $\phi : \R^8 \to \R^8$, and a $v$, so that
\begin{gather}
\phi(\overline{M_v}) = \lambda  \overline{M}, \quad \phi(\sing M_v) = \lambda \sing M_v, \label{thm:main4-concl1} \\
\phi|_{\R^8 \setminus \sing M_v} \text{ is a $C^2$ diffeomorphism}, \label{thm:main4-concl2} \\
\text{and } \Lip(\phi|_{B_r(0)}) \leq C_{\Lambda, I}(r) \text{ for every $r$} . \label{thm:main4-concl3}
\end{gather}
If $M$ is area-minimzing (resp. is an area-minimizing boundary), we can assume $M_v$ is area-minimizing (resp. is an area-minimizing boundary).  If the tangent cone of $M$ at infinity is multiplicity-one (e.g. as when $M$ is an area-minimizing boundary), then $\Lip(\phi) \leq C = C_{\Lambda, I}(10)$ on all of $\R^8$.
\end{theorem}

\begin{remark}\label{rem:main4}
\cite{chan} (see also \cite{white:deg}) has shown that if $\bC^7 \in \cC$ is strictly-minimizing, then for every ``slowly-decaying Jacobi field'' $v$ of $\bC$ there is a area-minimizing hypersurface $M \subset \R^8$ which is asymptotic to $\graph_\bC(v)$.  In particular, if $I = \mindex(\bC \cap \del B_1, S^7, g_{\mathrm{round}})$, this gives an $I$-parameter family of distinct, area-minimizing hypersurfaces asymptotic to $\bC$.  Theorem \ref{thm:main4} implies that this family has bounded geometry and topology.
\end{remark}

Lastly, let us mention that our smallness estimate of Theorem \ref{thm:scone} allows us to sharpen the Corollary \cite[Corollary 3.4]{me-luca}.  For a minimal surface $M$ in a non-Euclidean metric, in \cite{me-luca} we could not rule out the possibility that at different radii $M$ resembled different rotations of the Simons' cone (which was still sufficient to characterize singular nature of $M$).  Though we still cannot get the effective $C^{1,\alpha}$ estimates of \cite[Theorem 3.1]{me-luca} in Euclidean space, here we can at least show the cones do not rotate.  See \cite{me-luca} for more details on background and notation.  The proof of Theorem \ref{thm:main5} (given in Section \ref{sec:simons}) is something of a baby version of the proof of our main Theorem \ref{thm:diffeo}, and in fact it might be helpful for the reader to read this proof first.
\begin{theorem}\label{thm:main5}
Take any $n \in \N$, and let $\bC^n \subset \R^{n+1}$ be a minimizing quadratic hypercone.  Let $\{S_\lambda\}_\lambda$ be the leaves of the Hardt-Simon foliation of $\bC$, so that $S_0 = \bC$.  Given any $\eps > 0$, there is a $\delta(\bC,\eps)$ so that the following holds.  Let $g$ be a $C^3$ metric on $B_1 \subset \R^{n+1}$ satisfying $|g - \geucl|_{C^3(B_1)} \leq \delta$.  Let $V$ be an integral stationary $n$-varifold in $(B_1, g)$ satisfying
\begin{gather}\label{eqn:main5-hyp}
d_H( \spt V \cap B_1, \bC \cap B_1) \leq \delta, \quad (1/2) \theta_\bC(0) \leq \theta_V(0, 1/2), \quad \theta_V(0, 1) \leq (3/2)\theta_\bC(0).
\end{gather}

Then we can find an $a \in \R^{n+1}$, $\lambda \in \R$, $q \in SO(n+1)$, satisfying
\begin{equation}\label{eqn:main5-concl1}
|a| + |q - Id| + |\lambda| \leq \eps,
\end{equation}
and a $C^2$ function $u : (a + q(S_\lambda)) \cap B_{1/2}(a) \to S_\lambda^\perp$ such that
\begin{equation}\label{eqn:main5-concl2}
\spt V \cap B_{1/4} = \graph_{a + q(S_\lambda)}(u) \cap B_{1/4}, \quad |x - a|^{-1} |u| + |\nabla u| + |x - a| |\nabla^2 u| \leq \eps .
\end{equation}

If $\bC^n \subset \R^{n+1}$ is a general smooth (away from $0$) strictly-stable and strictly-minimizing hypercone, then the same conclusion holds if one additionally assumes $\spt V$ lies to one side of $\bC$.
\end{theorem}

\section{Outline of proof}\label{sec:outline}

The key to our argument is the observation (by no means our own) that the set of densities of smooth stable minimal $7$-dimensional hypercones is a discrete set, which a direct consequence of the Lojasiewicz-Simon inequality \cite{simon} and standard compactness/regularity theory for stable minimal hypersurfaces.  At some level this is saying that the ``complexity'' of the singular models $\bC$ is quantized.  We prove our main decomposition Theorem \ref{thm:main3} by induction on the density (including multiplicity) of the cone $\bC$.

Our basic idea is very simple.  Unfortunately (fortunately?) there is a highly non-trivial complication, due to the nature of ``cone regions,'' but this point is not immediately obvious so let us ignore it for the moment.  Let us illustrate our idea by proving that if $M_i$ is a sequence of smooth, area-minimizing boundaries in $B_1$ converging to an area-minimizing cone $\bC$, then after passing to a subsequence all the $M_i \cap B_{1/2}$ are diffeomorphic to each other.  (So, everything is mulitplicity-one and converges smoothly on the regular part).

Let us suppose, towards a contradiction, that none of the $M_i \cap B_{1/2}$ are diffeomorphic.  Write $1 = \theta_0 < \theta_1 < \ldots$ for the densities of area-minimizing cones, and suppose $\theta_\bC(0) = \theta_k$.  If $k = 0$ then $\bC$ is the plane, and so our assertion follows by Allard's theorem.  Let us assume by induction our assertion holds for every $\theta_\bC(0) \leq \theta_{k-1}$.

The $M_i$ converge smoothly to $\bC \cap B_1$ away from $0$.  For any $a \in B_{1/2}$ we can consider the smallest radius $\rho$ so that $A_{1/2, \rho}(a) \cap M_i$ is a ``$\bC$-cone region.''  Let's be intentionally vague for the moment, and just think of this as an annulus where $M_i$ looks very close to a cone, which may be $\bC$ or may be some other $\bC'$ diffeomorphic to $\bC$.

Let us pick $\rho_i$ to be the least such radius among all possible $a \in B_{1/2}$, and pick $a_i$ realizing this point.  Smooth convergence to $\bC$ on compact subsets of $B_1 \setminus \{0\}$ implies $\rho_i \to 0$ and $a_i \to 0$.  Since all the $M_i \cap A_{1/2, \rho_i}(a_i)$, being $\bC$-cone regions, are diffeomorphic to $\bC \cap A_{1, 1/2}$, our contradiction hypothesis implies that none of the $M_i \cap B_{2\rho_i}(a_i)$ are diffeomorphic.

Now consider the rescaled surfaces $M_i' = \rho_i^{-1}(M_i - a_i)$.  After passing to a subsequence, we get convergence $M_i' \to M'$, an area-minimizer in $\R^8$.  By construction, $M' \cap A_{\infty, 1}(a)$ is a $\bC$-cone region, and so $M'$ is smooth outside $B_1$.  Any tangent cone at infinity of $M'$ is a cone $\bC'$ very near to $\bC$, having the same density $\theta_{\bC'}(0) = \theta_\bC(0)$.

On the other hand $M'$ cannot be smooth, because then $M_i' \to M'$ smoothly on compact sets, contradicting the fact that none of the $M_i' \cap B_2$ are diffeomorphic.  So there are (finitely-many) singularities of $M'$ in $B_1$.  I claim that any such singularity $x$ has density $\theta_{M'}(x) < \theta_\bC(0)$.  Otherwise, if we had $\theta_{M'}(x) \geq \theta_\bC(0)$, then monotonicity would imply that $M' = x + \bC'$.  But then, for $i$ large each $M_i \cap A_{1/2, \rho_i/2}(a_i + \rho_i x)$ would be a $\bC$-cone region, contradicting our choice of $\rho_i$ as the least possible such radius.

So in fact every singularity of $M'$ has density $\leq \theta_{k-1}$.  In a small ball $B_{r_x}(x)$ around every singularity $x$, $M'$ looks scale-invariantly close to a cone $\bC_x$ having density $\leq \theta_{k-1}$, and so by induction (if we take $r$ sufficiently small and pass to a subsequence in $i$), all the $M_i' \cap B_{r_x}(x)$ are diffeomorphic.  Since $M_i'$ converge smoothly to $M'$ away from the singular set, we deduce all the $M_i' \cap B_2$ are diffeomorphic.  This is a contradiction.

\vspace{3mm}

OK so what's wrong with that argument.  The catch is in exactly how we can define a  ``cone region,'' and how these are glued to the other parts of $M_i$.  When we do this contradiction argument, without using any additional quantitative estimates, we can only pass information between finitely-many scales, e.g. between scale $\rho$ and $\rho/2$.  You will find that, no matter how you define cone region (by density drop, by graphicality, etc.), you will have to allow the model cone to change with the radius.  This argument will only show that for a given $M = M_i$, for every radius $r \in [\rho_i, 1/2]$, there is a cone $\bC_r$ diffeomorphic to $\bC$ (in the sphere), and having the same density, so that $M \cap A_{r, r/8}(a_i)$ is a small graph over $a_i + \bC_r$.  In certain special cases, such as when $\bC$ is planar with possible multiplicity (\cite{white:compact}, \cite{BuSh}) or strictly-minimizing with multiplicity-one (\cite{me-luca}, \cite{wang1}), one can rule out the $\bC_r$ changing, since there are good barriers which pass information across scales by ``trapping'' $M$ to stay close to the original $\bC$.  For more general $\bC$ this is not (to my knowledge) available.

The changing cones starts to becomes a problem because, although the $M_i \cap A_{1, \rho_i}(a_i)$ are still all diffeomorphic to an annulus $\bC \cap A_{1, 1/2}$, we lose control over how this annulus is glued into the region $M_i \cap B_{\rho_i}(a_i)$.  Actually, for our one single step of induction, this is not strictly speaking an issue, because we could push all the ``rotation'' of the cones out to scale $\approx 1$.  Once you start attempting to induct however, this is a problem, since we need to be able to glue the cone regions at both ends.

Precisely, we require the following.  Let $\phi_{ij} : M_i \cap A_{1/2, \rho_i}(a_i) \to M_j \cap A_{1/2, \rho_j}(a_j)$ be our diffeomorphisms.  By our convergence and our choice of subsequence, there is a fixed cone $\bC'$ diffeomorphic to $\bC$ so that all the $M_i \cap A_{1, 1/2}$ are small graphs over $\bC \cap A_{1, 1/2}$, and all the $M_i' \cap A_{2, 1}$ are small graphs over $\bC' \cap A_{2, 1}$, and therefore we obtain diffeomorphisms $\psi_i : \bC\cap A_{1, 1/2} \to M_i \cap A_{1, 1/2}$, $\psi_i' : \bC' \cap A_{2, 1} \to M_i' \cap A_{2, 1}$, which are all very small perturbations of the identity.  To adequately construct diffeomorphisms from $M_i$ to $M_j$, we require that all the
\[
\psi_j^{-1} \circ \phi_{ij} \circ \psi_i : \bC \cap A_{1, 1/2} \to \bC \cap A_{1, 1/2}
\]
are isotopic, \emph{and} all the
\[
\psi_j'^{-1} \circ \eta_{a_j, \rho_j} \circ \phi_{ij} \circ \eta_{a_i, \rho_i}^{-1} \circ \psi_i' : \bC' \cap A_{2, 1} \to \bC' \cap A_{2, 1}
\]
are isotopic, where we write $\eta_{x, r}(z) = (z - x)/r$.

Let us illustrate this with two (would-be) examples.  In the first, take $\bC$ a hypercone in $\R^4$, and identify $S^3  \setminus \{pt \}$ with $\R^3$.  Think of the link $\bC \cap \del B_1$ as a genus-two surface in $\R^3$, that looks like the number $8$ speared onto the $z$-axis.  Suppose there is a smooth family of smooth, minimal cones $\bC_t$ for $t \in \R$, so that as $t$ increases the cross section $\bC_t \cap \del B_1$ keeps one handle fixed, and rotates counter-clockwise the other handle around the $z$-axis, and does so periodically in the sense that $\bC_{t+i} = \bC_t$ for every integer $i$.

Consider now the surface $M$ defined by $M \cap \del B_r = \bC_{-\sigma \log r} \cap \del B_r$ for some small $\sigma > 0$, so that $M \cap A_{1, 0}$ is a $\bC$-cone region.  The annuli $M_i := M \cap A_{1, e^{-i/\sigma}}$ are all diffeomorphic, and in fact $(e^{i/\sigma} M_i) \cap A_{1, 1/2} = (e^{j/\sigma} M_j) \cap A_{1, 1/2}$ for integers $i$, $j$, but for any two diffeomorphisms $\phi_{ij} : M_i \to M_j$, $\phi_{ij'} : M_i \to M_{j'}$ (with $j \neq j'$) which are isotopic as maps $A_{1, 1/2} \cap M_i \to A_{1, 1/2} \cap M_j \equiv A_{1, 1/2} \cap M_{j'} \equiv A_{1, 1/2} \cap M_1$, the rescaled diffeomorphisms
\[
\eta_{0, e^{-j/\sigma}} \circ \phi_{ij} \circ \eta_{0, e^{-i/\sigma}}^{-1}, \eta_{0, e^{-j'/\sigma}} \circ \phi_{ij'} \circ \eta_{0, -i/\sigma}^{-1} : M_1 \cap A_{1,1/2} \to M_1 \cap A_{1, 1/2}
\]
cannot also be isotopic.

The second example is simpler, but avoidable.  If the $M_i$ had only bounded index, then $\bC$ could be a multiplicity-two plane.  As the radius $r$ decreases in any fixed cone region $M_i \cap A_{1/2, \rho_i}(a_i)$, one could imagine the equatorial spheres $\bC_r \cap \del B_r$ rotating $180^o$, with the effect that the $M_i \cap \del B_{1/2}(a_i)$ is glued to $M_i \cap \del B_{\rho_i}(a_i)$ via the antipodal map.  In this case in fact the antipodal map and the identity map are the only two possibilities, so one could avoid this by passing to a further subsequence.

It turns out that the second example more or less captures the behavior if all the $\bC_r$ are legitimately rotations of $\bC$ (so $\bC$ is ``integrable through rotations'').  In this case, one can arrange for the $\phi_{ij}|_{\del B_r}$ to be small perturbations of a rotation and dilation, and so after passing to a further subsequence we get our required isotopies.  For more general cones $\bC$ this could fail, as illustrated by the first example.  

\vspace{3mm}

Instead, we can show that in fact the cones $\bC_r$ cannot change at all.  In Theorem \ref{thm:scone}, we adapt the Lojasiewicz-Simon decay-growth estimates \cite{simon} to prove that while the density drop is small, $M$ must stay graphical over the same cone, \emph{independent} of whatever happens in smaller regions.  The intuition is that a minimal graph over a cone $\bC$ must either grow or decay polynomially or logarithically (in $r$) from one scale to the next, and most importantly (as proved in \cite{simon}) if $u$ is not already 1-homogenous then there is a lower/upper bound to this growth/decay.  If $u$ starts very small but then becomes too big, the graph of $u$ will start ``eating'' up a definite amount of density at each scale drop.  We highlight that this decay-growth theorem requires no a priori knowledge of the singular set.  We also mention that the estimate holds in any dimension or codimension, and for any surface with (small) $C^2$ mean curvature.

So, a posteriori after having constructed the cone regions by our contradiction argument, we can use Theorem \ref{thm:scone} to show that they are legitimately graphs over a single cone.  This allows our original contradiction argument to go through.  A minor modification of the above argument can also deal with $M_i$ having bounded index.

\vspace{3mm}

So far, that deals with smooth $M_i$.  But of course the $M_i$ could be singular, and in this case it's not really clear what $M_i$ being diffeomorphic means.  To adequately capture both the regular and singular structure, we build \emph{global} diffeomorphisms $\phi_{ij} : B_{1/2} \to B_{1/2}$, which take $M_i$ to $M_j$.  Actually, the $\phi_{ij}$ will only be globally bi-Lipschitz maps, and be smooth diffeomorphisms away from the singular set.  Getting better regularity of the $\phi_{ij}$ at singular points seems to be more or less equivalent to knowing all the singular cones are integrable through rotations.

Section \ref{sec:param} is devoted to constructing these maps.  Multiplicity is a significant technical thorn here.  If everything were single-sheeted, the section would be trivial (since all our maps would be $C^1$ perturbations of the identity), but since we have no control over how much sheets are squished together, gluing the various maps together is significantly more technical (since our maps are only $C^0$ perturbations of $id$).  The key idea is that two $C^1$ diffeomorphisms $\R \to \R$, which coincide with the identity outside $[-1, 1]$, are trivially isotopic just by taking an appropriate convex combination of the two.  By extension, if two $C^1$ diffeomorphisms act by moving along the flowlines of some fixed, smooth vector field (having no closed orbits), and agree outside some small set, then these can be isotoped together along the flowlines.  Even if the vector fields are different, but very close in $C^1$, this will work.  We then build our parameterizations $\phi$ to have this structure.

\section{Preliminaries}\label{sec:prelim}

We will mostly work in $\R^8$.  If $U \subset \R^8$, then generally all norms, lengths, areas, volumes, etc. we use in $U$ are taken with respect to the Euclidean metric $\geucl$ unless otherwise specified.  Write $\overline{U}$ for the set-theoretic closure of $U$.  We write $B_r(a)$ for the (Euclidean) open ball in $\R^8$ centered in $a$, and we define the open annuli $A_{R, r}(a) = B_R(a) \setminus \overline{B_r}(a)$, $A_{R, 0}(a) = B_R(a) \setminus \{a\}$, and $A_{\infty, r}(a) = \R^8 \setminus \overline{B_r(a)}$.

We write $B_r(U) = \{ x : \dist(x, U) < r \}$ for the open $r$-tubular neighborhood of $U$, where $\dist(\cdot, U)$ the Euclidean distance to $U$.  Define $d_H(U,U')$ to be the Hausdorff distance, i.e. the least $r$ for which $U \subset B_r(U')$ and $U' \subset B_r(U)$.  Write $\haus^n$ for the $n$-dimensional Hausdorff measure, and $\omega_n = \haus^n(B_1 \subset \R^n)$ for the volume of the unit $n$-ball.  Let $\eta_{x, r}(y) = (y - x)/r$ be the translation/dilation map.

We write $v \cdot w$ to indicate the usual Euclidean inner product, and $|v| = (v\cdot v)^{1/2}$ for the usual Euclidean norm of a vector.  If $T(v_1, \ldots, v_k)$ is a mult-linear map to $\R^P$, we write $|T| = \sup_{|v_i| \leq 1} |T(v_1, \ldots, v_k)|$ for the operator norm.  If we need to make explicit a metric $g$, we will write $|v|_g = g(v, v)^{1/2}$ for the $g$-length of $v$.  If $\phi : (N, g) \to (N', g')$ is a $C^1$ map between Riemannian manifolds, then for $x \in N$ we define $|D\phi|_x| = \sup \{ |D\phi|_x(v)|_{g'} : v \in T_x N, |v|_{g} \leq 1 \}$.

Given a $C^1$ embedded $n$-submanifold $M$ in $\R^8$, we write $T M$, $T^\perp M$ for the tangent, normal bundles of $M \subset (\R^8, \geucl)$, and for each $x \in M$ we identity $T_x M$, $T_x^\perp M$ with a subspace of $(\R^8, \geucl)$ in the obvious fashion.  We write $u : M \to M^\perp$ for a section of the normal bundle $T^\perp M$, and $\nabla$ for the induced connection derivative on $T^\perp M$.  Given such a $u$, we write $\graph_M(u) := \{ x + u(x) : x \in M \}$.  Define the $C^k$ and $C^{k,\alpha}$ norms of $u$ w.r.t. $\nabla$.  If $M_i$ is a sequence of embedded submanifolds, we say $M_i \to M$ in $C^k(U)$ (with multiplicity $m \in \N$) if there is a set $U' \supset U$, and a sequence of $C^k$ functions $\{ u_{ij} : M \cap U' \to M^\perp \}_{j = 1}^m$, so that for all $i >> 1$ we have
\[
M_i \cap U = \cup_{j=1}^m \graph_M(u_{ij}), \quad |u_{ij}|_{C^k(M \cap U')} \to 0 \text{ as } i \to \infty.
\]
For $x \in M$ and $M$ as above, define the $C^k$ regularity scale $r_k(x, M)$ to be the largest radius $r$, so that for some $n$-plane $P$ we can find a $C^k$ function $v : (x + P) \cap B_r(x) \to P^\perp$ so that
\[
M \cap B_r(x) \subset \graph_{x+P}(v), \quad \sum_{i=0}^k r^{1-i} |D^i v| \leq 1.
\]
We define the global regularity scale $r_k(M) = \inf_{x \in M} r_k(x, M)$.

Let $g$ be a $C^2$ metric on $U$.  Write $B^g_r(x)$ for the ball of radius $r$ centered at $x$ with respect to the metric $g$, $\haus^7_g$ for the Hausdorff measure w.r.t. $g$, and $B^g_r(U) = \{ x : \dist_g(x, U) < r \}$ for the $g$-geodesic open $r$-tubular neighborhood around $U$.  If $M$ is an embedded $C^1$ hypersuface of $U$, we write $T^{\perp_g} M$ for the normal bundle of $M$ w.r.t. $g$, and identity each $T_x^{\perp_g} M$ with a subspace of $\R^8$.  Write $u : M \to M^{\perp_g}$ for a section of normal bundle, and $\nabla^g$ for the induced connection derivative on $T^{\perp_g}M$ induced by the Levi-Civita connection $\nabla^g$ of $g$.

If $|g - \geucl| \leq \delta$, then for every $x \in B_1$ and $r < (1-\delta)(1-|x|)$, we have
\begin{equation}\label{eqn:g-ball}
B^g_{(1-\delta)r}(x) \subset B_r(x) \subset B^g_{(1+\delta)r}(x).
\end{equation}
If $M$ is an embedded $C^1$ hypersurface of $U$, then we can express
\begin{equation}\label{eqn:haus-F}
\haus^7_g(M) = \int_M F(g(x), T_x M) d\haus^7(x)
\end{equation}
for some analytic functional $F(g, V)$ (c.f. \cite[Remark 1]{ScSi}).  We have $|F(g, V) - 1| \leq c |g - \geucl|$, and so we get
\begin{equation}\label{eqn:hausdorff}
(1-c\delta)\haus^7(M) \leq \haus^7_g(M) \leq (1+ c\delta)\haus^7(M).
\end{equation}
for some absolute constant $c$.  In the special case when $M = \{ (x, u(x)) : x \in U' \subset \R^7 \}$ for some Lipschitz function $u$ defined on a subset of $\R^7$, then we can rewrite \eqref{eqn:haus-F} as $F d\haus^7 = F_0(g(x, u(x)), Du|_x) dx$ for $F_0$ an analytic function satisfying $F(\geucl, z) = \sqrt{1+|z|^2}$.

\vspace{3mm}

We say a $C^2$ embedded hypersurface $M$ is minimal in $(U, g)$ if the first variation of \eqref{eqn:haus-F} vanishes at $M$, i.e. for all $X \in C^1_c(U, \R^8)$, we ask
\begin{align}
0 = \delta_g M(X) &:= \left. \frac{d}{dt} \right|_{t = 0} \haus^7_g(\phi_t(M)) \nonumber \\
&= \int_M \mathrm{div}_{M, g}(X) d\haus^7_g \label{eqn:first-var-g}
\end{align}
where $\phi(x, t) = x + t X$.  In the last equality $\mathrm{div}_{M, g}(X)|_x = \sum_i g(e_i, \nabla^g_{e_i} X)$, where $e_i$ is an $g$-ON-basis for $T_x M$ at $x$.  We write $\cM_7(U, g)$ for the space of $C^2$ embedded minimal hypersurface in $U$ such that $\sing M$ is a discrete set.

When $M = \{(x, u(x)) : x \in B_1 \}$ for some $C^2$ function $u : B_1 \subset \R^7 \to \R$, then by considering the structure of $F_0$, and provided $|u|_{C^1(B_1)} \leq 1$ and $|g - \geucl|_{C^1(B_1^7 \times [-2, 2])} \leq \delta$ for some absolute constant $\delta$, the $u$ will satisfy a PDE of the form
\[
a_{ij} D^2_{ij} u + b_i D_i u = f_i D_i g + f_8 D_8 g,
\]
where $a_{ij}, b_i, f_i, f_8$ are analytic functions of $g(x, u(x)), Dg|_{(x, u(x))}, Du|_x$ satisfying $1/2 \leq a_{ij} \leq 2$.  In particular, by standard elliptic theory if $g \in C^k$ then $u \in C^{k,\alpha}$ for every $\alpha \in (0, 1)$, and for every $\theta \in (0, 1)$ we have the estimate
\begin{equation}\label{eqn:higher-reg}
|u|_{C^{k,\alpha}(B_\theta)} \leq c( ||u||_{L^2(B_1)} + |Dg|_{C^{k-1}(B_1)}) , 
\end{equation}
where $c = c(\theta, \alpha, k, |u|_{C^{1,\alpha}(B_1)}, |Dg|_{C^{k-1}(B_1\times [-2, -2])})$.  Typically we will apply \eqref{eqn:higher-reg} in normal coordinates centered at some point on the graph of $u$, so the restrictions on $g$ and $u$ are never an issue.

\vspace{3mm}

For $M$ minimal in $(U, g)$, we say $M$ is stable in $(U, g)$ if the second variation of \eqref{eqn:haus-F} is non-negative at $M$.  In other words, we ask
\[
\delta^2_g M(X, X) := \left. \frac{d^2}{dt^t} \right|_{t = 0} \haus^7_g(\phi_t(M)) \geq 0 , \quad \forall X \in C^1_c(U, \R^8),
\]
where $\phi(x, t) = x + t X$.  It is well known that if $u = \pi^{\perp_g}_M(X) : M \to M^{\perp_g}$, then $\delta^2_g M(X, X) = Q_{M, g}(u, u)$ for $Q$ being the quadratic form in \eqref{eqn:Q}.  Here $\pi_M^{\perp_g}$ denotes the $g$-orthogonal projection onto $T^{\perp_g} M$.  Let us define $\mindex(M, U, g)$ to be the maximal dimension of subspace $V \subset C^1_c(M \cap U, M^{\perp_g})$ on which $Q_{M, g}$ is strictly negative definite, i.e. so that $Q_{M,g}(\phi, \phi) < 0$ for all $\phi \in V$.  Observe the trivial inclusion $\mindex(M, U', g) \leq \mindex(M, U, g)$ if $U' \subset U$.

We note that by \cite[Section 18]{neshan}, if $\mindex(M, U, g) = 0$ and $\haus^6(U \cap \overline{M} \setminus M) = 0$, then $U \cap \overline{M} \setminus M$ consists of isolated points, and hence by an elementary cutoff function argument $M$ is stable in $(U, g)$.  This motivates our definition of $\cM_7$.

If $M$ is closed (as sets) in $U$, then $Q_{M,g}$ can be written
\[
Q_{M,g}(u, u) = -\int_{M} g(u, \Delta_g u + |A_{M, g}|^2 u + Ric_g(u, \cdot)) d\haus^7_g =: - \int_M g(u, L_{M,g} u ) d\haus^7_g,
\]
where $\Delta_{g}$ is the connection Laplacian on $T^{\perp_g} M$.  $L_{M, g}$ is a self-adjoint elliptic operator on $C^2_c(M, M^{\perp_g})$.  In this case, for every $U' \subset\subset U$ with $C^1$ boundary intersecting $M$ transversely, $\mindex(M, U', g)$ is finite and consists of the number of negative eigenvalues for $L_{M, g}$ under Dirichlet zero boundary conditions on $M \cap U'$.

As discussed in the Introduction, singular minimal surfaces with finite index, in the sense that we've described above, arise naturally in various min-max constructions (see e.g. \cite{li1}, \cite{hiesmayr}, \cite{gaspar}).  We note that if $M$ is a minimal cone in $(\R^8, \geucl)$ having $\haus^6(\overline{M} \setminus M) = 0$, it's easy to check that either $M$ is stable, or $\mindex(M, \R^8, \geucl) = \infty$.  So any kind of finite index must be ``macroscopic.''  This can be made more precise below.

For any $x \in U$, we define the ``stability radius'' $r_s(x) \equiv r_s(x, M, U, g)$ to be the largest radius $r_s$ with the property that $B_{r_s}(x) \subset U$, and $\mindex(M, B_{r_s}(x), g) = 0$.  Essentially the same quantity, as well as the below Lemma \ref{lem:rs-bound}, appeared in \cite{song}.  Since it is very short we include a proof here for the readers' convenience.  We remark that \cite[Corollary 3.8]{wang1} has shown a converse to Lemma \ref{lem:rs-bound}, i.e. if $M \in \cM_7(U, g)$ and $r_s > 0$ everywhere in $U$ then $M$ has bounded index on compact subsets.
\begin{lemma}\label{lem:rs-bound}
Let $M$ be a $C^2$ embedded minimal hypersurface in $(U, g)$, for $U \subset \R^8$ open and $g$ a $C^2$ metric on $U$.  Suppose $\mindex(M, U, g) < \infty$ and $\haus^6(U \cap \overline{M} \setminus M) = 0$.  Then $r_s > 0$ on $U$, and (hence) $M \in \cM_7(U, g)$.
\end{lemma}

\begin{proof}
If $x \in M$, then this follows trivially by the Poincare inequality in $B_r(x) \cap M$, for $r$ small.  Suppose $x \in \overline{M} \setminus M$ but $r_s(x) = 0$.  Then by our hypothesis and \cite{neshan} (as in our discussion above), we must have $\mindex(M, B_r(x), g) > 0$ for every $r > 0$ small.  We can therefore find a $C^1_c$ functions $u_r : M \cap B_r(x) \to M^{\perp_g}$ with $Q_{M, g}(u, u) < 0$ and $\spt u_r \cap \overline{M} \setminus M = \emptyset$.  But then for any such $r$ we can find an $r' > 0$ with $\spt u_r \cap B_{r'}(x) = \emptyset$, and thereby construct an infinite-dimensional subspace $V \subset C^1_c(M, M^{\perp_g})$ on which $Q_{M, g}$ is strictly negative definite.  This is a contradiction, so we must have $r_s(x) > 0$.
\end{proof}

\vspace{3mm}

It will be convenient to use the notion of varifold.  Recall that if $U \subset \R^8$, then an integral $7$-varifold $V$ in $(U, g)$ is a Radon measure on the Grassmanian bundle $G_7(TU) \equiv U \times \mathrm{Gr}(7, 8)$ that takes the form
\begin{equation}\label{eqn:int-varifold}
V(\phi(x, S)) = \int_{M_V} \phi(x, T_x M_V) \theta_V(x) d\haus^7_g(x)
\end{equation}
for some countably $7$-rectifiable set $M_V$, and some positive, integer-valued $\haus^7_g$-measurable function $\theta_V$.  The mass measure $\mu_V = \haus^7_g \llcorner \theta_V \llcorner M_V$ is the pushforward of $V$ under the projection map $G_7(TU) \to U$.  We write $\spt V = \spt \mu_V$, and define $\reg V$ to be the set of points $x \in \spt V$ having the property that for some $r > 0$, $\spt V \cap B_r(x)$ is a closed (as sets) embedded hypersurface of $B_r(x)$.  We define the closed set $\sing V = \spt V \setminus \reg V$.  If $V_i$, $V$ are $7$-varifolds in $\R^8$ and $V_i \to V$ as varifolds, we say $V_i \to V$ in $C^k(U)$ if for some $U' \supset U$ and for $i >> 1$, the $\spt V_i \cap U'$ and $\spt V \cap U'$ are $C^k$ embedded hypersurfaces, and $\spt V_i \to \spt V$ in $C^k(U)$.

If $M$ is a $C^1$ hypersurface of $U$, then $M$ induces a natural varifold $[M]_g$ by taking in \eqref{eqn:int-varifold} $M_V = M$ and $\theta_V = 1$.  Write $[M] \equiv [M]_{\geucl}$.  Given any $C^1$, proper map $f : (U, g) \to (U', g')$, we define the pushforward $f_\sharp V$ to be the integral varifold $f_\sharp V$ defined by
\[
(f_\sharp V)(\phi(x, S)) = \int \phi(f(x), Df|_x S) J_{S} f(x) dV(x, S),
\]
where, given any $g|_x$-ON basis $e_i$ for $S$, $J_{S}(x) = \det( g'|_{f(x)} (Df|_x(e_i), Df|_x(e_j) )^{1/2}$.  Equivalently, by the area formula, in the notation of \eqref{eqn:int-varifold} $M_{f_\sharp V} = f(M_V)$ and $\theta_{f_\sharp V}(x) = \sum_{y \in \theta^{-1}(x)} \theta_V(y)$.  Similarly, $f_\sharp [M]_g = [f(M)]_{g'}$.

$V$ is stationary in $(U, g)$ if the first variation $\delta_g V$ vanishes, i.e. if for all $U' \subset\subset U$ and $X \in C^1_c(U', \R^8)$, we have
\begin{equation}\label{eqn:first-var-V}
0 = \delta_g V(X) := \left. \frac{d}{dt} \right|_{t = 0} \mu_{(Id + tX)_\sharp V}(U') = \int_{U'} \mathrm{div}_{S, g}(X) dV(x, S),
\end{equation}
where $\mathrm{div}_{S, g}(X) = \sum_i g(e_i, \nabla^g_{e_i} X)$ for any $g|_x$-ON basis of $S$.  If $M$ is a minimal surface in $(U, g)$, then $[M]_g$ is stationary in $(U, g)$.

Let $V$ be a stationary integral varifold in $(B_1, g)$, and suppose $g$ is a $C^2$ metric satisfying $|g - \geucl| + |Dg| \leq \Gamma |x|$.  Then by \eqref{eqn:int-varifold}, \eqref{eqn:first-var-V} we have the inequality
\begin{align*}
|\delta_{\geucl}V(X)| &= \left| \int \mathrm{div}_{T_x M_V, \geucl}(X) d\mu_V \right| \\\
&\leq c \Gamma \int |x| |DX| + |X| d\mu_V(x) \quad \forall X \in C^1_c(B_1, \R^8),
\end{align*}
for $c$ an absolute constant.  Therefore by plugging in $X = \phi(|x|/\rho) x$ for $\phi$ being (a suitable smooth approximation for) $1_{(-\infty, 1]}$, we obtain the monotonicity
\begin{equation}\label{eqn:sharp-mono}
\int_{A_{r, s}} \frac{|\pi^\perp_V(x)|^2}{|x|^{7+2}} d\mu_V(x) \leq (1+ c \Gamma r)^{7+1} \theta_V(0, r) - (1+c\Gamma s)^{7+1} \theta_V(0, s).
\end{equation}
for all $0 < s \leq r \leq 1$, for $c$ an absolute constant, and for $\theta_V(x, r) = \omega_7^{-1} r^{-7} \mu_V(B_r(x))$ the density ratio of $V$.  Here $\pi^\perp_V(x)$ is the \emph{Euclidean} projection of the position vector $x$ onto the \emph{Euclidean} normal direction $T^\perp_x M_V$.  Of course if $g = \geucl$, then \eqref{eqn:sharp-mono} reduces to the usual sharp monotonicity formula.  We note also that, since $X$ is smooth near $0$ and $X(0) = 0$, \eqref{eqn:sharp-mono} continues to hold if $V$ is only stationary in $(B_1 \setminus \{0\}, g)$, provided $\mu_V(B_1 \setminus \{0\}) < \infty$, in which case $V$ extends to a stationary integral varifold in $(B_1, g)$.

Suppose now only $|g - \geucl|_{C^2(B_1)} \leq \delta$.  For each $x \in B_1$ and $r < (1-\delta)(1-|x|)$ we can find a $C^2$ coordinte transformation $z = \phi_x(y)$ in $B_r(x)$, so that the $(z^i)$ form a choice of normal coordinates of $g$ at $x$.  Applying the previous paragraph in this coordinate chart, and using \eqref{eqn:g-ball}, \eqref{eqn:hausdorff}, we obtain the monotonicity
\begin{equation}\label{eqn:monotonicity}
\theta_V(x, s) \leq (1+ c_0 \delta r) \theta_V(x, r)\quad \forall x \in B_1, 0 < s < r < (1-\delta) r,
\end{equation}
where $c_0$ is an absolute constant.  \eqref{eqn:monotonicity} implies that $\theta_V(x) := \lim_{r \to 0} \theta_V(x, r)$ exists, is upper-semi-continuous, and (since $V$ is integral) $\geq 1$ on $\spt V$.  Moreover, given any $\sigma < 1$ and $x \in \spt V \cap B_{\sigma}$, we have the Ahlfors regularity 
\begin{equation}\label{eqn:ahlfors}
\frac{1}{2} \leq \theta_V(x, r) \leq \frac{1+c_0}{(1-\sigma)^n} \mu_V(B_1), \quad 0 < r < \min\{ 1-|x|, \frac{1}{c_0 \delta} \}.
\end{equation}

\vspace{3mm}

We define $\cISV_7(U, g)$ to be the class of stationary integral $7$-varifolds $V$ in $(U, g)$ with the property that $\sing V$ is a discrete set of points.  Every such $V$ satisfies $\reg V \in \cM_7(U, g)$, and more generally has the following structure.
\begin{lemma}\label{lem:isv-decomp}
A varifold $V \in \cISV_7(U, g)$ if and only if for every $U' \subset \subset U$ there is a $k \in \N$ so that we can write
\begin{equation}\label{eqn:isv-decomp}
V \llcorner U' = \sum_{i=1}^k m_i [M_i]_g
\end{equation}
for $m_i \in \N$, $M_i \in \cM_7(U', g)$, and with $\overline{M_i}$ being disjoint.
\end{lemma}

\begin{proof}
Let us write $\reg V \cap U' = \cup_i M_i$ for each $M_i$ being a connected, disjoint, $C^2$, embedded hypersurface of $U'$.  Since $\cup_i (\overline{M_i} \setminus M_i) \subset \overline{\reg V} \setminus \reg V = \sing V$, each $\overline{M_i} \setminus M_i$ consists of a discrete set of points.  By the constancy theorem (\cite[Chapter 8, Theorem 4.1]{Sim}) $\theta_V$ is locally constant on each $M_i$, and hence $\theta_V|_{M_i} = m_i \in \N$.  By stationarity we get that each $M_i$ is minimal in $(U', g)$, and hence $M_i \in \cM_7(U', g)$.  By applying \eqref{eqn:ahlfors} to each stationary varifold $[M_i]_g$, we get that the collection $\{M_i\}_i$ is finite, and by the maximum principle of \cite{ilmanen}, the $\{\overline{M_i}\}_i$ are disjoint.  This proves the ``only if'' part of the Lemma; the ``if'' part is trivial.
\end{proof}

\begin{remark}[A note on scaling]\label{rem:scaling}
Let $B_R(a) \subset \R^8$ and $g$ be a $C^2$ metric on $B_R(a)$.  If $V \in \cISV_7(B_R(a), g)$, and we define $g' = g \circ \eta_{a, R}^{-1}$ and $V' = (\eta_{a, R})_\sharp V$ where we view $\eta_{a, R}$ as a map $(B_R(a), g) \to (B_1(0), g')$ (i.e. so \emph{not} an isometry), then $V' \in \cISV_7(B_1(0), g')$, and $\mu_{V'}(U) = R^{-n} \mu_V(\eta_{a, R}^{-1}(U))$ for all $U \subset B_1$.  We shall always interpret the rescaling $(\eta_{a, R})_\sharp V$ in this fashion.  Of course if $M \in \cM_7(B_R(a), g)$, then $(\eta_{a, R})_\sharp [M]_g = [\eta_{a, R}(M)]_{g'}$.
\end{remark}


\vspace{3mm}

There is good compactness theory for stationary integral varifolds, but for our purposes we will make use of the following compactness theory for $\cISV_7$.  Given $V \in \cISV_7(U, g)$ we define $\mindex(V, U, g) := \mindex(\reg V, U, g)$.

\begin{lemma}\label{lem:index-compact}
Let $g_i, g$ be $C^2$ metrics on $U$, such that $g_i \to g$ in $C^2(U)$.  Let $V_i \in \cISV_7(U, g_i)$ be a sequence satisfying
\begin{equation}\label{eqn:index-compact-hyp}
\sup_i \mu_{V_i}(U) < \infty, \quad \mindex(V_i, U, g) \leq I < \infty
\end{equation}
for some $I \in \N$.

Then after passing to a subsequence, we can find a $V \in \cISV_7(U, g)$, and at most $I$ points $\cI \subset \spt V \cap U$, so that: $V_i \to V$ as varifolds in $U$; $V_i \to V$ in $C^2$ on compact subsets of $U \setminus (\sing V \cup \cI)$; $\mindex(V, U, g) \leq I$; and $\mindex(V_i, B_r(y), g_i) \geq 1$ for every $r > 0$, $y \in \cI$, $i >> 1$.  Additionally, for any collection of disjoint balls $\{B_{r_\alpha}(x_\alpha)\}_\alpha \subset U$, we have
\begin{gather}\label{eqn:index-compact-concl1}
\sum_\alpha \mindex(V_i, B_{r_\alpha}(x_\alpha), g_i) \leq I \quad \forall i.
\end{gather}
\end{lemma}

\begin{proof}
Let $\cI$ consist of the points $y \in U$ with the property that $\limsup_i r_s(y, V_i) = 0$.  Trivially we have $\mindex(V_i, B_r(y), g_i) \geq 1$ for every $r > 0$, $y \in \cI$, $i >> 1$.  We first show $\#\cI \leq I$.  Suppose, on the contrary, we can find $\{y_1, \ldots, y_{I+1}\} \subset \cI$.  Then for $i >> 1$, we can find radii $r_i$ so that $\{B_{r_i}(y_j)\}_j \subset B_1$ are disjoint, and $C^1_c$ functions $u_{ij} : \reg V_i \cap B_{r_i}(y_j) \to \reg V_i^{\perp_{g_i}}$ so that $Q_{\reg V_i, g_i}(u_{ij}, u_{ij}) < 0$.  These functions are trivially linearly-independent, and form a strictly-negative-definite subspace for $Q_{\reg V_i, g_i}$, and hence $\mindex(V_i, U, g_i) \geq I+1$.  This is a contradiction.  An essentially verbatim argument shows \eqref{eqn:index-compact-concl1}.

Take any $x \in B_1 \setminus \cI$.  There is a radius $r = r_x > 0$ so that, after passing to a subsequence, $B_{2r}(x) \subset U$ and every $V_i$ is stable in $(B_r(x), g_i)$.  Each $\reg V_i \cap B_r(x)$ is orientable (being a set-theoretically-closed embedded hypersurface in the simply-connected region $B_r(x) \setminus \sing V_i$), and so by \cite[Theorems 1 and 2]{ScSi}, \eqref{eqn:haus-F}, and $C^0$ convergence $g_i \to g$, after passing to a subsequence we can find a $V \in \cISV_7(B_r(x), g)$ so that $[\spt V_i \cap B_r(x)]_{g_i} \to V$ as varifolds in $B_r(x)$, and $\spt V_i \to \spt V$ in $C^{1,\alpha}$ on compact subsets of $B_r(x) \setminus \sing V$ for every $\alpha \in (0, 1)$.  Passing to a further subsequence, \eqref{eqn:higher-reg} and our $C^2$ convergence $g_i \to g$ imply that $\spt V_i \to \spt V$ in $C^2$ on compact subsets of $B_r(x) \setminus \sing V$.  By \eqref{eqn:isv-decomp} and \eqref{eqn:ahlfors}, we can pass to yet a further subsequence, and modify the multiplicities of $V$, so that $V_i \to V$ as varifolds in $B_r(x)$.

The previous paragraph implies that there is an $V_x \in \cISV_7(B_r(x), g)$ so that $V_i \to V_x$ as varifolds in $B_r(x)$, and $V_i \to V_x$ in $C^2$ on compact subsets of $B_r(x) \setminus \sing V_x$.  Now by considering a countable subcover of $\{B_{r_x}(x) \}_{x \in U \setminus \cI}$, and applying a diagonalization argument, we can obtain an $V \in \cISV_7(U \setminus \cI, g)$ so that $V_i \to V$ as varifolds in $U \setminus \cI$, and in $C^2$ on compact subsets of $U \setminus (\sing V \cup \cI)$.  Lower-semi-continuity of mass implies $\mu_V(U \setminus \cI) < \infty$, so by \eqref{eqn:sharp-mono} $V$ extends to a stationary integral varifold in $(U, g)$, and $V_i \to V$ as varifolds in $U$.

We show $\mindex(V, U, g) \leq I$, which will by Lemma \ref{lem:rs-bound} imply $V \in \cISV_7(U, g)$.  Suppose, otherwise: there is a linear subspace $V \subset C^1_c(\reg V, \reg V^{\perp_g})$ of dimension $I+1$, so that $Q_{\reg V,g}(u, u) < 0$ for all $u \in V$.  Let $u_1, \ldots, u_{I+1}$ be a basis for $V$.  Then there is a $q < 0$ so that
\begin{equation}\label{eqn:index-1}
Q_{\reg V, g}(\sum_j \lambda_j u_j, \sum_j \lambda_j u_j) \leq q < 0 
\end{equation}
for all collections $\lambda_j$ satisfying $\sum_j \lambda_j^2 = 1$.

Take $\eta : \R \to \R$ be a fixed, smooth, increasing function satisfying
\[
\eta|_{(-\infty, 1/4]} \equiv 0, \quad \eta|_{[3/4, \infty)} \equiv 1, \quad |\eta'| \leq 10,
\]
and then for $\sigma > 0$ define $\phi_\sigma(x) = \eta(\sigma^{-1} \dist(x, \cI))$, and $u_{j,\sigma} = \phi_\sigma u_j$.  Given any $U' \subset\subset U \setminus \cI$, then $u_{i,\sigma} = u_i$ on $U'$ for $\sigma$ sufficiently small.  Since $\cI$ is discrete (codimension $7$), \eqref{eqn:ahlfors} implies $Q_{\reg V, g}(u_{j,\sigma}, u_{k,\sigma}) \to Q_{\reg V,g}(u_j, u_k)$ as $\sigma \to 0$, for any $j,k$.  We deduce (after taking $\sigma$ sufficiently small, and replacing $q$ in \eqref{eqn:index-1} with $q/2$), that it suffices to assume that the $u_i$ are supported away from $\cI$.

Now $C^2$ convergence $V_i \to V$ implies that we can find open sets $\spt u_j \subset W_j \subset U \setminus (\cI \cup \sing V)$, and $C^2$ functions $w_{ij} : W_j \cap \reg V \to \reg V^{\perp_g}$ such that
\[
\reg V_i \cap W_j \supset \graph_{\reg V}(w_{ij}) \cap W_j =: G_{ij}
\]
and $|w_{ij}|_{C^2} \to 0$ as $i \to \infty$.  Define $u_{ij} \in C^1_c(\reg V_i, V_i^{\perp_{g_i}})$ by setting $u_{ij}(y = x + w_{ij}(x)) = \pi_{G_{ij}}^{\perp_{g_i}}(u_j(x))$ for $y \in G_{ij}$, and $u_{ij} = 0$ else.  Here $\pi^{\perp_{g_i}}_{G_{ij}}|_y$ denotes the $g_i$-orthogonal projection to the normal space $T_y^{\perp_{g_i}} G_{ij}$.

Since $\mindex(V_i, U, g) \leq I$, and the $u_{ij}$ are linearly independent for $i >> 1$, we can find numbers $\lambda_{ij}$ so that $Q_{\reg V_i, g_i}(\sum_j \lambda_{ij} u_{ij}, \sum_j \lambda_{ij} u_j) \geq 0$, and (WLOG) $\sum_j \lambda_{ij}^2 = 1$.  Passing to a subsequence, we get $\lambda_{ij} \to \lambda_j$ for each $j$, and hence by $C^2$ convergence $V_i \to V$, $g_i \to g$, we get
\begin{equation*}
Q_{\reg V_i, g_i}(\sum_j \lambda_{ij} u_{ij}, \sum_j \lambda_{ij} u_{ij}) \to Q_{\reg V,g}( \sum_j \lambda_j u_j, \sum_j \lambda_j u_j) \geq 0.
\end{equation*}
However this contradictions \eqref{eqn:index-1}.  We must therefore have $\mindex(V, U, g) \leq I$.
\end{proof}

\section{Stable cones}

In this section we work in $\R^8$.  We let $\cC$ be the collection of stable minimal hypercones $\bC^7$ in $\R^8$ which are smooth, closed (as sets), embedded hypersurfaces away from $0$.  By \cite{ScSi}, \cite{neshan}, if $\bC \subset \R^8$ is a smooth, stable, embedded hypersurface and $\haus^6(\overline{\bC} \setminus \bC) = 0$, then $\bC \in \cC$.  Note that each link $\bC \cap \del B_1$ is necessarily orientable (as a closed, embedded, codimension-one submanifold of a simply-connected manifold) and connected (by the maximum principle, or more generally by Frankel's theorem).

Given $\Lambda > 0$, let $\cC_\Lambda \subset \cC$ denote the collection of $\bC \in \cC$ with $\theta_\bC(0) \leq \Lambda$.  Given $\bC \in \cC$, define $\cC(\bC)$ to be the collection of $\bC' \in \cC$ with property that $\theta_{\bC'}(0) = \theta_\bC(0)$, and there is a $C^2$ diffeomorphism $\phi : \del B_1 \to \del B_1$ such that $\phi(\bC \cap \del B_1) = \bC' \cap \del B_1$.  Of course the condition $\bC' \in \cC(\bC)$ is both symmetric and transitive.

Given $\bC \in \cC$ and function $u : \bC \cap A_{r, s}(0) \to \bC^\perp$, let us define the conical graph 
\[
G_\bC(u) \cap A_{r, s}(0) = \left\{ \frac{x + u(x)}{|x + u(x)|} |x| : x \in \bC \cap A_{r, s}(0) \right\}.
\]
If $a \in \R^8$ and $u$ is defined instead on $(a + \bC) \cap A_{r, s}(a)$, we define
\[
G_{a + \bC}(u) \cap A_{r, s}(a) = a + G_\bC(u( a + \cdot)) \cap A_{r, s}(0).
\]
When $|u|_{C^1}$ is small this notion of graph is effectively equivalent to the ``usual'' graph $\graph_{a + \bC}(u)$, but is more convenient to work with in the conical setting.  Similarly, if $\Sigma$ is a $C^2$ closed embedded hypersurface in $\del B_1$, and $v : \Sigma \to \Sigma^\perp$, then let us write
\[
G_\Sigma(v) = \left\{ \frac{\theta + v(\theta)}{\sqrt{1+|v|^2}} : \theta \in \Sigma \right\}
\]
for the spherical graph

The key facts we will use about $\cC$ are listed below.
\begin{theorem}\label{thm:cones}
The following are true.
\begin{enumerate}
\item \label{item:cones-1} Given any sequence $\bC_i \in \cC$ with $\sup_i \theta_{\bC_i}(0) < \infty$, there is a subsequence $i'$ and cone $\bC \in \cC$ so that $\bC_{i'} \cap \del B_1$ converges to $\bC \cap \del B_1$ smoothly with multiplicity-one, and moreover so that $\theta_{\bC_{i'}}(0) = \theta_\bC(0)$ for all $i'$.

\item \label{item:cones-2} The set of densities $\{ \theta_\bC(0) : \bC \in \cC \}$ forms a discrete set $1 = \theta_0 < \theta_1 < \theta_2 < \ldots $.

\item \label{item:cones-3} Given $\Lambda > 0$, there is an $\eps_0(\Lambda) > 0$ so that if $\bC, \bC' \in \cC_\Lambda$ satisfy
\[
d_H(\bC \cap \del B_1, \bC' \cap \del B_1) \leq \eps \leq \eps_0,
\]
then $\bC' \in \cC(\bC)$, and in fact we can write
\[
\bC' = G_\bC(v), \quad |x|^{1-k} |\nabla^k v| \leq c(\Lambda, k) \eps 
\]
for a smooth, $1$-homogenous function $v : \bC \setminus \{0\} \to \bC^\perp$.

\item \label{item:cones-4} Additionally, for $\bC \in \cC_\Lambda$ we have $r_3(\bC \cap \del B_1) \geq \eps_0$.
\end{enumerate}
\end{theorem}

\begin{proof}
The main ingredient is the Lojasiewicz-Simon inequality of \cite{simon}.  Let $\Sigma$ be any smooth, closed, embedded minimal submanifold of $S^7$.  Then \cite[Theorem 3]{simon} implies there is a $\sigma(\Sigma, \mu) > 0$ and $\theta(\Sigma) \in (0,1/2)$ so that given any $C^{2,\mu}$ function $u : \Sigma \to \Sigma^\perp$ (for $\mu \in (0, 1)$) with $|u|_{C^{2,\mu}(\Sigma)} \leq \sigma$ we have the inequality
\[
| \haus^6(G_\Sigma(u)) - \haus^6(G_\Sigma(0))|^{1-\theta} \leq ||\cM(u)||_{L^2(\Sigma)},
\]
where $\cM(u)$ is the negative $L^2(\Sigma)$-gradient of $u \mapsto \haus^6(G_\Sigma(u))$ (see Appendix \ref{sec:decay}).  In particular, $G_\Sigma(u)$ is minimal in $S^7$ precisely when $\cM(u) = 0$, in which case we have $\haus^6(G_\Sigma(u)) = \haus^6(\Sigma)$.  Of course minimality of $\Sigma$ in $S^7$ is equivalent to minimality of the cone over $\Sigma$ in $\R^8$.

The secondary ingredient is smooth, multiplicity-one compactness for $\cC_\Lambda$.  Let us first note that every smooth, closed, minimal $6$-submanifold in $S^7$ is necessarily connected and orientable.  Given a sequence $\bC_i$ as in the Theorem, then by the usual varifold compactness, sheeting, and dimension reducing argument (e.g. \cite[Theorem 2]{ScSi} or Lemma \ref{lem:index-compact}) we can find a $\bC \in \cC$ and $m \in \N$ so that $\bC_i \to \bC$ smoothly with multiplicity $m$ away from $0$.  However since each $\bC_i \cap \del B_1$, $\bC \cap \del B_1$ is connected while $\del B_1 = S^7$ is simply-connected, we must have $m = 1$.  Theorem \ref{thm:cones} follows in an obvious way from these two ingredients.
\end{proof}

We require some auxilary Lemmas.  Our first shows how the sequence compactness of Theorem \ref{thm:cones} implies a topological compactness as well.
\begin{lemma}\label{lem:cover}
Take $\Lambda > 0$, and a mapping $\bC \in \cC_\Lambda \mapsto \delta_\bC \in (0, \infty)$ assigning to every cone in $\cC_\Lambda$ a positive ``radius.''  Then we can find a finite collection $\bC_1, \ldots, \bC_N \subset \cC_\Lambda$, so that for every $\bC \in \cC_\Lambda$ we have
\[
d_H( \bC \cap \del B_1, \bC_i \cap \del B_1) < \delta_{\bC_i}
\]
for some $i \in \{ 1, \ldots, N\}$.
\end{lemma}

\begin{proof}
Let $\cV$ be the space of stationary integral $6$-varifolds in $S^7$ (stationary with respect to the spherical metric), having total mass $\leq 2\Lambda$, and let this space have the topology induced by varifold convergence.  The usual compactness theorem for stationary varifolds \cite[Chapter 8, Theorem 5.8]{Sim} implies $\cV$ is sequentially compact, and if we include $\cC_\Lambda$ into $\cV$ by identifying $\bC$ with the varifold $[\bC \cap \del B_1]$, then by Theorem \ref{thm:cones} $\cC_\Lambda$ is closed in $\cV$.

Define the metric $d$ on $\cV$ as follows: fix a countably-dense subcollection $\{\phi_i \}_i \subset \{ \phi \in C^0(B_2) : |\phi|_{C^0(B_2)} \leq 1 \}$, and then given $V, W \in \cV$ set
\begin{equation}
d(V, W) = \sum_i 2^{-i} \left| \int \phi d\mu_V - \int \phi d\mu_W \right|.
\end{equation}
Since any $V \in \cV$ is uniquely determined by its mass measure $\mu_V$, since $\cV$ is sequentially compact, and since every $V \in \cV$ has $\mu_V(S^7) \leq 2\Lambda$, it's straightfoward to verify that the varifold topology on $\cV$ coincides with the topology induced by $d$.  Moreover, the monotonicity formula and sequential compactness imply: given any $\eps > 0$, there is a $\delta > 0$ so that for any $V, W \in \cV$:
\begin{equation}\label{eqn:cover1}
d(V, W) < \delta \quad \implies \quad d_H(\spt V, \spt W) < \eps .
\end{equation}
A converse relation holds if $V, W \in \cC_\Lambda \subset \cV$.

By \eqref{eqn:cover1} we can find an $\eps_\bC > 0$ for each $\bC$ so that
\begin{align}\label{eqn:cover2}
U_\bC &:= \{ \bC' \in \cC_\Lambda : d([\bC \cap \del B_1], [\bC' \cap \del B_1]) < \eps_\bC \} \\
&\subset \{ \bC' \in \cC_\Lambda : d_H(\bC \cap \del B_1, \bC' \cap \del B_1) < \delta_\bC \} 
\end{align}
(of course $\spt [\bC \cap \del B_1] = \bC \cap \del B_1$ since $\bC$ is smooth away from $0$).  So the $U_\bC$ form an open cover of $\cC_\Lambda \subset \cV$.  Since compactness and sequential compactness are equivalent for a metric space, and $\bC_\Lambda$ is closed in $\cV$, we can find a finite subcollection $U_{\bC_1}, \ldots, U_{\bC_N}$ which covers $\bC_\Lambda$.  The Lemma then follows by \eqref{eqn:cover1}.
\end{proof}

The second is a quantification of the fact that none of the cones in $\bC$ have any translational symmetry.
\begin{lemma}\label{lem:nearby-cones}
Given $\eps > 0$, $\Lambda > 0$ there is a $\beta(\Lambda, \eps)$ so that if $\bC, \bC' \in \cC_\Lambda$, and $a \in \R^8$ satisfy
\[
d_H( (a + \bC') \cap A_{1, 1/2}, \bC \cap A_{1, 1/2}) \leq \beta,
\]
then $|a| \leq \eps$, $d_H(\bC' \cap B_1, \bC \cap B_1) \leq \eps$, and $\theta_{\bC}(0) = \theta_{\bC'}(0)$.  If $\bC$ is planar, then we require $a \in {\bC'}^\perp$.
\end{lemma}

\begin{proof}
Suppose, towards a a contradiction, the lemma failed: there was a sequence $\bC_i, \bC_i' \in \cC_l$, $a_i \in \R^8$, so that
\begin{equation}\label{eqn:nearby-1}
d_H( (a_i + \bC_i') \cap A_{1, 1/2}, \bC_i \cap A_{1, 1/2}) \to 0,
\end{equation}
but either $|a_i| > \eps$, $d_H(\bC_i \cap B_1, \bC_i' \cap B_1) > \eps$, or $\theta_{\bC_i}(0) \neq \theta_{\bC_i'}(0)$

Passing to a subsequence, as in Theorem \ref{thm:cones} we get smooth, multiplicity-one convergence $\bC_i \to \bC$, $\bC_i' \to \bC'$ for $\bC, \bC' \in \cC_\Lambda$.  By Theorem \ref{thm:cones}, $\theta_{\bC_i}(0) = \theta_\bC(0)$ and $\theta_{\bC_i'}(0) = \theta_{\bC'}(0)$ for $i >> 1$.  In particular, the $\bC_i$ are planar iff $\bC$ is planar.

Assume first $\bC$ is non-planar.  If we had $\limsup_i |a_i| = \infty$, then after passing to a subsequence Theorem \ref{thm:cones} implies $a_i + \bC_i'$ converges in $C^2$ to a plane on compact subsets of $\R^8$, which implies by \eqref{eqn:nearby-1} that $\bC_i$ converges to a plane also, which is a contradiction.  So we can assume $a_i \to a$ also.  From \eqref{eqn:nearby-1} we get $a + \bC' = \bC$, which implies $a = 0$ since $\theta_\bC(x) < \theta_\bC(0)$ for $x \neq 0$.  We therefore have $a = 0$, $\bC = \bC'$, which is a contradiction.

Assume $\bC$ is a plane.  We have
\begin{align}
|a_i| &= d(0, a + \bC')  \\
&\leq d_H ((a_i + \bC_i') \cap B_1, \bC \cap B_1) \\
&\leq c d_H( (a_i + \bC_i') \cap A_{1, 1/2}, \bC \cap A_{1, 1/2}) \to 0 ,
\end{align}
so $a_i \to 0$.  As before we deduce $\bC' = \bC$, which is a contradiction.  This proves Lemma \ref{lem:nearby-cones}.
\end{proof}

The third says that we can regraph over nearby cones.
\begin{lemma}\label{lem:regraph}
Given $\Lambda > 0$, there is a $\gamma(\Lambda)$ so that if $\bC, \bC' \in \cC_\Lambda$, $a \in \R^8$, $0 \leq \tau, \beta \leq \gamma$ are such that
\[
|a| \leq \tau, \quad d_H(\bC' \cap \del B_1, \bC \cap \del B_1) \leq \tau, 
\]
and $v : (a + \bC') \cap A_{1, 1/2}(a) \to \bC'^\perp$ is a $C^2$ function satisfying $|v|_{C^2} \leq \beta$ then we can write
\[
\left( G_{a + \bC'}(v) \cap A_{1, 1/2}(a) \right) \cap A_{1-a, 1/2+a}(0) = G_\bC(u) \cap A_{1-a, 1/2+a}(0),
\]
for $u : \bC \cap A_{1-a, 1/2 + a}(0) \to \bC^\perp$ a $C^2$ function satisfying $|u|_{C^2} \leq c(\Lambda)(\beta + \tau)$.
\end{lemma}

\begin{proof}
Follows by Theorem \ref{thm:cones}:(\ref{item:cones-3}),(\ref{item:cones-4}) and the inverse function theorem (c.f. Section \ref{sec:param}).
\end{proof}

\section{Cone regions}\label{sec:cone-region}

Here we define various notions of ``cone region,'' and prove a key smallness estimate in Theorem \ref{thm:scone}.  Ultimately, we will only be interested in strong-cone regions of hypersurfaces, but we will use cone regions, weak-cone regions, and their varifold formulations as intermediary constructions.  Since we want to keep track of multiplicity it is most useful to phrase these in terms of varifolds.

\begin{definition}[Weak cone region]\label{def:wcone}
Let $g$ be a $C^2$ metric on $B_R(a) \subset \R^8$, and $V$ an integral varifold in $(B_R(a), g)$.  Take $\bC \in \cC$, $m \in \N$, $\beta, \tau, \sigma \in [0, 1/4]$.  We say $V \llcorner (A_{R, \rho}(a), g)$ is a $(\bC, m, \beta, \tau, \sigma)$-weak-cone region\footnote{Let us emphasize that, although for ease of notation we write $A_{R, \rho}(a)$, the actual annulus in which $V$ resembles a cone will be some perturbation of $A_{R, \rho}(a)$ contained in $A_{R, \rho}(a)$.} if the following occurs: For every $r \in [\rho, (1-\sigma)R) \cap (0, \infty)$, there is an $a_r \in \R^8$ with $B_{r}(a_r) \subset B_R(a)$, a cone $\bC_r \in \cC(\bC)$, and $C^2$ functions $\{ u_{r, j} : (a_r + \bC_r) \cap A_{r, r/8}(a_r) \to \bC^\perp_r \}_{j=1}^m$ satisfying
\begin{equation}\label{eqn:def-cone-2}
r^{-1} |u_{r,j}| + |\nabla u_{r,j}|  + r|\nabla^2 u_{r,j}| \leq \beta, \quad j = 1, \ldots, m,
\end{equation}
so that $V$ is a multi-graph:
\begin{equation}\label{eqn:def-cone-1}
V \llcorner A_{r, r/8}(a_r) = \sum_{j=1}^m [G_{a_r + \bC_r}(u_{r,j}) \cap A_{r, r/8}(a_r)]_g , 
\end{equation}
and the density ratios are close to constant:
\begin{equation}\label{eqn:def-cone-3}
m\theta_\bC(0) - \beta \leq \theta_V(a_r, r) \leq m \theta_\bC(0) + \beta .
\end{equation}
Additionally, we ask that for all $\rho \leq r \leq s \leq \min\{ 2r, (1-\sigma)R\}$ we have
\begin{equation}\label{eqn:def-cone-tau}
|a_r - a_s| \leq \tau s , \quad d_H(\bC_r \cap B_1, \bC_s \cap B_1) \leq \tau.
\end{equation}
When there is ambiguity we may be explicit and write $a_r = a_r(V)$, $\bC_r = \bC_r(V)$ to specify the varifold in question.
\end{definition}

Provided $\beta$, $\tau$, $\sigma$ are sufficiently small, then after shrinking our radius and enlarging $\beta$, every weak-cone region is actually a ``cone region.''
\begin{definition}[Cone region]
In the notation of the above, we say $V \llcorner (A_{R, \rho}(a), g)$ is a $(\bC, m, \beta)$-cone region if for every $r \in [\rho, R] \cap (0, \infty)$ there is as $\bC_r \in \cC(\bC)$ and $C^2$ functions $\{ u_{r,j} : (a + \bC_r) \cap A_{r, r/8}(a) \to \bC_r^\perp\}_{j=1}^m$ so that \eqref{eqn:def-cone-2}, \eqref{eqn:def-cone-1}, \eqref{eqn:def-cone-3} hold with $a_r \equiv a$.
\end{definition}

A much more non-trivial theorem, and the main theorem of this section, says that provided $\beta$ is sufficiently small, every cone region is a ``strong-cone'' region (for some larger choice of $\beta$).
\begin{definition}[Strong cone region]
In the notation of the above, we say $V \llcorner (A_{R, \rho}(a), g)$ is a $(\bC, m, \beta)$-strong-cone region if there are $C^2$ functions $\{ u_j : (a + \bC)\cap A_{R, \rho/8}(a) \to \bC^\perp\}_{j=1}^m$ so that \eqref{eqn:def-cone-2}, \eqref{eqn:def-cone-1}, \eqref{eqn:def-cone-3} hold for every $r \in [\rho, R] \cap (0, \infty)$ with $a_r \equiv a$, $\bC_r \equiv \bC$, and $u_{r,j} \equiv u_j$.
\end{definition}

It will be convenient to also define cone regions for hypersurfaces:
\begin{definition}[Cone regions for $M$]
If $M$ is a $C^2$ hypersurface in $B_R(a)$, then in the notation of the above we say $M \cap (A_{R, \rho}(a), g)$ is a  $(\bC, m, \beta)$-(strong-)cone region if $[M]_g \llcorner (A_{R, \rho}(a), g)$ is a $(\bC, m, \beta)$-(strong-)cone region.
\end{definition}

A few remarks.  If $V \in \cISV_7(B_R(a), g)$, and $V \cap (A_{R, \rho}(a), g)$ is a $(\bC, m, \beta)$-cone region, then $V$ is regular in $A_{R, \rho}(a)$ by the maximum principle and the fact that $\dim(\sing V) = 0$, and any two graphs in \eqref{eqn:def-cone-1} either coincide or are disjoint.  If $\rho = 0$, then the maximum principle due to \cite{ilmanen} implies $V = m[\spt V]_g$ in $B_R(a)$, and $[\spt V]_g \cap (A_{R,0}(a), g)$ is a $(\bC, 1, \beta)$-cone region.  \eqref{eqn:def-cone-1} is just saying that $\spt V \cap A_{r, r/8}(a_r)$ splits into at most $m$ conical graphs, but for various reasons it is useful to keep track of the multiplicity of these graphs.

The following lemma proves that weak-cone regions are effectively cone regions, after recentering/rescaling, and makes explicit that weak-cone regions can be recentered, and that extending the recentered region to smaller inner radii is equivalent to extending the original cone region.  These last facts are important in our contradiction argument later on.
\begin{lemma}\label{lem:recenter}
Let $V \llcorner (A_{R, \rho}(a), g)$ be a $(\bC, m, \beta, \tau, \sigma)$-weak-cone region, for $\bC \in \cC_\Lambda$.  Assume $\sigma \in [0, 1/4]$, $\tau \leq \sigma/10$.  Then:

\begin{enumerate}
\item \label{item:aras} We have $|a_r - a_s| \leq 5\tau \max \{s, r\}$ for all $r, s \in [\rho, (1-\sigma)R]$.  Hence, the point
\[
a_\rho = \left\{ \begin{array}{l l} a_\rho & \rho > 0 \\ \lim_{r \to 0} a_r & \rho = 0 \end{array} \right.
\]
exists, and satisfies $|a_\rho - a| \leq 2\sigma R$.

\item \label{item:recenter} Take $R' \leq (1-2\sigma )R$.  Given any $r \in [\rho, (1-\sigma)R']$, and $\rho' \leq r$, then $V \llcorner (A_{R',\rho'}(a_r), g)$ is a $(\bC, m, \beta, \tau, \sigma)$-weak-cone region $\iff$ $V \llcorner (A_{R, \rho'}(a), g)$ is a $(\bC, m, \beta, \tau, \sigma)$-weak-cone region.  In particular, $V \llcorner (A_{R', \rho}(a_\rho), g)$ is a $(\bC, m, \beta, \tau, \sigma)$-weak-cone region.

\item \label{item:cone-region} There is a $\gamma(\Lambda) > 0$ so that if, additionally, $\beta, \tau \leq \gamma$, then $V \llcorner A_{(1-3\sigma) R, \rho}(a_\rho)$ is a $(\bC, m, c(\Lambda, m)(\beta + \tau))$-cone region.
\end{enumerate}

\end{lemma}

\begin{proof}
Item \ref{item:aras} follows easily from \eqref{eqn:def-cone-tau} and \eqref{eqn:aar}.  We prove Item \ref{item:recenter}.  To prove the direction $\impliedby$ if will suffice to show that $B_{R'}(a_r) \subset B_R(a)$, and that for every $s \in [\rho', (1-\sigma)R']$ we have $B_s(a_s) \subset B_{R'}(a_r)$.  The first inclusion follows because
\begin{equation}\label{eqn:aar}
|a - a_r| \leq |a - a_{(1-\sigma)R}| + |a_{(1-\sigma)R} - a_r| \leq \sigma R + 5\tau(1-\sigma) R \leq 2\sigma R.
\end{equation} 
To prove the second inclusion, we observe for any $s$ as above that
\[
s + |a_s - a_r| \leq s + 5\tau \max\{ s, r\} \leq (1+5\tau) (1-\sigma)R' < R' .
\]

The direction $\implies$ is trivial if $\rho' \geq \rho$.  If $\rho' < \rho$, then it suffices to verify that $B_{s}(a_s) \subset B_R(a)$ for $s \in [\rho', \rho]$.  But by assumption and \eqref{eqn:aar} we have
\[
B_s(a_s) \subset B_{R'}(a_r) \subset B_R(a).
\]
This proves Item \ref{item:recenter}.

To prove Item \ref{item:cone-region}, we first observe that by Item \ref{item:aras}, \eqref{eqn:aar} and our assumption $\tau \leq \sigma/10$, we get for every $r \in [\rho, (1-3\sigma)R] \cap (0, \infty)$:
\begin{gather}
B_r(a_\rho) \subset B_{(1+10\tau)r}(a_{(1+10\tau)r}) \subset B_{R}(a), \label{eqn:recenter-1} \\
B_r(a_\rho) \supset B_{(1-5\tau)r}(a_{(1-5\tau)r}) \quad \text{ if } (1-5\tau)r \geq \rho,  \label{eqn:recenter-2} \\
\text{and} \quad (1+10\tau)r \leq (1-\sigma)R . \label{eqn:recenter-3}
\end{gather}
Therefore, ensuring $\tau(\Lambda), \beta(\Lambda)$ are sufficiently small, for every such $r$ we can apply Lemma \ref{lem:regraph} in each annulus $A_{r, r/2}(a_\rho)$ to get \eqref{eqn:def-cone-1}, \eqref{eqn:def-cone-2}.

To get the upper density bound \eqref{eqn:def-cone-3}, we can use \eqref{eqn:recenter-1}, \eqref{eqn:recenter-3} to get for $r$ as above:
\begin{align*}
\theta_V(a_\rho, r) \leq (1+10\tau)^7\theta_V(a_{(1+10\tau)r}, (1+10 \tau) r) \leq m \theta_\bC(0) + c(\Lambda, m) (\tau + \beta).
\end{align*}
The lower bound is similar, using \eqref{eqn:recenter-2}:
\begin{align*}
\theta_V(a_\rho, r) &\geq (1-5\tau)^7 \theta_V(a_{ \max\{ (1-5\tau) r, \rho\}}, \max\{ (1-5\tau)r, \rho \}) \\
&\geq m\theta_\bC(0) - c(\Lambda, m)(\beta + \tau).
\end{align*}
This proves Item \ref{item:cone-region}.
\end{proof}

\vspace{3mm}

We now work towards proving Theorem \ref{thm:scone}, which is a smallness estimate that implies cone regions are in fact strong-cone regions.  We first require a helper lemma which implies that small graphicality ``propogates'' to nearby scales.

\begin{lemma}\label{lem:cone-extend}
Given $\eps, \Lambda > 0$, there is a $\gamma(\Lambda)$ so that the following holds.  Let $\bC \in \cC_\Lambda$, $g$ be a $C^2$ metric on $B_1$, and $V$ an integral varifold in $(B_1, g)$.  Suppose $V \llcorner (A_{1, 1/2}, g)$ is a $(\bC, m, \beta)$-cone region for $\beta \leq \gamma$, and
\[
V \llcorner A_{1, 1/2} = \sum_{j=1}^m [G_\bC(u_j) \cap A_{1, 1/2}]_g, \quad |u_j|_{C^2} \leq \eps \leq \gamma.
\]
Then
\[
V \llcorner A_{1, 1/16} = \sum_{j=1}^m [G_\bC(u_j) \cap A_{1, 1/16}]_g, \quad |u_j|_{C^2} \leq c(\Lambda)(\eps + \beta).
\]
\end{lemma}

\begin{proof}
Follows from the definition of cone region and Lemma \ref{lem:regraph}.
\end{proof}

\begin{theorem}\label{thm:scone}
Given $\eps, m, \Lambda > 0$, there are $\beta(\Lambda, m, \eps)$, $\delta(\Lambda, m, \eps)$ so that the following holds.  Let $g$ be a $C^3$ metric on $B_1$ such that $|g - \geucl|_{C^3(B_1)} \leq \beta$.  Take $\bC \in \cC_\Lambda$, $m \in \N$, $\rho \in [0, 1]$, and let $V \in \cISV_7(B_1, g)$ be such that $V \llcorner (A_{1, \rho}, g)$ is a $(\bC, m, \beta)$-cone region.  Suppose there are $C^2$ functions $\{ u_j : \bC \cap A_{1,1/2} \to \bC^\perp \}_j$ such that
\begin{equation}\label{eqn:scone-hyp}
V \llcorner A_{1, 1/2} = \sum_{j=1}^m [G_\bC(u_j) \cap A_{1, 1/2}]_g, \quad |u_j|_{C^2(\bC \cap A_{1, 1/2})} \leq \delta.
\end{equation}

Then the $u_j$ can be extended to $C^2$ functions on $\bC \cap A_{1, \rho/8}$, so that
\begin{equation}\label{eqn:scone-concl1}
V \llcorner A_{1, \rho/8} = \sum_{j=1}^m [G_\bC(u_j) \cap A_{1, \rho/8}]_g, \quad |x|^{-1} |u_j| + |\nabla u_j| + |x| |\nabla^2 u_j| \leq \eps .
\end{equation}
In particular, $V \cap (A_{1, \rho}, g)$ is a $(\bC, m, \eps)$-strong-cone region.

If $\theta_\bC(0) = 1$ then we additionally have the curvature estimate
\begin{equation}\label{eqn:planar-A-bound}
\int_{A_{1, \rho/8}} |A_{V, g}|^7 d\mu_V(x) \leq \eps
\end{equation}
where $|A_{V, g}|$ is the length of the second fundamental form of $\reg V \subset (B_1, g)$.
\end{theorem}

\begin{remark}\label{rem:scone-stable}
If $V$ is stable in $(B_1, g)$, then in place of assuming that $V \llcorner (A_{1, \rho}, g)$ is a $(\bC, m, \beta)$-cone region, one can instead assume that
\begin{equation}\label{eqn:scone-rem-stable}
\theta_V(0, 1) \leq m \theta_\bC(0) + \beta, \quad \theta_V(0, \rho/16) \geq m \theta_\bC(0) - \beta.
\end{equation}
Compare to Theorem \ref{thm:scone-1}, Lemma \ref{lem:cone-extend-1}.
\end{remark}

\begin{remark}
In place of assuming \eqref{eqn:def-cone-3} in Theorem \ref{thm:scone} one could instead assume directly a bound of the form
\begin{equation}\label{eqn:scone-rem-ann}
\int_{A_{1, \rho/8}} |x|^{-9}|\pi_V^\perp(x)|^2 d\mu_V(x) \leq \beta .
\end{equation}
In this case it would suffice to know $V$ is only defined and stationary in the annulus $A_{1, \rho/8}$.
\end{remark}

\begin{remark}
By Remark \ref{rem:decay} (and considering the change of variables $t = \log r$) the same conclusion in Theorem \ref{thm:scone} holds if instead of \eqref{eqn:scone-hyp} we assume there is a $\rho' \in [\rho/4, 1]$ so that
\[
V \llcorner A_{\rho', \rho'/2} = \sum_{j=1}^m [G_\bC(u_j) \cap A_{\rho', \rho'/2}]_g
\]
for some $C^2$ functions $u_j$ satisfying $\rho'^{-1} |u_j| + |\nabla u_j| + \rho' |\nabla^2 u_j| \leq \delta$.
\end{remark}

\begin{remark}\label{rem:dini}
If $g(0) = \geucl$ and $Dg|_0 = 0$ then we additionally get the Dini-type estimate
\begin{equation}\label{eqn:scone-concl2}
\int_{ A_{1, \rho/8}} |x|^{-8} |\pi_V^\perp(x)| d\mu_V(x) \leq \eps, 
\end{equation}
which can be viewed as a quantitative version of the statement that tangent cones are unique.  Here (like in \eqref{eqn:sharp-mono}) $\pi^\perp_V|_x$ is the Euclidean orthogonal projection onto the Euclidean normal space $T_x^\perp V$.

One can recast \eqref{eqn:scone-concl2} as a $C^0$-Dini-type estimate in the following fashion.  In analogue to the planar Jones' $\beta$-numbers, let us define the ``conical $\beta$-numbers'' $\beta_\bC(V, a, r)$ by
\[
\beta_\bC(V, a, r) = \inf_{\bC' \in \cC(\bC)} r^{-1} d_H(\spt V \cap A_{r, r/8}(a), (a + \bC') \cap A_{r, r/8}(a)).
\]
With $\beta$, $V$ as in Theorem \ref{thm:scone}, there is a $\gamma(\bC) \geq 2$ so that if $V \llcorner (A_{1, \rho}, g)$ is a $(\bC, 1, \beta)$-strong-cone region, then $V$ satisfies the estimate
\begin{equation}\label{eqn:scone-dini}
\int_\rho^1 \beta_\bC(V, 0, r)^\gamma \frac{dr}{r} \leq \eps.
\end{equation}
Here $\gamma$ is the $C^0$-Lojacsiewicz exponent (\cite[(2.1)]{simon}), and if $\bC$ is integrable then one can take $\gamma = 2$.  In particular, if $\bC$ is planar then
\begin{equation}\label{eqn:scone-concl3}
\int_{A_{1,\rho/8}} |x|^{-5} |A_{V, g}|^2 d\mu_V(x) \leq \eps,
\end{equation}
which is a ``deeper'' reason for estimate \eqref{eqn:planar-A-bound}.  If $m \geq 2$, then by the maximum principle $V \llcorner A_{1, \rho/8}$ splits into $m$ varifolds $V_j$, each of which satisfies \eqref{eqn:scone-dini}.
\end{remark}

\begin{remark}
By using sharper linear estimates, it probably suffices to assume $g$ is only $C^2$ close to Euclidean (see \cite{simon-extrema}).
\end{remark}

\begin{proof}
I claim that it suffices to prove Theorem \ref{thm:scone} with $\delta$ depending on $(\bC, m, \eps)$.  Let us assume we can do this, and show how to get $\delta$ to depend only on $(\Lambda, m, \eps)$.  Let $c_2(\Lambda)$, $\gamma_2(\Lambda)$ be the constants from Lemma \ref{lem:regraph}, and $\eps_0(\Lambda)$ the constant from Theorem \ref{thm:cones}.  For each $\bC \in \cC_\Lambda$, apply (by assumption) Theorem \ref{thm:scone} with $\eps/(10 c_2)$ in place of $\eps$, to obtain a $\delta_\bC(\bC, m, \eps)$ and $\beta_\bC(\bC, m, \eps)$.  WLOG of course we can assume $\delta_\bC \leq \min(\eps, \gamma_2, \eps_0)$.  Now we can apply Lemma \ref{lem:cover} to obtain a finite collection $\bC_1, \ldots, \bC_\Lambda$, so that given any $\bC \in \cC_\Lambda$, we can find an $i$ so that
\begin{equation}\label{eqn:scone-30}
d_H(\bC \cap \del B_1, \bC_i \cap \del B_1) < \delta_{\bC_i}/(10 c_2) 
\end{equation}
Set $\delta = \min_{i = 1, \ldots, N} \delta_{\bC_i}/(10 c_2)$, $\beta = \min_{i=1, \ldots, N} \beta_{\bC_i}$.

Now take a general $\bC \in \cC_\Lambda$, and suppose \eqref{eqn:scone-hyp} holds with this $\bC$ and $\beta, \delta$ as chosen above.  Pick $\bC_i$ so that \eqref{eqn:scone-30} holds, and then by Lemma \ref{lem:regraph} we can write
\[
V \llcorner A_{1, 1/2} = \sum_{j=1}^m [G_{\bC_i}(\tilde u_j) \cap A_{1, 1/2}]_g, \quad |\tilde u_j|_{C^2} \leq 2 c_2 \delta_{\bC_i}/(10 c_2) \leq \delta_{\bC_i}.
\]
Moreover, by Theorem \ref{thm:cones} we get that $V \llcorner (A_{1, \rho}, g)$ is a $(\bC_i, m, \beta_\bC)$-cone region.  By our choice of constants we can apply Theorem \ref{thm:scone} to deduce
\[
V \llcorner A_{1, \rho/8} = \sum_{j=1}^m [G_{\bC_i}(\tilde u_j) \cap A_{1, \rho/8}]_g , \quad |x|^{-1} |\tilde u_j| + |\nabla \tilde u_j| + |x| | \nabla^2 \tilde u_j| \leq \eps/(10 c_2),
\]
and then finally apply Lemma \ref{lem:regraph} at each scale $r \in [\rho/4, 1]$ to get \eqref{eqn:scone-concl1}.  This proves $\delta$ can be made independent of $\bC$.


\vspace{3mm}

I next claim that it suffices to prove Theorem \ref{thm:scone} under the assumption that $g(0) = \geucl$, $Dg|_0 = 0$.  We prove this.  Let $V$, $\bC$, $g$ satisfy the hypotheses of Theorem \ref{thm:scone}.  By taking $\phi^{-1}$ to be a quadratic perturbation of a linear map, we can find a smooth diffeomorphism $\phi : B_{1/2} \to B_1$ satisfying
\begin{equation}\label{eqn:scone-40}
|\phi(x) - x| \leq c \beta |x|, \quad |D\phi - Id| + |D^2 \phi| \leq c \beta,
\end{equation}
so that the pullback metric $g' = \phi^* g$ satisfies
\[
g'(0) = \geucl, \quad Dg'|_0 = 0, \quad |g' - \geucl|_{C^3(B_{1/2})} \leq c \beta,
\]
where $c$ is an absolute constant.

From Theorem \ref{thm:cones}, \eqref{eqn:scone-40}, and the inverse function theorem, ensuring $\beta(\Lambda)$ is sufficiently small, if $V' = (\phi^{-1})_\sharp V \llcorner B_{1/2}$ then $V' \in \cISV_7(B_{1/2}, g')$, $V' \llcorner (A_{1/2, 2\rho}, g')$ is a $(\bC, m, c(\Lambda) \beta)$-cone region, and (using Lemma \ref{lem:cone-extend}) we can find $C^2$ functions $\{ u_j' : \bC \cap A_{1/2, 1/4} \to \bC^\perp \}_{j=1}^m$ so that
\[
V' \llcorner A_{1/2, 1/4} = \sum_{j=1}^m [G_\bC(u_j') \cap A_{1/2, 1/4}]_{g'}, \quad |u_j'|_{C^2(\bC \cap A_{1/2, 1/4})} \leq c(\Lambda)(\delta + \beta) .
\]

For any $\eps' > 0$, ensuring $\beta(\Lambda, m, \eps')$, $\delta(\Lambda, m, \eps')$ are sufficiently small, we can apply Theorem \ref{thm:scone} to the rescaled varifold/metric $(\eta_{0, 1/2})_\sharp V'$, $g' \circ \eta_{0, 1/2}^{-1}$ to deduce
\[
V' \llcorner A_{1/2, \rho/4} = \sum_{j=1}^m [G_\bC(u_j') \cap A_{1/2, \rho/4}]_{g'}, \quad |x|^{-1} |u_j'| + |\nabla u_j'| + |x| |\nabla^2 u_j'| \leq \eps'.
\]
We can again apply the inverse function theorem and Lemma \ref{lem:cone-extend} to deduce \eqref{eqn:scone-concl1} with $c(\Lambda)(\eps' + \beta)$ in place of $\eps$.  Provided $\eps'(\Lambda, \eps)$, $\beta(\Lambda, \eps)$ are sufficiently small, this proves \eqref{eqn:scone-concl1}.

If $\theta_\bC(0) = 1$, then since $\phi$ is an isometry we obtain from the previous paragraph
\[
\int_{A_{1/4, \rho/2}} |A_{V, g}|^7 d\mu_V \leq \eps' .
\]
\eqref{eqn:planar-A-bound} then follows by \eqref{eqn:def-cone-2}, provided $\eps' \leq \eps/2$ and $\beta(m, \eps)$ is sufficiently small.  This proves my second reduction.

\vspace{3mm}

We now prove Theorem \ref{thm:scone} with $\bC$ fixed, and $\delta$, $\beta \leq \delta$ depending on $(\bC, m, \eps)$, and under the assumption that $g(0) = \geucl$, $Dg|_0 = 0$.  Note this implies that
\begin{equation}\label{eqn:scone-41}
|g|_x - \geucl| \leq \beta |x|^2, \quad |Dg|_x| \leq \beta |x|.
\end{equation}
We note also that there is no loss in assuming $\eps(\bC)$ is as small as we like.  Let $\rho_*$ be the least radius for which we have
\begin{equation}\label{eqn:scone1}
V \llcorner  A_{1, \rho_*} = \sum_{j=1}^m [G_\bC(u_j) \cap A_{1, \rho_*}]_g, \quad |x|^{-1} |u_j| + |\nabla u_j| + |x| |\nabla^2 u_j| \leq \eps .
\end{equation}
By Lemma \ref{lem:cone-extend}, ensuring $\delta(\Lambda, \eps, m)$, $\beta(\Lambda, \eps, m)$ are sufficiently small, we can assume $\rho_* \leq \max(e^{-6}, \rho/8)$.  Note also that Lemma \ref{lem:cone-extend} and our restriction $\beta \leq \delta$ imply we can assume
\begin{equation}\label{eqn:scone-43}
|u_j|_{C^2(\bC \cap A_{1, e^{-2}})} \leq c(\Lambda) \delta.
\end{equation}

We will suppose $\rho_* > \rho/8$, and derive a contradiction.  By the maximum principle and our singular set bound we can assume the graphs $G_\bC(u_j) \cap A_{1, \rho_*}$ are disjoint, and in particular each graph $G_\bC(u_j) \cap A_{1, \rho_*}$ is stationary in $(A_{1, \rho_*}, g)$.  By this and by Lemma \ref{lem:cone-extend}, it will suffice to show that for any $0 < \eps' \leq \eps$ of our choosing, we can pick any one $u = u_j$ and get the bound
\begin{equation}\label{eqn:scone-10}
|x|^{-1} |u| + |\nabla u| + |x| |\nabla^2 u| \leq \eps' \quad x \in \bC \cap A_{e^{-1}, e^3 \rho_*},
\end{equation}
provided $\delta(\bC, m, \eps')$, $\beta(\bC, m, \eps')$ are sufficiently small.  Our strategy is to use the decay-growth Theorem \ref{thm:decay}, to show small bounds on $u$ must persist as long as the density bound \eqref{eqn:scone-hyp} holds.  We will use notation from \cite{simon}; see also Appendix \ref{sec:decay}.

Let $\Sigma = \bC \cap \del B_1$, and write $\hat\nabla$ for the connection on $T^\perp \Sigma \subset TS^7$.  Write $r = |x|$, and define the change of variables $t = -\log r$, $v(t, \theta) = r^{-1} u(r\theta)$, so that $v$ is a $C^2$ function on $\Sigma \times (0, T_*)$ for $T_* = -\log \rho_*$.  Write $\dot v \equiv \del_t v$.  Recall the norms $|v|_k^*(t)$ (defined in Section \ref{sec:decay}), and $||v(t)|| = ||v(t, \cdot)||_{L^2(\Sigma)}$.  We have
\begin{equation}\label{eqn:scone-11}
\frac{1}{c} |v|_2^*(t) \leq \sup_{x \in e^{-t} \Sigma } |x|^{-1} |u| + |\nabla u| + |x| |\nabla^2 u| \leq c |v|_2^*(t)
\end{equation}
for every $t$, and some absolute constant $c$.  Define the function $G : \Sigma \times (\rho_*, 1) \to \R^8$ by setting
\[
G(\theta, r) = \frac{r \theta + u(r\theta)}{\sqrt{1+ r^{-1} |u(r\theta)|^2}} \equiv \frac{r \theta + r v}{\sqrt{1+ |v|^2}}
\]

Provided $\eps(\Sigma)$ is sufficiently small, then for any $0 \leq a < b \leq T_*$ we can write
\begin{align*}
\cF_{a, b}(v) &:= \haus^7_g(G_\bC(u) \cap A_{e^{-b}, e^{-a}}) \\
&= \int_a^b \int_\Sigma e^{-7t}( F(\theta, v, \hat\nabla v, \dot v) + H(\theta, t, v, \hat\nabla v, \dot v)) d\theta dt,
\end{align*}
where the integrand is really just the Jacobian of the map $(\theta, t) \mapsto G(\theta, e^{-t})$ w.r.t. to the metric $g$, and $e^{-7t} F$ is the Jacobian w.r.t. $\geucl$.  By embedding $T^\perp \Sigma , T\Sigma \subset T\R^8$, one can think of $F(\theta, z, p, q)$, resp. $H(\theta, t, z, p, q)$, as being defined on $\Sigma \times \R^8 \times \R^{64} \times \R^8$, resp. $\Sigma \times \R \times \R^8 \times \R^{64} \times \R^8$ (c.f. \cite{simon-extrema}).

$F$, $H$ satisfy the following properties: $F(\theta, z, p, q)$ is smooth (and analytic in $z, p, q$) and independent of $g$, and takes the form
\[
F(\theta, z, p, q)^2 = (1+ |q|^2 + q^2 \cdot z^2 \cdot F_1(z)) E(\theta, z, p)^2 + q^2 \cdot p^2\cdot  F_2(\theta, z, p),
\]
for $F_1, F_2$ smooth, and $E(\theta, z, p) \equiv F(\theta, z, p, 0)$; the functional
\[
\cE(v(t)) := \int_\Sigma E(\theta, v, \hat\nabla v) \equiv e^{6t} \haus^6(G_\bC(u) \cap \del B_{e^{-t}})
\]
satisfies the convexity and analyticity hypotheses of \cite[(1.2), (1.3)]{simon}, and has the form
\[
E(\theta, z, p)^2 = 1 + z^2 \cdot E_1(\theta, z, p) + p^2 \cdot E_2(\theta, z, p)
\]
for smooth functions $E_1, E_2$; $H(\theta, t, z, p, q)$ is $C^3$ in the $(\theta, t, z)$ variables, smooth in the $(p, q)$ variables, and from \eqref{eqn:scone-41} $H$ satisfies the bounds
\[
\sum_{\substack{0 \leq k \leq 3 \\ 0 \leq i + j \leq 1}} |D^k D_p^i D_q^j H| \leq c(\Sigma) \beta e^{-2t},
\]
here $D$ being the full derivative on $\Sigma \times \R \times \R^8 \times \R^{64} \times \R^8$.  In the above we are abusing notation slightly: when we write $F_2$ (e.g.) we really mean a collection of functions $F_{2 ijkl}^{\alpha \beta}(\theta, z, p)$ defined on $\Sigma \times \R^8 \times \R^{64}$ so that
\[
q^2 \cdot p^2 \cdot F_2(\theta, z, p) \equiv \sum_{\substack{i, j, k, l = 1, \ldots, 8 \\ \alpha, \beta = 1, \ldots, 8}} q^i q^j p^k_\alpha p^l_\beta F_{2ijkl}^{\alpha \beta}(\theta, z, p)
\]

A straightforward computation then shows that stationarity of $v$ for the area functional $\cF_{a, b}$ implies that $v$ solves a PDE of the form
\begin{gather}\label{eqn:scone-5}
-\ddot v + 7 \dot v - \cM(v) - \cR_1(v) - \cR_2(v) = f
\end{gather}
where, as in Section \ref{sec:decay}, $\cM = -\mathrm{grad} \cE$ in the $L^2(\Sigma)$ sense; $\cR_1$ is independent of $g$ and has the form of Section \ref{sec:decay}:Item \ref{item:decay-2}; and $\cR_2$, $f$ are as in Section \ref{sec:decay}:Items \ref{item:decay-1}, \ref{item:decay-3}, except with $c(\Sigma) \beta e^{-2t}$ in place of $\delta e^{-\eps t}$.

As exploited in \cite{simon}, for any $\delta' > 0$ provided $\beta(\Sigma, \delta')$ and $\eps(\Sigma, \delta')$ are sufficiently small then both $v$ and $w = \dot v$ solves an PDE of the form
\begin{gather}\label{eqn:scone-1}
-\ddot w + 7 \dot w - L w + e_1 \cdot \hat\nabla^2 w + e_2 \cdot \hat\nabla \dot w + e_3 \ddot w + e_4 \cdot \hat\nabla w + e_5 \dot w + e_6 w = f'
\end{gather}
for $e_i(x, t)$, $f'$ being $C^1$ functions (different for $v$ and $w$) having the estimates $|e_i|_1^*(t) \leq \delta'$, $|f'|_1^*(t) \leq c(\Sigma) \beta e^{-2t}$, and $L$ being the (elliptic) linearization of $\cM$ at $0$.  In fact, for us $L = \hat \Delta_\Sigma + |A_\Sigma|^2 + 6$ where $A_\Sigma$ is the second fundamental form of $\Sigma \subset S^7$, and $\hat \Delta_\Sigma$ the connection Laplacian on $T^\perp \Sigma$.  Ensuring $\delta'(\Sigma)$ is small, we can apply standard elliptic theory to get
\begin{gather}\label{eqn:scone-4}
|v|_2(t)^* \leq c(\Sigma) ( ||v||_{L^1(\Sigma \times (t-1,t+1))} + \beta e^{-2t} ) \quad t \in [1, T_*-2],
\end{gather}
and the same for $\dot v$.

\vspace{3mm}

The key that allows us to apply Theorem \ref{thm:decay} to $v$ is the ``Hardt-Simon'' type inequality (c.f. \cite{Simon1}): provided $\eps(\Sigma)$ is sufficiently small, then we have
\begin{equation}\label{eqn:scone-2}
|\dot v(t, \theta)|^2 \leq 2 e^{2t} |\pi^\perp_V(G(\theta, e^{-t}))|^2 \leq 4 |\dot v(t, \theta)|^2 , \quad \forall t \in (0, T_*), \theta \in \Sigma,
\end{equation}
where $\pi^\perp_V$ is the Euclidean projection onto the Euclidean normal direction $T^\perp_{G(\theta, r)} V$.  This follows by first observing that $r \mapsto G(r, \theta)$ is a curve in $\reg V$, and hence 
\[
-\frac{\pi^\perp_V(r \del_r v)}{\sqrt{1+|v|^2}} = (1-(r \del_r v) \cdot (G/r))\pi^\perp_V(G/r),
\]
and then second observing that we have the bound $|\pi^\perp V - \pi^\perp_{T_{r \theta} \bC}| \leq c(\Sigma) \eps$.

Using the monotonicity \eqref{eqn:sharp-mono}, \eqref{eqn:scone-41}, our hypothesis on $g$, and observing that the Jacobian of $(r, \theta) \mapsto G(\theta, r)$ is $\geq (1-c(\Sigma) (\eps + \beta))r^6$, we can use \eqref{eqn:scone-2} to estimate:
\begin{align*}
\int_{t}^{s} ||\dot v(t)||^2 dt 
&\leq 2 \int_{e^{-s}}^{e^{-t}} \int_\Sigma r^{-3} |\pi^\perp_V(G(r, \theta))|^2 d\theta dr \\
&\leq 4 \int_{A_{e^{-s}, e^{-t}} \cap V} \frac{|\pi^\perp_V(x)|^2}{|x|^{7+2}} d\mu_V(x) \\
&\leq c(\Lambda) \beta,
\end{align*}
provided $\eps(\Sigma), \beta(\Sigma)$ are sufficiently small, and $0 < s < t \leq T_* - 3$

We deduce that, for $s, t$ as above, we have
\begin{equation}\label{eqn:scone-3}
||v(t) - v(s)|| \leq \int_s^t ||\dot v(t)|| dt \leq c(\Lambda) \beta^{1/2} |t - s|^{1/2} .
\end{equation}
Therefore, combining \eqref{eqn:scone-3} with estimates \eqref{eqn:scone-4}, we obtain for any $s, t \in [1, T_* - 3]$ the bounds
\begin{align}
|v|_2^*(t)
&\leq c(\Sigma) ( \sup_{t' \in (t-1,t+1)} ||v(t)|| + \beta) \nonumber \\
&\leq c(\Sigma) |v|_2^*(s) + c(\Sigma) \beta^{1/2} \max\{ 1, |t - s| \}, \label{eqn:scone-7}
\end{align}
and hence $v$ has $(c(\Sigma), c(\Sigma) \beta^{1/2})$-linear growth on $\Sigma \times [1, T_* - 3]$.

On the other hand, we can use \eqref{eqn:scone-5}, the $C^1$ Lojasiewicz-Simon inequality for $\cE$ \cite[(2.2)]{simon}, and \eqref{eqn:scone-4} to get
\begin{align}
|\cE(v(t)) - \cE(0)|^{1-\theta}
&\leq ||\cM(v(t))|| \nonumber \\
&\leq c(\Sigma) |\dot v|_1^*(t) + c(\Sigma) \beta e^{-2t} \nonumber \\
&\leq c(\Sigma) ||\dot v||_{L^2(\Sigma \times (t-1, t+1) )} + c(\Sigma) \beta e^{-2t} \nonumber \\
&\leq c(\Sigma) \beta^{1/2}, \label{eqn:scone-6}
\end{align}
for any $t \in [1, T_* - 3]$.  Here $\theta(\Sigma) \in (0, 1/2)$ is the Lojasiewicz exponent as in \cite[Theorem 3]{simon}.

Our hypothesis \eqref{eqn:scone-43}, and \eqref{eqn:scone-7}, \eqref{eqn:scone-6}, imply that provided $\delta(\Sigma)$, $\beta(\delta, \Sigma)$ are sufficiently small, we can apply Theorem \ref{thm:decay} to deduce that
\begin{equation}\label{eqn:scone-20}
|v|_2^*(t) \leq c(\Sigma) \delta^\alpha \quad \forall t \in [1, T_* - 3], 
\end{equation}
for some $\alpha(\Sigma) \in (0, 1/2)$.  In particular, by \eqref{eqn:scone-11}, ensuring $\delta(\Sigma, \eps')$ is small, we deduce \eqref{eqn:scone-10}.  From Theorem \ref{thm:decay} we also have the bound
\begin{equation}\label{eqn:scone-22}
\int_1^{T_* - 3} ||\dot v(t)|| dt \leq c(\Sigma) \delta^{\alpha/2}
\end{equation}
which together with \eqref{eqn:scone-20}, \eqref{eqn:scone-2} implies \eqref{eqn:scone-concl2}.

Using a similar computation to \eqref{eqn:scone-6}, the $C^0$ Lojasiewicz-Simon inequality \cite[(2.1)]{simon} implies there is a $\gamma(\Sigma) \geq 2$, so that for every $t \in [1, T_*-3]$ we can find a $\zeta_t \in C^{2}(\Sigma, \Sigma^\perp)$ solving $\cM(\zeta_t) = 0$ so that
\begin{align*}
||v(t) - \zeta_t||^\gamma 
&\leq 2 ||\cM(v(t))|| \\
&\leq c(\Sigma) ||\dot v||_{L^1(\Sigma \times (t-1,t+1))}  + c(\Sigma) \beta e^{-2t} \\
&\leq c(\Sigma) \int_{t-1}^{t+1} ||\dot v(t)|| dt + c(\Sigma) \beta e^{-2t}  .
\end{align*}
Here $\gamma$ is the Lojasiewicz exponent, and we note that if $\bC$ is ```integrable'' then one can take $\gamma = 2$ (see e.g. \cite[Page 80]{simon:maps}).

If $2 \leq i \leq T_* - 7$, then by \eqref{eqn:scone-3}, the above, and \eqref{eqn:scone-20} (since $\gamma \geq 2$), we get
\[
\sup_{t \in [i, i+3]} ||v(t) - \zeta_i||^\gamma \leq c(\Sigma) \int_{i-1}^{i+4} ||\dot v(t)|| dt + c(\Sigma) \beta e^{-2t} .
\]
Setting $\zeta_i = 0$ if $i \leq 2$ or $i > T_* - 7$, we deduce there are numbers $\{\eps_i\}_i$ so that every integer $i \in (0, T_*)$ and for every $t \in [0, T_*) \cap [i-3,i]$, we have
\begin{align}\label{eqn:scone-23}
||v(t) - \zeta_i||^\gamma \leq \eps_i, \quad \sum_i \eps_i \leq c(\Sigma) \delta^{\alpha/2}.
\end{align}

When $\bC$ is planar, this in particular implies there are planes $\bC_i$ so that
\[
\left( e^i d_H(G_\bC(u) \cap A_{e^{-i}, e^{-i-3}}, \bC_i \cap A_{e^{-i}, e^{-i-3}})\right)^2 \leq \eps_i,
\]
for $i, \eps_i$ as in \eqref{eqn:scone-23}.  Hence, ensuring $\beta$ is sufficiently small, by \eqref{eqn:higher-reg} we deduce that for every $r \in [2\rho, 1/2]$ we have the bound 
\begin{align*}
\sup_{G_\bC(u) \cap A_{r, r/8}} r^2|A|^2 
&\leq c |g - \geucl|^2_{C^0(B_{2r})} + cr^2 |Dg|^2_{C^1(B_{2r})} + c\eps_i  \\
&\leq c \beta r^2 + c \eps_i, 
\end{align*}
for $c$ an absolute constant, and $A$ be the second fundamental form of $G_\bC(u) \subset (B_1, g)$.  Together with \eqref{eqn:scone-23}, \eqref{eqn:ahlfors}, this implies
\[
\int_{A_{1/2, \rho/4}} |x|^{-5} |A|^2 d\mu_V \leq c(m)(\beta + \delta^{\alpha/2}) ,
\]
which combined with \eqref{eqn:def-cone-2} and Theorem \ref{thm:cones} gives \eqref{eqn:scone-concl3}, \eqref{eqn:planar-A-bound} provided $\delta(m, \eps), \beta(m, \eps)$ are chosen sufficiently small.
\end{proof}

\section{Decomposition}\label{sec:diffeo}

Here is the core of our argument.  We show that a minimal surface lying sufficiently close to a stable minimal cone with multiplicity must break up into a ``cone decomposition,'' consisting of a controlled number of ``smooth regions'' (multi-graphs over a smooth model surface) and a controlled number of ``strong cone regions'' (multi-graphs over cones).  We already defined cone regions in the previous section.  We first define our smooth models/regions and cone decompositions.  In this section we work in $\R^8$.

A smooth model is an entire minimal hypersurface-with-multiplicity $S$ in $\R^8$ along with a choice of disjoint balls, so that $S$ is entirely smooth outside of these balls.  The multiplicities of $S$ specify the number of graphs we expect to see over each component.  Each hole more or less corresponds to a would-be singularity or point of index concentration.
\begin{definition}[Smooth model]\label{def:model}
Given $\Lambda, \gamma \geq 0$, $\sigma \in (0, 1/3)$, we say that $(S, \bC, \{ (\bC_\alpha, B_{r_\alpha}(y_\alpha) )\}_\alpha )$ is a $(\Lambda, \sigma, \gamma)$-smooth model if $S \in \cISV_7(\R^8, \geucl)$ and satisifes $\theta_S(0, \infty) \leq \Lambda$, $\bC \in \cC_\Lambda$, $\{B_{2r_\alpha}(y_\alpha)\}_\alpha$ is a finite collection of disjoint balls in $B_{1-3\sigma}$, and $\{ \bC_\alpha\}_\alpha$ is a corresponding finite collection of cones in $\cC_\Lambda$, such that the following is satisfied:
\begin{enumerate}
\item \label{item:def-model-1} There is a decomposition
\begin{equation}\label{eqn:def-model-1}
S = m_1 [S_1] + \ldots + m_k[S_k],
\end{equation}
where $m_i \in \N$, and $S_i$ are smooth, disjoint, closed (as sets), embedded minimal hypersurfaces in $\R^8 \setminus \cup_\alpha y_\alpha$.

\item \label{item:def-model-2} There is an $m \in \N$ satisfying $m \theta_\bC(0) \leq \Lambda$ so that $S \llcorner (A_{\infty, 1}, \geucl)$ is a $(\bC, m, \gamma)$-strong-cone region.

\item \label{item:def-model-3} For each $\alpha$, there is a $j$ so that $\spt S \cap A_{2r_\alpha, 0}(y_\alpha) = S_j \cap A_{2r_\alpha, 0}(y_\alpha)$ and $S_j \cap (A_{2r_\alpha, 0}(y_\alpha), \geucl)$ is a $(\bC_\alpha, 1, \gamma)$-strong-cone region.
\end{enumerate}

Often we will refer to the smooth model as simply $S$, where the existence of the associated $\bC, \{(\bC_\alpha, B_{r_\alpha}(y_\alpha))\}_\alpha$ is implicitly understood.
\end{definition}

Up to scaling, any $S \in \cISV_7(\R^8, \geucl)$ with $\theta_S(0, \infty) \leq \Lambda$ and $\mindex(S, \R^8, \geucl) < \infty$ is a smooth model for some choice of balls $\{B_{r_\alpha}(y_\alpha)\}_\alpha$.  In some sense the main content of this definition is the choice of balls, as pinning down a scale by which to measure $S$, which is made precise in the following definition.
\begin{definition}
Given a smooth $(\Lambda, \sigma, \gamma)$-model $S$, we let $\eps_S$ be the largest number $ \leq \min\{ 1, \min_\alpha \{r_\alpha\}\}$ for which the map $G_S : T^\perp(\cup_j S_j) \to \R^8$ defined by $G_S(x, v) = x + v$ is a diffeomorphism from $\{ (x, v) \in T^\perp(\cup_j S_j) : x \in B_2 \setminus \cup_\alpha B_{r_\alpha/8}(y_\alpha), |v| < 2\eps_S \}$ onto its image, and satisfies $|DG_S|_{(x, t)} - Id| \leq \eps_S^{-1} |t|$.  Note that $\eps_S$ is always positive.
\end{definition}

A smooth region is a multi-graph over $S$, with $C^2$ norm small at the scale of $\eps_S$.
\begin{definition}[Smooth region]
Given a smooth model $S$, $g$ a $C^2$ metric on $B_R(a) \subset \R^8$, and $V$ an integral varifold in $(B_R(a), g)$, we say $V \llcorner (B_R(a), g)$ is a $(S, \beta)$-smooth region if for each $i = 1, \ldots, k$ there are $C^2$ functions $\{ u_{ij} : S_i \to S_i^\perp \}_{j=1}^{m_i}$ so that
\[
\left( (\eta_{a, R})_\sharp V \right) \llcorner B_1 \setminus \cup_\alpha B_{r_\alpha/4}(y_\alpha) = \sum_{i=1}^k \sum_{j=1}^{m_i} [\graph_{S_i}(u_{ij}) \cap B_1 \setminus \cup_\alpha B_{r_\alpha/4}(y_\alpha)]_{g\circ \eta_{a, R}^{-1}},
\]
and
\[
|u_{ij}|_{C^2(S_i)} \leq \beta \eps_S, \quad \forall i, j.
\]
\end{definition}

The ``global'' structure of one of our minimal surfaces/stationary varifolds is captured by a strong-cone decomposition, which is essentially a finite collection of strong-cone regions and smooth regions which fit together to cover all of $B_R(x)$.  The total number of regions is controlled, and the smooth models can only come from some (finite) predefined pool of models $\cS$.  Similar structures have appeared in many other contexts, sometimes called bubble trees (e.g. for harmonic maps \cite{parker}) or neck decompositions (e.g. for limits of spaces with bounded Ricci \cite{naber-jiang}).

\begin{definition}[Cone decomposition]\label{def:decomp}
Take $\theta, \gamma, \beta \in \R$, $\sigma \in (0, 1/3)$, $N \in \N$.  Let $g$ be a $C^2$ metric on $B_R(x) \subset \R^8$, $V$ be an integral varifold in $(B_R(x), g)$, and let $\cS = \{ S_s \}_s$ be a finite collection of $(\theta, \sigma, \gamma)$-smooth models.  A $(\theta, \beta, \cS, N)$-cone decomposition of $V \llcorner (B_R(x), g)$ consists of integers $N_S, N_C$, satisfying $N_C + N_S \leq N$, points $\{x_a\}_a, \{x_b\}_b \subset B_R(x)$, radii $\{R_a \geq 2 \rho_a\}_a$, $\{ R_b\}_b$, indices $\{ s_b \}_b$, integers $\{1 \leq m_a \leq [\theta]\}_a$, and cones $\{\bC_a\}_a \subset \cC$, where $a = 1, \ldots, N_C$, $b = 1, \ldots, N_S$, such that:
\begin{enumerate}
\item Every $V \llcorner (B_{R_b}(x_b), g)$ is a $(S_{s_b}, \beta)$-smooth region, and every $V \llcorner (A_{R_a, \rho_a}(x_a), g)$ is a $(\bC_a, m_a, \beta)$-strong-cone region.

\item There is either a smooth region $B_{R_b}(x_b)$ with $R_b = R$, $x_b = x$, or a strong-cone region $A_{R_a, \rho_a}(x_a)$ with $R_a = R$, $x_a = x$.

\item If $B_{R_b}(x_b)$ is a smooth region, with associated smooth model $(S, \bC, \{(\bC_{\alpha}, B_{r_{\alpha}}(y_{\alpha}))\}_\alpha) \in \cS$, then for each $\alpha$ there we can find a radius $R_{b,\alpha}$ and point $x_{b,\alpha}$ satisfying
\[
|x_{b,\alpha} - y_{\alpha}| \leq \beta R_b r_{\alpha}, \quad \frac{1}{2} \leq \frac{R_{b,\alpha}}{R_b r_{\alpha}} \leq 1+\beta,
\]
so that there is either a strong-cone region $A_{R_a, \rho_a}(x_a)$ with $R_a = R_{b,\alpha}$, $x_a = x_{b,\alpha}$, or another smooth region $B_{R_{b'}}(x_{b'})$ with $R_{b'} = R_{b,\alpha}$, $x_{b'} = x_{b,\alpha}$.

\item If $A_{R_a, \rho_a}(x_a)$ is a $(\bC_a, m_a, \beta)$-strong-cone region, and $\rho_a > 0$, then there is either a smooth region $B_{R_b}(x_b)$ with $R_b = \rho_a$, $x_b = x_a$, or another cone region $A_{R_{a'}, \rho_{a'}}(x_{a'})$ with $R_{a'} = \rho_a$, $x_{a'} = x_a$.  If $\rho_a = 0$, then $\theta_{\bC_a}(0) > 1$.
\end{enumerate}
\end{definition}

Later on we will be primarily interested in cone decompositions for hypersurfaces.  We define them in terms of the associated (multiplicity-one) varifold.
\begin{definition}[Smooth region/Cone decomposition for $M$]
Given a $C^2$ hypersurface $M \subset B_R(a)$, and $g$ a $C^2$ metric on $B_R(a)$, we say $M \cap (B_R(x), g)$ is a $(S, \beta)$-smooth region or admits a $(\theta, \beta, \cS, N)$-cone decomposition if $[M]_g \llcorner (B_R(a), g)$ is a $(S, \beta)$-smooth region or admits a $(\theta, \beta, \cS, N)$-cone decomposition (respectively).
\end{definition}

Recall that Theorem \ref{thm:cones} gives an enumeration of $\{\theta_\bC(0) : \bC \in \cC\} = \{\theta_i\}_i$.  Let us enumerate the set of densities-with-multiplicity as $\{ m \theta_i : m \in \N, i = 0, 1, \ldots \} = \{ \tilde \theta_i \}_i$, for $1 = \tilde\theta_0 < \tilde\theta_1 < \ldots$.  We prove our decomposition theorem $(\dagger_{l, I})$ by induction on both the density-with-multplicity $\tilde\theta_l$ and index $I$.
\begin{theorem}[$\dagger_{l, I}$]\label{thm:diffeo}
Given $l, I \in \N$, $0 < \beta \leq \gamma \leq 1$, $\sigma \in (0, \frac{1}{100(I+1)}]$, there are constants $\delta_{l, I}$, $N$, and a finite collection of $(\tilde \theta_l, \sigma, \beta)$-smooth models $\{S_s\}_s = \cS$, all depending (only) on $(l, I, \gamma, \beta, \sigma)$, so that the following holds.

Let $g$ be a $C^3$ metric on $B_1$ satisfying $|g - \geucl|_{C^3(B_1)} \leq \delta_{l, I}$.  Take $V \in \cISV_7(B_1, g)$ with $\mindex(V, B_1, g) \leq I$, $\bC \in \cC$, $m \in \N$, such that $m\theta_\bC(0) \leq \tilde\theta_l$.  Suppose that
\begin{equation}\label{eqn:diffeo-hyp1}
d_H( \spt V \cap B_1, \bC \cap B_1) \leq \delta_{l, I},
\end{equation}
and
\begin{equation}\label{eqn:diffeo-hyp2}
(m - 1/2)\theta_\bC(0) \leq \theta_V(0, 1/2), \quad \theta_V(0, 1) \leq (m+1/2) \theta_\bC(0).
\end{equation}

Then there is a radius $r \in (1-20 (I+1) \sigma, 1)$, so that $V \llcorner (B_r, g)$ admits a $(\tilde \theta_l, \beta, \cS, N)$-strong-cone decomposition.  If $I = 0$ then one can take $r = 1-5\sigma$.

Additionally, each $S_s$ satisfies $\mindex(S_s) \leq I$.  If $V$ is area-minimizing,\footnote{In the sense that $V$ is the varifold associated to some mass-minimizing current.  This statement is a direct consequence of the nature of the proof, and the compactness for mass-minimizing currents \cite[Chapter 7, Theorem 2.4]{Sim}.} then one can assume all the $S_s$ appearing in its cone decomposition are area-minimizing also.
\end{theorem}


\begin{remark}
We emphasize that $\delta_{l, I}$ depends also on $\gamma, \beta, \sigma$.  We write only the indices $l, I$ as these are the integers we induct over.
\end{remark}

\begin{remark}\label{rem:diffeo-A-bound}
By the curvature estimate \eqref{eqn:planar-A-bound}, if all the cones appearing in the cone decomposition of $V$ are planar, then $\sing V = \emptyset$ and we have the bound
\[
\int_{B_r} |A_{V, g}|^7 d\mu_V(x) \leq c(l, \cS, N).
\]
\end{remark}

\begin{remark}
We will often use Theorem \ref{thm:diffeo} at some different scale $B_r(x)$.  Let us mention explicitly that if $g$ is a $C^3$ metric on $B_r(x) \subset \R^8$ such that $|g - \geucl|_{C^3(B_r(x))} \leq \delta$, then $g \circ \eta_{x, r}^{-1} \equiv g(x + r \cdot )$ is a $C^3$ metric on $B_1$ satisfying $| g \circ \eta_{x, r}^{-1} - \geucl|_{C^3(B_1)} \leq \max\{1, r\} \delta$.
\end{remark}

\begin{remark}
The cone decomposition constructed by Theorem \ref{thm:diffeo} has the additional properties: the outermost region is always a smooth region, and no two cone regions are ``stacked'' directly inside each other (so you cannot have large annulus consisting of numerous cone regions where $V$ looks like a rotating cone).
\end{remark}

\begin{proof}
Fix $l, I$.  We shall prove $(\dagger_{l, I})$ by induction and by contradiction.  We shall choose our constants as we go along, but it may help the reader to keep in mind:
\[
\beta'' << \tau << \beta' << \beta \leq \gamma << \sigma < 1. 
\]
Note there is no loss in assuming $\beta(l, \sigma)$ is as small as we like.

Suppose $(\dagger_{l, I})$ fails.  Then there are sequences $\delta_i \to 0$, $C^3$ metrics $g_i$, $V_i \in \cISV_7(B_1, g_i)$, $\bC_i \in \cC$, $m_i \in \N$, so that $g_i$, $V_i$ satisfies the hypotheses of $(\dagger_{l, I})$ with $\delta_i$, $g_i$, $m_i$, $\bC_i$ in place of $\delta, g, m, \bC$, but with the property that for any finite collection $\cS'$ of $(\tilde\theta_l, \sigma, \beta)$-smooth models, and any $N' \in \N$, there is an $i_0$ such that $V \llcorner (B_r, g_i)$ does not admit a $(\tilde\theta_l, \beta, \cS', N')$-strong-cone decomposition for all $i > i_0$, $r \in (1-20 (I+1) \sigma, 1)$.

Passing to a subsequence, by Theorem \ref{thm:cones} we can assume that $\bC_i \to \bC \in \cC$ smoothly with multiplicity-one away from $0$, and $\theta_{\bC_i}(0) = \theta_\bC(0)$ for all $i$.  By Lemmas \ref{lem:index-compact}, \ref{lem:isv-decomp}, and our hypotheses, passing to a further subsequence, there are at most $I$ points $\cI \subset \bC \cap B_1$ so that $V_i \to m[\bC]$ as varifolds in $B_1$, in $C^2$ on compact subsets of $\cI \cup \sing \bC$.  Note that since $m \theta_\bC(0) \leq \tilde \theta_l$, we have $\max\{ m, \theta_\bC(0) \} \leq \tilde \theta_l$.

We first note that if $l = 0$, corresponding to $m[\bC] = [\bC]$ being a multiplicity-one plane, and $I$ is arbitrary, then provided $i >> 1$ \cite{All} proves implies each $V_i \llcorner (B_{1-5\sigma}, g_i)$ is a $(S, \beta)$-smooth region, for $S$ being the smooth model $([\bC], \bC, \emptyset)$.  Similarly, if $\bC$ is planar, $I = 0$, and $l = m$ is arbitrary, then since $\reg V$ is stable in $(B_1, g_i)$ \cite{ScSi} implies $V_i \llcorner (B_{1-5\sigma}, g_i)$ is a $(S, \beta)$-smooth region for $i >> 1$, where $S$ is the smooth model $(m[\bC], \bC, \emptyset)$.  In either case we get that $V_i \llcorner (B_{1-5\sigma}, g_i)$ admits a $(\tilde \theta_l, \beta, \{S\}, 1)$-cone decomposition, which is a contradiction.

By our inductive hypothesis we can therefore assume that $(\dagger_{l', I})$ holds for every $l' < l$, $I$ arbitrary, and also that $(\dagger_{l, I'})$ holds for any $I' < I$.

\vspace{3mm}

We break into three cases.

\textbf{Case 1: $\cI \subset \{0\}$.} For each $i$, let $\rho_i$ be the smallest number such that $V_i \llcorner (A_{1, \rho_i}(0), g_i)$ is a $(\bC, m, \beta'', \tau, \sigma)$-weak-cone region, for $\beta'', \tau$ to be chosen later depending only on $(l, \beta, \gamma)$.  Let $a_i = a_{\rho_i}(V_i)$ be the annulus center at radius $\rho_i$ (as in Definition \ref{def:wcone}).  By our varifold convergence $V_i \to m[\bC]$, and $C^2$ convergence on compact subsets of $B_1 \setminus \{0\}$, we have $a_i \to 0$ and $\rho_i \to 0$.

Ensuring $\beta''(l, \beta')$, $\tau(l, \beta', \sigma)$ are sufficiently small, Lemma \ref{lem:recenter} implies each $V_i \llcorner (A_{1-3\sigma, \rho_i}(a_i), g_i)$ is a $(\bC, m, \beta')$-cone region.  Ensuring $\beta'(l, \beta)$ is sufficiently small, and $i >> 1$, we can apply Theorem \ref{thm:scone} to deduce that each $V_i \llcorner (A_{1-3\sigma, \rho_i}(a_i), g_i)$ is a $(\bC, m, \beta)$-strong-cone region.

First suppose $\rho_i = 0$ for infinitely-many $i$.  For each such $i >> 1$, the maximum principle of \cite{ilmanen} implies $V_i \llcorner B_{1-4\sigma} = m[M_i]_{g_i}$ for some $M_i \in \cM_7(B_{1-4\sigma}, g_i)$.  If $\theta_\bC(0) = 1$, corresponding to $m[\bC]$ being a multiplicity-$m$ plane, then $[M_i]_{g_i} \to [\bC]$ in $B_{1-4\sigma}$ and so by \cite{All} we deduce $V_i \llcorner (B_{1-5\sigma}, g_i)$ is a $(S, \beta)$-smooth region for $S$ the smooth model $(m[\bC], \bC, \emptyset)$.  So infinitely-many $V_i \llcorner (B_{1-5\sigma}, g_i)$ admit a $(\tilde \theta_l, \beta, \{S\}, 1)$-cone decomposition, which is a contradiction.

If $\theta_\bC(0) > 1$, define the smooth model $(S = [\bC], \bC, \{(\bC, B_{1/10})\})$, and let $r = 1-5\sigma$.  By $C^2$ convergence, $V_i \llcorner (B_r, g_i)$ is a $(S, \beta)$-smooth region for $i >> 1$.  Since $|a_i| \leq \beta r/10$ and $V_i \llcorner (A_{r/10, \rho_i}(a_i), g_i)$ is a $(\bC, m, \beta)$-strong-cone region for every $i >> 1$, we deduce that infinitely-many $V_i \llcorner (B_r, g_i)$ admit a $(\tilde \theta_l, \beta, \{S\}, 2)$-strong-cone decomposition.  This is a contradiction.

So we must have $\rho_i > 0$ for all $i >> 1$.  Define the rescaled varifolds $V_i' = (\eta_{a_i, \rho_i})_\sharp V_i$, and rescaled metrics $g_i' = g_i \circ \eta_{a_i, \rho_i}^{-1}$.  Then for any $R > 2$ and $i >> 1$, $V_i' \in \cISV_7(B_R(0), g_i')$, and by Lemma \ref{lem:recenter}(\ref{item:recenter}) and scale-invariance (Remark \ref{rem:scaling}),  we have that
\begin{equation}\label{eqn:inf-cone}
1 = \inf \{ \rho : V_i' \llcorner (A_{R, \rho}(0), g_i') \text{ is a $(\bC, m, \beta'', \tau, \sigma)$-weak-cone region} \}.
\end{equation}

Of course we also have that $a_1(V_i') = 0$, this being the center of the annular region of $V_i'$ at radius $1$.  Moreover, for any $R > 2$ and $i >> 1$, each $V_i' \llcorner (A_{R, 1}, g_i')$ is a $(\bC, m, \beta)$-strong-cone region.  Finally, note that we continue to have $C^3$ convergence $g_i' \to \geucl$ on compact subsets of $\R^8$.

Monotonicity \eqref{eqn:monotonicity} and the fact $a_i \to 0$ imply that $\theta_{V_i'}(0, R) \leq (m + o(1)) \theta_\bC(0)$ for any particular $R > 0$, where here $o(1) \to 0$ as $i \to \infty$.  Trivially we also have $\mindex(V_i', B_R, g_i') \leq I$ for every $R > 0$.  Therefore, by Lemma \ref{lem:index-compact}, we can find an $V' \in \cISV_7(\R^8, \geucl)$ with $\theta_{V'}(0, \infty) \leq m \theta_\bC(0)$, and at most $I$ points $\cI' \subset \spt V'$, so that $V_i' \to V'$ as varifolds in $\R^8$ and in $C^2$ on compact subsets of $\R^8 \setminus (\sing V' \cup \cI')$.  Provided $\beta(l)$ is sufficiently small, Arzela-Ascoli, Theorem \ref{thm:cones}, and \eqref{eqn:higher-reg} imply that $V' \llcorner (A_{\infty, 1}, \geucl)$ is a $(\bC, m, \beta)$-strong-cone region, and $V_i' \to V'$ in $C^2$ on compact subsets of $\R^8 \setminus \overline{B_{1/8}}$.  In particular $\sing V' \cup \cI' \subset \overline{B_{1/8}}$ is a finite set.

I claim that any tangent cone to $V'$ at infinity takes the form $m[\bC']$ for some $\bC' \in\cC(\bC)$.  To see this, take any $R_i \to \infty$, and suppose $(\eta_{0, R_i})_\sharp V' \to V'' \in \cISV_7(\R^8)$ as varifolds.  Since (by Lemma \ref{lem:index-compact}) $\mindex(V'', \R^8, \geucl) < \infty$ and $V''$ is a cone, $V''$ is stable and hence $V'' = m''[\bC'']$ for some $\bC'' \in \cC$ and $m'' \in \N$ satisfying $m''\theta_{\bC''}(0) \leq m \theta_\bC(0) \leq \tilde \theta_l$.  We can assume that $m''[\bC''] \llcorner (A_{\infty, 0}, \geucl)$ is a $(\bC, m, \beta)$-strong-cone region, and hence
\[
d_H(\bC'' \cap B_1, \bC \cap B_1) \leq \beta, \quad m'' \theta_{\bC''}(0) \geq m \theta_\bC(0) - \beta.
\]
Ensuring $\beta(l)$ is sufficiently small, Theorem \ref{thm:cones} implies $\bC'' \in \cC(\bC)$, and hence $\theta_{\bC''}(0) = \theta_\bC(0)$ and $m'' = m$.  This proves my claim.

We break into two subcases.

\textit{Subcase 1A: $\theta_{V'}(a) \geq \tilde\theta_l$ for some $a \in \spt V'$.}  In this case, by the monotonicity formula we must have $V' = m [a + \bC']$ for some $\bC' \in \cC(\bC)$.  If $\bC$ is planar, there is no loss in assuming $a \in \bC^\perp$.

Since each $V_i' \llcorner (A_{4, 1}(0), g_i')$ is a $(\bC, m, \beta'', \tau, \sigma)$-weak-cone region, for every $r \in [1, 4]$ there are $a_{r,i} \equiv a_r(V_i') \in \R^8$, $\bC_{r, i} \in \cC(\bC)$, and $C^2$ functions $\{u_{r, i, j} : (a_{r, i} + \bC_{r, i}) \cap A_{r, r/8}(a_{r, i}) \to \bC_{r, i}^\perp \}_{j=1}^m$ such that
\begin{equation}\label{eqn:shrink3}
V_i' \llcorner A_{r, r/8}(a_{r, i}) = \sum_{j=1}^m [G_{a_{r, i} + \bC_{r, i}}(u_{r, i}) \cap A_{r, r/8}(a_{r, i})]_{g_i'},
\end{equation}
and
\begin{equation}\label{eqn:shrink4}
r^{-1} |u_{r, i, j}| + | \nabla u_{r, i, j}| + r|\nabla^2 u_{r,i,j}| \leq \beta'' \quad \forall i, j, r.
\end{equation}
Moreover, by construction we have $a_{1, i} = 0$.

It follows by \eqref{eqn:shrink3}, \eqref{eqn:shrink4}, Ahlfors regularity \eqref{eqn:ahlfors}, and varifold convergence $V_i' \to [a + \bC']$, that for every $r \in [1, 4]$ and $i >> 1$, we have
\[
d_H( (a + \bC') \cap A_{r/2, r/4}(a_{r, i}), (a_{r,i} + \bC_{r, i}) \cap A_{r/2, r/4}(a_{r,i})) \leq 8 \beta''.
\]
(Note the $\beta''$ instead of $\beta'$!)  By Lemma \ref{lem:nearby-cones} (applied in the ball $B_{r/2}(a_{r, i})$), ensuring $\beta''(l, \tau)$ is sufficiently small, we deduce that
\begin{equation}\label{eqn:nearby-0}
|a| \leq 2^{-10} \tau \leq 2^{-4}, \quad |a - a_{r,i}| \leq 2^{-10} \tau r , \quad d_H(\bC', \bC_{r,i}) \leq 2^{-10} \tau,
\end{equation}
for each $r \in [1, 4]$, and all $i >> 1$.

I claim that $\cI' \not\subset \{a\}$ if $\theta_{\bC'}(0) > 1$, and $\# \cI' \geq 2$ if $\theta_{\bC'}(0) = 1$.  Let us prove this.  If $\theta_{\bC}(0) > 1$ and $\cI' \subset \{a\}$ then set $a' = a$ and $r' = 1/2$.  If, on the other hand, $\theta_\bC(0) = 1$ and $\cI' = \{y'\} \subset \overline{B_{1/8}} \cap (a + \bC')$, then set $a' = a + 2^{-4}\tau (y' - a)$, $r' = 1-2^{-5}\tau < 1$.  In this case we get
\begin{equation}\label{eqn:nearby}
|a' - y'| < r'/8, \quad |a' - a| \leq 2^{-6} \tau, \quad |a' - a_{r,i}| \leq 2^{-5} \tau .
\end{equation}
Either way $a' + \bC' = a + \bC'$.

In either case, by assumption and our choice of $a', r'$, we can find for all $i >> 1$ $C^2$ functions $\{ u_{i,j}' : (a + \bC') \cap A_{1, r'/8}(a') \to \bC'^\perp \}_{j=1}^m$ with the properties that
\begin{equation}\label{eqn:shrink1}
V_i' \llcorner A_{1, r'/8}(a') = \sum_{j=1}^m [G_{a' + \bC'}(u_{i,j}') \cap A_{1, r'/8}(a')]_{g_i'} ,
\end{equation}
and
\begin{equation}\label{eqn:shrink2}
|u_{i,j}'| + |\nabla u_{i,j}'| + |\nabla^2 u_{i,j}'| \leq \beta''/2.
\end{equation}
On the other hand, since $V_i' \to [a' + \bC']$ as varifolds and $\theta_{\bC'}(0) = \theta_\bC(0)$, we have $|\theta_{V_i'}(a', r) - m \theta_{\bC}(0)| \leq \beta''$ for all $i >> 1$ and $r \in [1/2, 4]$.  This, together with \eqref{eqn:shrink3}, \eqref{eqn:shrink4}, \eqref{eqn:nearby}, \eqref{eqn:shrink1}, \eqref{eqn:shrink2}, implies that $V_i' \llcorner (A_{4, r'}(0), g_i')$ is a $(\bC, m, \beta'', \tau, \sigma)$-weak-cone region, contradicting \eqref{eqn:inf-cone}.  This proves my claim.

If $\theta_{\bC'}(0) > 1$, we have $\cI' \not\subset \{a\}$, and $\cI' \subset B_{1/4}(a)$, and so we can apply Case 2 of $(\dagger_{l, I})$ to the translated/dilated sequences $\eta_{a, 1/4}(V_i')$, $g_i' \circ \eta_{a, 1/4}^{-1}$.  We obtain a finite collection $\cS'$ of $(\tilde\theta_l, \sigma, \beta)$-models, an integer $N'$, and a radii $r''_i \in (\frac{4}{5} \cdot\frac{1}{4}, \frac{1}{4})$ (recall our assumption on $\sigma$), so that every $V_i' \llcorner (B_{r''_i}(a), g_i')$ (for $i >> 1$) admits a $(\tilde\theta_l, \beta, \cS', N')$-strong-cone decomposition.

If $\theta_{\bC'}(0) = 1$, we have $\#\cI' \geq 2$ and $\cI' \subset B_{1/4}(a)$, and so we can apply Case 3 of $(\dagger_{l, I})$ as in the previous paragraph to deduce the same conclusion.

By Lemma \ref{lem:regraph} and \eqref{eqn:nearby-0}, for $\tau(\gamma, l)$ sufficiently small, we can write 
\[
(a + \bC') \cap A_{\infty, 1/2} = G_{\bC'}(u') \cap A_{\infty, 1/2}
\]
for $u' : \bC' \cap A_{\infty,1/2} \to \bC'^\perp$ a $C^2$ function satisfying $|x|^{-1} |u'| + |\nabla u'| + |x| |\nabla^2 u'| \leq \gamma$.  Letting $S' = m [ a + \bC']$, then this and our requirement $|a| \leq 1/16$ imply $(S', \bC, \{ (\bC, B_{1/4}(a))\})$ is a $(\tilde \theta_l, \sigma, \gamma)$-smooth model.  $C^2$ convergence to $m[a + \bC']$ means that for $i >> 1$, every $V_i' \llcorner (B_1, g_i')$ is a $(S', \beta)$-smooth region.

Combining the above two paragraphs, we get that every $V_i' \llcorner (B_1, g_i')$, and hence $V_i \llcorner (B_{\rho_i}(a_i), g_i)$, admits a $(\tilde \theta_l, \beta, \{ S' \} \cup \cS', 1 + N')$-strong-cone decomposition.  If we let $S$ be the smooth model as constructed in the $\rho_i = 0$ case, then like we argued before we deduce that $V_i \llcorner (B_{1-5\sigma}(0), g_i)$ admits a $(\tilde\theta_l, \beta, \{S, S'\} \cup \cS', 3 + N')$-strong-cone decomposition.  This is a contradiction, and completes the proof of Case 1A.

\textit{Subcase 1B: $\theta_{V'}(x) \leq \tilde\theta_{l-1}$ for all $x \in \spt V'.$}  Write $\sing V' \cup \cI' = \{y_1', \ldots, y_{d'} '\}$.  By Lemma \ref{lem:rs-bound} and our subcase hypothesis, for each $\alpha = 1, \ldots, d'$ there is a tangent cone to $V'$ at $y_\alpha'$ of the form $m_\alpha [\bC_\alpha']$, where $\bC_\alpha' \in \cC$ and $m_\alpha \theta_{\bC_\alpha'}(0) \leq \tilde\theta_{l-1}$.  By Lemmas \ref{lem:rs-bound}, \ref{lem:index-compact}, Theorem \ref{thm:scone}, Remark \ref{rem:scone-stable}, and the maximum principle of \cite{ilmanen}, we can choose radii $r_\alpha$ so that the following holds: 
\begin{enumerate}
\item the collection $\{B_{2r_\alpha}(y_\alpha') \}_\alpha \subset B_{1/2}(0)$ are disjiont;
\item for each $\alpha$, $\spt V' \cap (A_{2r_\alpha, 0}(y_\alpha), \geucl)$ is a $(\bC_\alpha', 1, \gamma)$-strong-cone region;
\item for each $\alpha$, we have
\[
d_H(\spt V' \cap B_{r_\alpha}(y_\alpha'), (y_\alpha' + \bC_\alpha') \cap B_{r_\alpha}(y_\alpha') \leq \delta_{l-1, I}/2,
\]
and
\[
(m_\alpha - 1/4) \theta_{\bC_\alpha'}(0) \leq \theta_{V'}(y_\alpha, r_\alpha/2), \quad \theta_{V'}(y_\alpha, r_\alpha) \leq (m + 1/4) \theta_{\bC_\alpha'}(0) 
\]
\end{enumerate}
Here $\delta_{l-1, I} = \delta_{l-1, I}(l-1, I, \beta, \sigma)$ is the constant from $(\dagger_{l-1, I})$.

For each $\alpha$, we can therefore apply $(\dagger_{l-1, I})$ to the sequences $(\eta_{y_\alpha, r_\alpha})_\sharp V_i'$, $g_i' \circ \eta_{y_\alpha, r_\alpha}^{-1}$.  We obtain a finite collection $\cS'$ of $(\tilde\theta_l, \sigma, \beta)$-smooth models, an integer $N'$, and radii $r_{\alpha,i}' \in (\frac{4}{5} r_\alpha, r_\alpha)$, so that each $V_i' \llcorner (B_{r_{\alpha,i}'}(y_\alpha'), g_i')$ admits a $(\tilde\theta_l, \beta, \cS', N')$-strong-cone decomposition, for $i >> 1$.

Recall that $V' \llcorner (A_{\infty, 1}, \geucl)$ is a $(\bC, m, \beta)$-strong-cone region, and $\beta \leq \gamma$.  By our choice of $\alpha$, we get that $(S = V', \bC', \{ (\bC_\alpha', B_{r_\alpha}(y_\alpha')) \}_\alpha)$ is a $(\tilde\theta_l, \sigma, \gamma)$-smooth model.  $C^2$ convergence implies $V_i' \llcorner (B_1, g_i')$ is a $(S, \beta)$-smooth region provided $i >> 1$.

It follows that each $V_i' \llcorner (B_1, g_i')$ admits a $(\tilde\theta_l, \beta, \{ S \} \cup \cS', 1 + N')$-strong-cone decomposition.  The proof of Case 1B now continues as in Case 1A.


\textbf{Case 2: $\cI \not\subset \{0\}$ and $\theta_\bC(0) > 1$.}  Let $\cI \setminus \{0\} = \{y_1, \ldots, y_{d'} \}$, which is non-empty by assumption.  We can choose a radius $r \in (1-20\sigma (I+1), 1-5\sigma)$ with the property that $\max \{ ||y_\alpha| - r|| : y_\alpha \in \cI \} \geq 5\sigma$.

Set $r_0 = \min \{ |y_\alpha| : y_\alpha \in \cI \}/4$, and $y_0 = 0$, and then for each $\alpha = 1, \ldots, d'$ with $|y_\alpha| < r$, choose a radius $r_\alpha$ so that:
\begin{enumerate}
\item the collection $\{B_{2r_\alpha}(y_\alpha)\}_\alpha \subset A_{r-4\sigma, 2r_0}(0)$ are disjoint;
\item for each $\alpha$, if we set $\bC_\alpha = T_{y_\alpha} \bC$ to be the tangent plane of $\bC$, then $\bC \cap (A_{2r_\alpha, 0}(y_\alpha), \geucl)$ is a $(\bC_\alpha, 1, \gamma)$-strong-cone region;
\item for each $\alpha$, we have 
\[
d_H( \bC \cap B_{r_\alpha}(y_\alpha) , (y_\alpha + \bC_\alpha) \cap B_{r_\alpha}(y_\alpha)) \leq \delta_{l-1, I}/2, 
\]
and
\[
m - 1/4 \leq \theta_{\bC}(y_\alpha, r_\alpha/2), \quad \theta_\bC(y_\alpha, r_\alpha) \leq m+1/4.
\]
\end{enumerate}
Note that every $m \theta_{\bC_\alpha}(0) = m < m \theta_\bC(0)$, and so in fact $m \theta_{\bC_\alpha}(0) \leq \tilde\theta_{l-1}$.

For $i >> 1$ we can apply $(\dagger_{l-1, I})$ to find a finite collection $\cS'$ of $(\tilde\theta_l, \sigma, \beta)$-smooth models, an integer $N'$, and radii $r_{\alpha,i}' \in (\frac{4}{5} r_\alpha, r_\alpha)$, so that every $V_i \llcorner (B_{r'_{\alpha,i}}(y_\alpha), g_i)$ admits a $(\tilde\theta_l, \beta, \cS', N')$-strong-cone decomposition.

On the other hand, since $\mindex(V_i, B_{r_0}(0), g_i) \leq I-1$ for all $i >> 1$, we can apply $(\dagger_{l, I-1})$ to find another finite collection $\cS''$, an integer $N'' \in \N$, and radii $r'_{0,i} \in (\frac{4}{5} r_0, r_0)$, so that each $V_i \llcorner (B_{r'_{0,i}}(0), g_i)$ admits a $(\tilde\theta_l, \beta, \cS'', N'')$-strong-cone decomposition.

Now observe that $(m[\bC], \bC, \{ (\bC_\alpha, B_{\frac{r_\alpha}{r}}(\frac{y_\alpha}{r})) \}_{\alpha=0}^{d'})$ is a $(m\theta_\bC(0), \sigma, \gamma)$-smooth model, and each $V_i \llcorner (B_r, g_i)$ is a $(S, \beta)$-smooth region for $i >> 1$.

Combining all of the above, if we take $\cS = \{ S \} \cup \cS' \cup \cS''$, $N = I N' + N'' + 1$, then every $V_i \llcorner (B_r, g_i)$ admits a $(\tilde\theta_l, \beta, \cS, N)$-strong-cone decomposition.  This is a contradiction, and finishes the proof of Case 2.


\textbf{Case 3: $\cI \not\subset \{0\}$ and $\theta_\bC(0) = 1$.} In this case, since $\sing \bC = \emptyset$, $V_i \to m[\bC]$ in $C^2$ on compact subsets of $B_1 \setminus \cI$.  Pick a radius $r \in (1-20\sigma (I+1), 1-5\sigma)$ with the property that $\max \{ ||y_\alpha| - r|| : y_\alpha \in \cI \} \geq 5\sigma$.  Let $\cI' = \cI \cap B_r = \{1, \ldots, d'\}$.

For each $y_\alpha \in \cI'$, choose a radius $r_\alpha$ so that all the $\{B_{2r_\alpha}(y_\alpha)\}_\alpha \subset B_{r-4\sigma}$ are disjoint.  Define the rescaled varifolds $V^\alpha_i =  (\eta_{y_\alpha, r_\alpha})_\sharp V_i$ and metrics $g_i^\alpha = g_i \circ \eta_{y_\alpha, r_\alpha}^{-1}$.  Then for each $\alpha, i$ we have $\mindex(V^\alpha_i, B_1, g_i^\alpha) \leq I$, and $V^\alpha_i$ converge to $m[\bC]$ as varifolds in $B_1$, and converge to $\bC$ in $C^2$ on compact subsets of $B_1 \setminus \{0\}$.  We can therefore apply Case 1 of $(\dagger_{l, I})$ to each sequence to deduce the existence of a finite collection $\cS'$ of $(\tilde\theta_l, \sigma, \beta)$-smooth models, and an integer $N'$, and radii $r'_{\alpha,i} \in (\frac{4}{5}r_\alpha, r_\alpha)$, so that every $V_i \llcorner (B_{r'_{\alpha,i}}(y_\alpha), g_i)$ admits a $(\tilde\theta_l, \beta, \cS', N')$-strong-decomposition whenever $i >> 1$.

An important note: if $\# \cI \geq 2$, then by Lemma \ref{lem:index-compact} and our choice of balls $B_{r_\alpha}(y_\alpha)$, we have $\mindex(V^\alpha_i, B_1, g_i^\alpha) \leq I-1$ for each $\alpha = 1, \ldots, d'$ and $i >> 1$.  We can therefore apply Case 1 of $(\dagger_{I-1, l})$ to the sequence $V^\alpha_i$ rather than Case 1 of $(\dagger_{l, I})$, to deduce the same result.

Now $(S = m[\bC], \bC, \{ (\bC_\alpha = \bC, B_{\frac{r_\alpha}{r}}(\frac{y_\alpha}{r}))\}_{\alpha=1}^{d'})$ is a $(m, \sigma, 0)$-smooth model.  By $C^2$ convergence of $V_i$ to $m[\bC]$ away from $\cI$, and our choice of $r$, the $V_i \llcorner (B_r, g_i)$ are $(S, \beta)$-smooth regions for $i >> 1$.

Therefore, we deduce that $V_i \llcorner (B_r(0), g_i)$ admits a $(\tilde\theta_l, \beta, \{ S\} \cup \cS', I N' + 1)$-strong-cone decomposition for all $i >> 1$.  This is a contradiction, and completes the proof of Case 3. \qedhere

\end{proof}

\section{Parameterization}\label{sec:param}

In this section we construct ``parameterizations'' for our cone regions.  We show that every $M$ admitting a cone-decomposition can be realized as the image under some Lipschitz map of one of finitely-many $\{M_v\}_v$, depending only on the parameters of the decomposition.  In this section we work in $\R^8$.

This section is annoying technical, due almost entirely to multiplicity.  If all our $M$s were single-sheeted (e.g. as when $M$ is an area-minimizing boundary), then all our parameterizing maps would be $C^1$ perturbations of the identity, and this section would be more or less trivial.  Unfortunately, maps $\phi$ taking multi-sheeted $M_1$ to $M_2$ are only $C^0$ perturbations of the identity, with the $C^1$ norm depending on the relative gaps between in $M_1$ versus $M_2$.  So to glue various parameterizations together we cannot use ``soft'' methods, but must rely much more explicitly on the exact structure of these maps.

We cannot get a priori control on $D\phi^{-1}$ in Theorem \ref{thm:param} for two reasons.  Intuitively: the first is that within a ``class'' of $M$ having the same kind of cone decomposition, the inner radius of cone regions can be made arbitrarily small, but the outer radius has an upper bound; the second is that sheets can be made arbitrarily close together, but not arbitrarily far apart.  For an example of the first behavior, consider scalings of a leaf of the Hardt-Simon foliation associated to an area-minimizing hypercone.  For an example of the second, simply consider two planes converging to each other.

Getting global regularity of $\phi$ better than Lipschitz seems to be more or less equivalent to knowing that all the cones you see are integrable through rotations.  Even knowing this, getting better a priori bounds on $\phi$ requires sharper decay estimates in Theorem \ref{thm:scone}.  In certain circumstances this is known.  See Remark \ref{rem:c1alpha}.

\vspace{3mm}

It will be convenient to introduce some further notation.  For each $\bC \in \cC_\Lambda$ fix a choice of unit normal $\nu_\bC(x)$, and write $B^*_\eps(\bC) = \{ x \in \R^8 : \mathrm{dist}(x, \bC) < \eps |x| \}$ for the conical $\eps$-neighborhood around $\bC$.  Let us abuse notation slightly, and define the function $G_\bC : \bC \times \R \to \R^8$ by
\[
G_\bC(x, t) = \frac{ x + t \nu_\bC(x)}{|x + t\nu_\bC(x)|} |x|.
\]
By Theorem \ref{thm:cones}, after shrinking $\eps_0(\Lambda)$ as necessary, if $U_\bC = \{ (x, t) \in \bC \times \R : |t| < 2\eps_0 |x| \}$, then $G_\bC$ restricts to a smooth diffeomorphism from $U_\bC$ onto its image, and we have $|G_\bC(x, t)| = |x|$, and
\begin{gather}
|G_\bC(x, t) - (x + t \nu_\bC(x))| \leq |t|^2 , \quad |DG_\bC|_{(x, t)} - Id| \leq |t|/\eps_0, \label{eqn:GC1}\\
(1-|x|^{-2} t^2) |t| \leq d(G_\bC(x, t), \bC) \leq |t|, \label{eqn:GC2}
\end{gather}
for all $(x, t) \in U_\bC$.  In the second inequality we identify $T_x\bC \times \R$ with $\R^8$ via $(v, t) \leftrightarrow v + t \nu_\bC(x)$.

Let $X_\bC$ be the vector field in $G(U_\bC) \supset B^*_{\eps_0}(\bC)$ defined by $X_\bC = DG_\bC(\del_t)$, so that $y(t) = G_\bC(x, t)$ is a solution to the ODE: $\dot y(t) = X_\bC(y(t))$, $y(0) = x$.  After shrinking $\eps_0(\Lambda)$ possibly further we can assume that for $x \in B_{\eps_0}^*(\bC)$, we have
\begin{gather}\label{eqn:X-prop}
X_\bC(x) \perp x, \quad 0.99 \leq |X_\bC| \leq 1, \quad |x||DX_\bC| + |x|^2|D^2 X_\bC| \leq 1/\eps_0.
\end{gather} 
Let us henceforth fix $\eps_0(\Lambda)$ so that the previous discussion holds, for any $\bC \in \cC_\Lambda$.

Given a smooth vector field $X$ defined on some open set $U$, we let $F_X(x, t)$ be the flow of $X$, i.e. the map defined by $\del_t F_X(x, t) = X(F_X(x, t))$, $F_X(x, 0) = x$.  We say a mapping $\phi : U \to U$ preserves the flowlines of $X$ if $\phi(x) = F_X(x, t_\phi(x))$ for some real-valued function $t_\phi$.

Given $R > \rho > 0$, $a \in \R^8$, $\eps \leq \eps_0$, we say $\phi|_{A_{R, \rho}(a)}$ is a $(\bC, \eps)$-map if $\phi$ is a $C^1$ diffeomorphism $A_{R, \rho}(a) \to A_{R, \rho}(a)$, $\phi = id$ outside $a + B^*_{\eps}(\bC)$, and $\phi$ preserves the flowlines of $X_\bC(\cdot - a)$.

\vspace{3mm}

Below is our main theorem of this section.  We remind the reader that all distances, derivatives, inner products, etc. are taken with respect to $\geucl$ unless explicitly stated otherwise.
\begin{theorem}[Parameterizing cone decompositions]\label{thm:param}
Given $\Lambda , \eps > 0, \sigma \in (0, 10^{-2}]$, there are constants $\beta(\Lambda, \sigma, \eps)$, $\gamma(\Lambda, \sigma, \eps)$ so that the following holds.  Take $\cS$ be any finite collection of $(\Lambda, \sigma, \gamma)$-smooth models, $N \in \N$, and $\cG$ any set of $C^3$ metrics on $B_1$.  Then we can find finite collections $\{g_v\}_v \subset \cG$, $\{ M_v \in \cM_7(B_1, g_v) \}_v$, $\{\bC_v\}_v \subset \cC_\Lambda$, and an increasing function $C_{\cS, N, \cG} : [0, 1) \to \R$, with the properties
\[
\sing M_v \subset B_{1-\sigma}, \quad d_H( M_v \cap A_{1, 1-\sigma}, \bC_v \cap A_{1, 1-\sigma}) \leq \eps, 
\]
and so that: given any $g \in \cG$ and $M \in \cM_7(B_1, g)$ with $M \cap B_1$ admitting a $(\Lambda, \beta, \cS, N)$-strong-cone decomposition, then we can find a $v$ and a local bi-Lipschitz map $\phi : B_1 \to B_1$ such that the following holds true:
\begin{enumerate}
\item $\phi(\overline{M_v} \cap B_1) = \overline{M} \cap B_1$, $\phi(\sing M_v \cap B_1) = \sing M \cap B_1$; \label{item:param1}
\item $\phi|_{B_1 \setminus \sing M_v}$ is a $C^2$ diffeomorphism onto its image; \label{item:param2}
\item $\Lip(\phi|_{B_r}) \leq C_{\cS, N, \cG}(r)$ for any $r < 1$; \label{item:param3}
\item $\phi|_{A_{1,1-\sigma}}$ is a $(\bC_v, \eps)$-map; \label{item:param4}
\item $\sing M \subset B_{1-\sigma}$ and $d_H( M \cap A_{1, 1-\sigma}, \bC_v \cap A_{1, 1-\sigma}) \leq \eps$. \label{item:param5}
\end{enumerate}
If $M$ is area-minimizing, then one can assume that the $M_v$ is area-minimizing also.
\end{theorem}


\begin{remark}\label{rem:c1alpha}
If all the singularities of $M$ are modelled on cones of a certain type, for example quadratic cones or other special cones being ``integrable through rotations,'' then one can arrange $\phi$ in Theorem \ref{thm:param} to be globally $C^{1,\alpha}$, though without any a priori $C^{1,\alpha}$ estimates.  When $\bC$ is planar (as in \cite{white:compact}), or when $\bC$ is a minimizing quadratic cone in Euclidean space (as in \cite{me-luca}), then I would expect one could get $C^{1,\alpha}$ estimates on the $\phi$.  See Remark \ref{rem:cone-c1alpha}.  
\end{remark}

\begin{remark}
The restriction $r < 1$ in Item \ref{item:param3} is purely technical.  The reason is that the Lipschitz norm of $\phi$ depends on distance between sheets in each $M_v$.  In a cone decomposition this distance is bounded away from zero inside $B_1$, but could tend to $0$ as you approach $\del B_1$.
\end{remark}

Before proving Theorem \ref{thm:param} we require several technical Lemmas.  The first two construct parameterizations for a single smooth or cone region.  The third is a gluing Lemma, that allows us to glue parameterizations between two regions.

\begin{lemma}[Smooth region parameterization]\label{lem:smooth-map}
Given $\Lambda, \eps > 0$, $\sigma \in (0, 10^{-2}]$, there are $\gamma(\Lambda, \sigma, \eps)$, $\beta(\Lambda, \eps)$ so that the following holds.  Let $m \in \N$, and $(S, \bC, \{(\bC_\alpha, B_{r_\alpha}(y_\alpha))\}_\alpha)$ be a $(\Lambda, \sigma, \gamma)$-smooth model.  Let $M_1, M_2 \in \cM_7(B_1, g)$ for any $C^2$ metric $g$, and suppose that each $M_l \cap (B_1, g)$ is a $(S, \beta)$-smooth region, $l = 1, 2$.

Write $O = B_1 \setminus \cup_\alpha B_{r_\alpha/4}(y_\alpha)$.  Then we can find a $C^2$ diffeomorphism $\phi : O \to O$ such that
\begin{gather}
\phi(M_1) = M_2, \quad \phi|_{A_{1, 1-\sigma}} \text{ is a $(\bC, \eps)$-map, } \\
\phi|_{A_{r_\alpha, r_\alpha/4}(y_\alpha)} \text{ is a $(\bC_\alpha, \eps)$-map, for each $\alpha$,}
\end{gather}
and
\begin{gather}
|D\phi|_{C^0(O \cap B_r)} \leq c(S, M_1, r) < \infty \quad \forall r < 1.
\end{gather}
\end{lemma}

\begin{proof}
Let $\eps' = \eps'' \min\{ \eps_S, \eps/32 \}$ for some constant $\eps''(\Lambda) \leq  \eps_0(\Lambda)/32$ to be determined later, and ensure $\gamma \leq \min\{\eps_0, \eps\}/32$.  Let $O' = B_1 \setminus \cup_\alpha B_{r_\alpha/5}(y_\alpha)$.

Write $S \llcorner B_2 \setminus \cup_\alpha B_{r_\alpha/8}(y_\alpha) = m_1[S_1] + \ldots + m_k [S_k]$ as in the Definition \ref{def:model} of smooth model.  Each $S_i$ is orientable, being a smooth, closed (as sets) codimension-1 submanifold of the simply-connected region $B_{2} \setminus \cup_\alpha B_{r_\alpha/8}(y_\alpha)$.  We can therefore choose a normal vector $\nu_i$ on $S_i$ so that $|\nu_i - X_\bC| \leq c(\Lambda) \gamma$ on $S_i \cap A_{1+\sigma, 1-3\sigma}$.  After replacing $X_{\bC_\alpha}$ with $-X_{\bC_\alpha}$ if necessary, we have also $|\nu_i - X_{\bC_\alpha}(\cdot - y_\alpha)| \leq c(\Lambda) \gamma$ on $S_i \cap A_{2r_\alpha, r_\alpha/8}(y_\alpha)$, for each $\alpha$.

Consider the map $G_i : S_i \cap B_{1+\sigma} \times (-2\eps_S, 2\eps_S) \to \R^8$ by $G_i(x', t) = x' + t \nu_i(x')$.  By definition of the $\eps_S$, $G_i$ is a smooth diffeomorphism onto its image $O_i$, and the images $O_i$ are disjoint.  Provided $\gamma(\Lambda, \sigma)$ is sufficiently small, we have $O_i \supset B_{2\eps_S}(S_i) \cap O'$.

Define the smooth vector field $X_i = DG_i(\del_t)$ on $O_i$, so that $t \mapsto G_i(x', t)$ is the flow of $X_i$.  Let $\xi(t)$ be a smooth, increasing function satisfying
\[
\xi|_{(-\infty, 1-2\sigma]} \equiv 0, \quad \xi|_{[1-\sigma, \infty)} \equiv 1, \quad |\xi'| \leq 10/\sigma,
\]
and then define $Y_i$ on $B_{2\eps'}(S_i) \cap O'$ by
\[
Y_i(x) = \left\{ \begin{array}{l l} 
 (1-\xi(|x|)) X_i(x) + \xi(|x|) X_\bC(x) & x \in A_{1, 1-3\sigma} \\
 \xi( \frac{|x - y_\alpha|}{2r_\alpha}) X_i(x)  + (1-\xi(\frac{|x - y_\alpha|}{2r_\alpha})) X_{\bC_\alpha}(x - y_\alpha) & x \in A_{2r_\alpha, r_\alpha/5}(y_\alpha) \\
 X_i(x) & otherwise
 \end{array} \right. 
\]

For $Z_i = X_i, Y_i$, we have
\begin{equation}\label{eqn:smooth-map1}
\sup_{B_{2\eps'}(S_i) \cap A_{1,1-3\sigma}} |Z_i - X_\bC| + |DZ_i - DX_\bC| \leq c\gamma, 
\end{equation}
and for each $\alpha$, 
\begin{equation}\label{eqn:smooth-map2}
\sup_{B_{2\eps'}(S_i) \cap A_{2r_\alpha, r_\alpha/5}(y_\alpha)} |Z_i - X_{\bC_\alpha}(\cdot - y_\alpha)| + r_\alpha |DZ_i - DX_{\bC_\alpha}(\cdot - y_\alpha)| \leq c\gamma, 
\end{equation}
where $c = c(\Lambda, \sigma)$.  And of course $X_i = Y_i$ in $B_{1-3\sigma} \setminus \cup_\alpha B_{r_\alpha}(y_\alpha)$.  From the above and \eqref{eqn:X-prop}, provided $\gamma(\Lambda, \sigma)$ is sufficiently small, the flow $F_{X_i}(x', t)$ is defined for all $|t| < \eps'$, $x' \in S_i \cap O'$.

Let $H_i : S_i \cap O' \times (-\eps', \eps') \to \R^8$ be the smooth map given by $H_i(x', t) = F_{X_i}(x', t)$.  Recalling our restriction $\eps' \leq \eps_0 \min_\alpha r_\alpha/16$, \eqref{eqn:smooth-map1}, \eqref{eqn:smooth-map2}, \eqref{eqn:X-prop} imply
\begin{gather}
|H_i - G_i| \leq c(\Lambda, \sigma) \gamma |t|, \quad |DH_i - DG_i| \leq c(\Lambda, \sigma)\gamma \label{eqn:smooth-map3} \\
0.98|t| \leq d(G_i(x', t), S_i) \leq |t| \label{eqn:smooth-map4} \\
B_{\eps'/2}(S_i) \cap O' \subset H_i( S_i \cap O \times (-\eps', \eps')) \subset B_{2\eps'}(S_i) \cap O'. \label{eqn:smooth-map5}
\end{gather}
In particular, after identitying $S_i \cap O' \times \R$ with $T^\perp S_i$ via $\nu_i$, we have $|DH_i - Id| \leq \eps'' + c(\Lambda, \sigma)\gamma$.  Ensuring $\gamma(\Lambda, \sigma)$ and $\eps''(\Lambda)$ are sufficiently small, we get that $H_i$ is a smooth diffeomorphism onto its image.

By considering the map $H_i^{-1} \circ G_i$, ensuring $\beta(\eps, \eps'')$, $\gamma(\Lambda, \sigma)$ are sufficiently small, and recalling the definition of smooth region, for each $p = 1, 2, i = 1, \ldots, k, j = 1, \ldots, m_k$, we can find $C^2$ functions $u_{pij} : S_i \to (-\eps', \eps')$ satisfying
\begin{gather}
M_p \cap O' = \cup_{i=1}^k \cup_{j=1}^{m_i} \{ H_i(x', u_{pij}(x')) : x' \in S_i \cap O' \}, \label{eqn:smooth-map9} \\
|u_{pij}| \leq (1+c \gamma) \beta \eps_S \leq \eps'/10, \quad |u_{pij}|_{C^1} \leq c(\eps'' + \gamma + \beta \eps_S) \leq 1, \label{eqn:smooth-map10}
\end{gather}
for $c = c(\Lambda, \sigma)$.  Since each $M_p$ is multiplicity-one, by the maximum principle we can assume that
\begin{gather}\label{eqn:smooth-map8}
u_{pi1} < u_{pi2} < \ldots < u_{pi m_k} \quad \forall p = 1, 2; i = 1, \ldots, k.
\end{gather}

For each $i = 1, \ldots, k$, define the $C^2$ diffeomorphism $g_i : S_i \cap O \times (-\eps', \eps') \to S_i \cap O \times (-\eps, \eps)$ by
\[
g_i(x', t) = g_i(x', f_{\{ u_{1ij}(x') \}_j, \{u_{2ij}(x')\}_j}(t) ),
\]
where $f_{\{a_j\}_j, \{b_j\}_j}(t)$ is the function from Lemma \ref{lem:monof} with $m_i$ in place of $k$, and $\eps'/10$ in place of $\eps$.  Then $g_i$ is the identity for $|t| > \eps/5$ and $g_i(x', u_{1ij}(x')) = (x', u_{2ij}(x'))$ for every $j = 1, \ldots, m_i$.  If we let
\[
D_r = \min \{ u_{1ij}(x') - u_{1ij'}(x') : j < j', x' \in S_i \cap B_r \cap O \},
\] 
then $D_r > 0$ for all $r < 1$ (by \eqref{eqn:smooth-map8} and our definition of $O'$), and it follows by Lemma \ref{lem:monof} and \eqref{eqn:smooth-map10} that $|Dg_i| \leq c(\Lambda, \sigma, D_r) < \infty$ on $B_r$.

Define $\phi_i : O \to O$ by setting
\[
\phi_i(x) = H_i \circ g \circ H_i^{-1}  \text{ if } x \in B_{\eps'/2}(S_i) \cap O, \quad \phi_i(x) = x \text{ otherwise.}
\]
By \eqref{eqn:smooth-map4}, \eqref{eqn:smooth-map5} and our construction of $g_i$, $\phi_i$ is a well-defined $C^2$ diffeomorphism, satisfying
\[
\phi_i(\overline{M_1} \cap B_{\eps'}(S_i) \cap O) = \overline{M_2} \cap B_{\eps'}(S_i) \cap O, \quad |D\phi_i| \leq c(\Lambda, \sigma, M_1).
\]

Now set $\phi = \phi_1 \circ \cdots \circ \phi_k$.  We have $\phi( M_1 \cap O) = M_2\cap O$ by disjointness of the $O_i$, and $\phi|_{A_{1,1-\sigma}}$, $\phi|_{A_{r_\alpha, r_\alpha/4}(y_\alpha)}$ are $(\bC, \eps)$-, $(\bC_\alpha, \eps)$-maps by our construction and our choice of $\gamma, \eps'$.
\end{proof}

\begin{lemma}[Cone region parameterization]\label{lem:cone-map}
Take $\Lambda, \eps > 0$, $\beta \geq 0$.  Let $\bC \in \cC_\Lambda$, $\rho \in [0, 1)$, $m \in \N$.  For $i = 1, 2$, $j = 1, \ldots, m$ let $u_{ij} : \bC \cap A_{1, \rho} \to \R$ be $C^1$ functions satisfying
\[
|x|^{-1} |u_{ij}| + |\nabla u_{ij}| \leq \beta \leq \eps/4 \leq \eps_0/8, \quad u_{i1} < u_{i2} < \ldots < u_{im} \text{ for } i = 1, 2.
\]
Write $M_i = \cup_{j=1}^m G_\bC(u_{ij} \nu_\bC) \cap A_{1, \rho}$.

Then we can find a $C^1$ diffeomorphism $\phi : A_{1, \rho} \to A_{1,\rho}$ satisfying
\begin{gather}
\phi(M_1) = M_2, \quad |\phi(x)| = |x|, \quad |D\phi|_{C^0(A_{r, \rho/r})} \leq c(\Lambda, M_1, r) \text{ for every $r < 1$} \label{eqn:cone-map-concl1} \\
\text{$\phi|_{A_{1, \rho}}$ is a $(\bC, \eps)$-map.} \label{eqn:cone-map-concl2}
\end{gather}

If $m = 1$, then in fact we have $|D\phi - Id| \leq c(\Lambda) \frac{\beta}{\eps}$.  If the $u_{ij}$ are $C^k$, then $\phi$ is $C^k$ also.
\end{lemma}

\begin{remark}\label{rem:extend-lip}
If $\rho = 0$, the maximum principle \cite{ilmanen} implies $m = 1$, so provided $\beta/\eps$ is sufficiently small (depending only on $\Lambda$), $\phi$ extends to a bi-Lipschitz map $B_1 \to B_1$, satisfying
\[
(1-c(\Lambda) \beta/\eps)|x - y| \leq |\phi(x) - \phi(y)| \leq (1+ c(\Lambda) \beta/\eps) |x - y| \quad \forall x, y \in B_1.
\]
\end{remark}

\begin{remark}
It's tempting to think the $\phi$ in Lemma \ref{lem:cone-map} should also satisfy an estimate like $|x| |D^2 \phi| \leq c(\Lambda, M_1)$, provided the $u_{ij}$ satisfy a similar estimate $|x| |\nabla^2 u_{ij}| \leq \beta$.  This is certainly true if $m = 1$, but when $m \geq 2$ this seems to be false.  The issue is that second derivatives of the function $f_{\{a_j\}_j, \{b_j\}_j}$ from Lemma \ref{lem:monof} depend on the spacing of \emph{both} the $\{a_j\}_j$ and the $\{b_j\}_j$ (as opposed to the first derivative which depends only on the spacing of the $\{a_j\}_j$).
\end{remark}

\begin{proof}
Let $f_{\{a_j\}_j, \{b_j\}_j}(t)$ be the function from Lemma \ref{lem:monof}, with $m$ in place of $k$, $\eps/4$ in place of $\eps$, and let $U = \{ (x, t) \in \bC \cap A_{1, \rho} \times \R : |t| < 2\eps_0|x| \}$.  If $m \geq 2$, or $\beta \geq 10^{-3} \eps$, define $g : U \to U$ by
\[
g(x, t) = (x, |x| f_{\{ |x|^{-1} u_{1j}(x)\}_j, \{ |x|^{-1} u_{2j}(x) \}_j}(|x|^{-1} t ) ).
\]
If $m = 1$ and $\beta \leq 10^{-3}\eps$, then take $\eta(t)$ to be the function as in the proof of Lemma \ref{lem:monof}, and then instead set
\[
g(x, t) = (x, t + (u_{21}(x) - u_{11}(x)) \eta( 4 \eps^{-1} |x|^{-1} (t - u_{11}(x)) )).
\]
Then by Lemma \ref{lem:monof} $g$ is a $C^1$ diffeomorphism $U \to U$, satisfying
\begin{gather*}
g(x, t) = (x, t) \text{ for } |t| \geq \frac{\eps}{2} |x|, \quad |D g|_{C^0(U \cap A_{r, \rho/r})} \leq c(\Lambda, M_1, r) \text{ for all $r < 1$}, \\
g(x, u_{1j}(x)) = (x, u_{2j}(x)) \text{ for each } j = 1, \ldots, m .
\end{gather*}
If $m = 1$, then $|Dg - Id| \leq c(\Lambda) \beta/\eps$ on all of $U$.

Now define $\phi$ by
\[
\phi(x) = \left\{ \begin{array}{l l} G_\bC \circ g \circ G_\bC^{-1} & x \in A_{1,\rho} \cap G_\bC(U) \\ x & x \in A_{1,\rho} \setminus G_\bC(U) \end{array} \right. . \qedhere
\]
\end{proof}

\begin{remark}\label{rem:cone-c1alpha}
If $\rho = 0$ and $\bC$ is ``integrable through rotations'' (e.g. if $\cC(\bC)$ consists of rotations of $\bC$), then provided $\beta(\bC, \eps)$, $\eps(\Lambda)$ are sufficiently small, instead of \eqref{eqn:cone-map-concl1}, \eqref{eqn:cone-map-concl2} we can arrange $\phi$ to instead satisfy:
\begin{gather}
\text{$\phi$ is a $C^2$ diffeomorphism $A_{1,0} \to A_{1,0}$}, \quad \phi(M_1) = M_2, \quad |\phi(x)| = |x|,  \label{eqn:c1alpha-1} \\
\phi \text{ extends in a $C^{1,\alpha}$ fashion to $B_1$, for some $\alpha(\bC) > 0$}, \label{eqn:c1alpha-2} \\
\text{$\phi|_{A_{1, 1/8}}$ is a $(\bC, \eps)$-map}. \label{eqn:c1alpha-3}
\end{gather}
By using this $\phi$ in the proof of Theorem \ref{thm:param} instead of the $\phi$ generated by Lemma \ref{lem:cone-map}, one can get the map in Theorem \ref{thm:param} to be globally $C^{1,\alpha}$.

We outline how to do this.  \cite{AllAlm} (see also \cite[Lemma 1]{Simon1}) implies that for such $\bC$ there are constants $\alpha(\bC), \delta(\bC) \in (0, 1)$ so that if $M \in \cM_7(B_1, g)$, $M \cap (A_{1, 0}, g)$ is a $(\bC, 1, \delta)$-strong-cone region, and $|g - \geucl|_{C^3(B_1)} \leq \delta$, then there is a $\bC' \in \cC(\bC)$ being a rotation of $\bC$ such that
\[
M \cap A_{1, 0} = G_{\bC'}(u) \cap A_{1, 0}, \quad |x|^{-1} |u| + |\nabla u| + |x| |\nabla^2 u| \leq c(\bC) \delta |x|^\alpha,
\]
for some $C^2$ function $u : \bC' \cap A_{1, 0} \to \bC'^\perp$.

Now take $M_1, M_2$ satisfying the hypotheses of Lemma \ref{lem:cone-map} with this $\bC$, and $\beta(\bC,\eps)$, $\eps(\Lambda)$ to be chosen.  Combining the previous paragraph with Lemma \ref{lem:cone-map}, provided $\beta(\bC,\eps)$ is small we deduce there are $\bC_1, \bC_2$ rotations of $\bC$ and $C^2$ diffeomorphisms $\phi_1, \phi_2 : A_{1, 0} \to A_{1, 0}$, so that $\phi_i(\bC_i \cap A_{1,0}) = \spt M_i \cap A_{1,0}$, $|D\phi_i|_x - Id| \leq c(\bC) \beta |x|^\alpha$, and $\phi_i|_{A_{1, 0}}$ is a $(\bC_i, \eps)$-map.  In particular each $\phi_i$ extends to a $C^{1,\alpha}$ map on $B_1$ with $D\phi_i|_0 = Id$.  Let $P$ be any rotation taking $\bC_1$ to $\bC_2$, with $|P - Id| \leq c(\bC) \beta$.

On the other hand, we can also apply Lemma \ref{lem:cone-map} directly to obtain a $C^2$ diffeomorphism $\phi$ satisfying \eqref{eqn:cone-map-concl1}, \eqref{eqn:cone-map-concl2}.  Since $d_H(\bC_i \cap \del B_1, \bC \cap \del B_1) \leq c(\bC) \beta$, using Theorem \ref{thm:cones} and Lemma \ref{lem:glue}, ensuring $\eps(\Lambda)$ and $\beta(\Lambda, \eps)$ are sufficiently small, we can obtain a $C^2$ diffeomorphism $\psi : A_{1, 0} \to A_{1, 0}$ satisfying
\[
\psi|_{A_{1, 1/8}} = \phi, \quad \psi_{A_{1/16, 0}} = \phi_2 \circ P \circ \phi_1^{-1}, \quad |\psi(x)| = |x|, \quad |D\psi| \leq c(\bC).
\]
This $\psi$ will satisfy \eqref{eqn:c1alpha-1}, \eqref{eqn:c1alpha-2}, \eqref{eqn:c1alpha-3}.

Unfortunately, integrability is in general an open condition, so it's not clear if one can take $\beta$ to depend only on $\Lambda$, rather than the specific integrable cone $\bC$.  Of course, the only \emph{known} stable $7$-dimensional hypercones are the quadratic cones, which are all integrable through rotations, and one could obviously arrange $\beta$ so that the previous argument works with these cones.  In fact I believe there are no known minimal hypercones which are not integrable through rotations.  At any rate it would be interesting to understand this better.
\end{remark}

\begin{lemma}[Gluing Lemma]\label{lem:glue}
Given $\Lambda, \eps > 0$, $\sigma \in (0, 10^{-2}]$, there are constants $\delta_1(\Lambda, \sigma)$, $\delta'_1(\Lambda, \eps, \sigma)$, $\eps_1(\Lambda) \leq \eps_0/4$ so that provided $\eps \leq \eps_1$ the following holds.  Let $\bC \in \cC_\Lambda$, $m \in \N$.  For $i = 1, 2, j = 1, \ldots, m$, take $C^1$ functions $u_{ij} : \bC \cap A_{1+\sigma, 1-4\sigma} \to \R$ satisfying
\[
|u_{ij}|_{C^1} \leq \eps/10, \quad u_{i1} < u_{i2} < \ldots < u_{im} ,
\]
and then set $M_i = \cup_{j=1}^m G_\bC(u_{ij} \nu_\bC) \cap A_{1+\sigma, 1-4\sigma}$.

Let $O_1, O_2 \supset A_{1+\sigma, 1-4\sigma}$ be open sets in $\R^8$.  For $l = 1, 2$, let $T_l : O_l \to \R^8$, $\phi_l : O_l \to O_l$ be $C^1$ diffeomorphisms such that
\begin{gather*}
|T_l - id|_{C^1(O_l)} \leq \delta'_1, \quad T_l(O_l) \supset A_{1+\sigma, 1-4\sigma}, \\
T_l(\phi_l(M_1 \cap O_l)) = M_2 \cap T_l(O_l), \quad \phi_l = id \text{ outside } B_{\eps/10}(\bC) .
\end{gather*}
Assume additionally there are smooth vector fields $X_l$ on $O_l \cap B_{2\eps}(\bC)$ satisfying
\[
|X_l - X_\bC|_{C^1(B_{2\eps}(\bC) \cap O_l)} \leq \delta_1,
\]
so that each $\phi_l$ preserves the flowlines of $X_l$.

Define
\begin{gather*}
O = (O_1 \setminus B_{1-2\sigma}) \cup (O_2 \cap B_{1-\sigma}), \\
O' = (T_1(O_1) \setminus B_{1-2\sigma}) \cup (T_2(O_2) \cap B_{1-\sigma}).
\end{gather*}
Then we can find a $C^1$ diffeomorphism $h : O \to O'$ with the properties that:
\begin{gather*}
h|_{O \setminus B_{1-0.5\sigma}} = T_1 \phi_1, \quad h|_{O \cap B_{1-2.5\sigma}} = T_2 \phi_2, \\
h(M_1 \cap O) = M_2 \cap O', \quad |Dh| \leq c(\Lambda, \sigma, M_1, |D\phi_1|_{C^1(O_1)}, |D\phi_2|_{C^1(O_2)}).
\end{gather*}
If the $u_{ij}$, $T_l$, $\phi_l$ are $C^k$, then $h$ is $C^k$ also.
\end{lemma}

\begin{remark}
Of course the same Lemma holds true if $X_\bC$ is replaced with $-X_\bC$.
\end{remark}

\begin{proof}
Take $\eps \leq \eps_0/4$.  For ease of notation write $U = (\bC \cap A_{1, 1-3\sigma}) \times (-\eps, \eps)$.  Fix $\xi(t)$ to be a smooth, increasing function satisfying
\[
\xi|_{(-\infty, 1-2\sigma]} = 0, \quad \xi|_{[1-\sigma, \infty)} = 1, \quad \xi' \leq 10/\sigma.
\]
Define $T : A_{1+\sigma, 1-4\sigma} \to \R^8$ by $T(x) = \xi(|x|) T_1(x) + (1-\xi(|x|)) T_2(x)$.  Then for $\delta'(\sigma)$ sufficiently small, $T$ is a $C^1$ diffeomorphism onto its image satisfying
\[
|T - id|  \leq \delta', \quad |DT - Id| \leq c(\sigma) \delta'.
\]
Observe that from our hypotheses on $\phi_l$ and our construction of $T$, we have
\begin{gather}\label{eqn:glue1}
T_1 \phi_1 = T \text{ on }  A_{1, 1-\sigma} \setminus B_{\eps/10}(\bC), \quad T_2 \phi_2 = T \text{ on }  A_{1-2\sigma, 1-3\sigma} \setminus B_{\eps/10}(\bC) .
\end{gather}

Define on $B_{2\eps}(\bC) \cap A_{1+\sigma, 1-4\sigma}$ the smooth vector field
\[
X(x) = \xi(|x|) X_1(x) + (1-\xi(|x|)) X_2(x).
\]
Then we have
\[
|X - X_\bC|_{C^1(B_{2\eps}(\bC) \cap A_{1+\sigma, 1-4\sigma})} \leq c(\sigma) \delta.
\]
By the above and \eqref{eqn:X-prop}, provided $\delta(\sigma)$ is sufficiently small, the flow $F_X(x', t)$ of $X$ exists for all $|t| < \eps$ and $x' \in \bC \cap A_{1,1-3\sigma}$.

Define $G : U \to \R^8$ by $G(x', t) = F_X(x', t)$.  Then, ensuring $\eps(\Lambda)$ and $\delta(\Lambda, \sigma)$ are small, $G$ is a smooth diffeomorphism onto its image satisfying:
\begin{gather}
|G - G_\bC| \leq c \delta |t|, \quad |DG - DG_\bC| + |D (G^{-1} \circ G_\bC) - Id| \leq c \delta, \label{eqn:glue2} \\
(1-c \delta) |x'| \leq |G(x', t)| \leq (1+c\delta)|x'| \label{eqn:glue3} \\
0.98|t| \leq d(G(x', t), \bC) \leq 1.02|t|, \label{eqn:glue4}
\end{gather}
for $c = c(\Lambda, \sigma)$, and
\begin{gather}\label{eqn:glue5}
B_{0.9\eps}(\bC) \cap A_{1-0.1\sigma, 1-2.9\sigma} \subset G(U) \subset B_{1.1\eps}(\bC) \cap A_{1+0.1\sigma, 1-3.1\sigma},
\end{gather}
Similarly, ensuring $\delta'(\eps, \sigma)$ is small, then the diffeomorphism $TG \equiv T \circ G$ satisfies (with $c = c(\Lambda, \sigma)$)
\begin{gather}
|T G - G_\bC| \leq c (\delta|t| + \delta'), \quad | D(T G) - G_\bC| + | D( G^{-1} T^{-1} G_\bC) - Id| \leq c (\delta + \delta') \label{eqn:glue6} \\
B_{0.8\eps}(\bC) \cap A_{1-0.2\sigma, 1-2.8\sigma} \subset (T G)(U) \subset B_{1.2\eps}(\bC) \cap A_{1+0.2\sigma, 1-3.2\sigma}. \label{eqn:glue7}
\end{gather}

By considering the maps $G^{-1} G_\bC$ and $G^{-1} T^{-1} G_\bC$, ensuring $\delta'(\eps, \Lambda, \sigma)$, $\delta(\Lambda, \sigma)$ are small, then \eqref{eqn:glue2}-\eqref{eqn:glue7} imply we can find $C^1$ functions $u^G_j , u^T_j : U \to \R$ ($j = 1, \ldots, m$) satisfying
\begin{gather}
M_1 \cap G(U) = \cup_{j=1}^m \{ G(x', u^G_j(x')) : x' \in U \} , \nonumber \\
|u^G_j| \leq (1+c \delta) \eps/10 \leq \eps/9, \quad |u^G_j|_{C^1(U)} \leq c(\beta + \delta) \leq 1 , \label{eqn:glue8} \\ 
u^G_1 < u^G_2 < \ldots < u^G_m ,  \nonumber 
\end{gather}
and
\begin{gather}
M_2 \cap TG(U) = \cup_{j=1}^m \{ TG(x', u^T_j(x')) : x' \in U \} , \nonumber \\
|u^T_j| \leq (1+c\delta) \eps/10 + c \delta' \leq \eps/9, \quad |u^T_j|_{C^1(U)} \leq c( \beta + \delta + \delta') \leq 1, \label{eqn:glue9} \\
u^T_1 < u^T_2 < \ldots < u^T_m, \nonumber
\end{gather}
where $c = c(\Lambda, \sigma)$.

By \eqref{eqn:glue1}, \eqref{eqn:glue3} (taking $\delta(\Lambda, \sigma)$ sufficiently small), \eqref{eqn:glue4}, and our hypotheses on $T_l$, $\phi_l$, we can define $C^1$ functions $t_1 : \bC \cap A_{1, 1-0.9\sigma} \times (-\eps, \eps) \to \R$, $t_2 : \bC \cap A_{1-2.1\sigma, 1-3\sigma} \to \R$ by
\[
(G^{-1} \circ T^{-1} \circ T_l \circ \phi_l \circ G)(x', t) = (x', t_l(x', t)) .
\]
The $t_l$ have the properties that: for each $x'$, $t_l(x', \cdot)$ is a diffeomorphism of $(-\eps, \eps)$ coinciding with $id$ for $|t| > \eps/9$; $t_l(x', u_j^G(x')) = u_j^T(x')$; and
\begin{gather}\label{eqn:glue10}
|\nabla t_l| + |\del_t t_l| \leq c(\Lambda, \sigma, |\phi_l|_{C^1(O_l)}).
\end{gather}

Let $\eta_1, \eta_2 : \R \to \R$ be a smooth functions satisfying
\begin{gather*}
\eta_1|_{(-\infty, 1-2.3\sigma]} \equiv 1, \quad \eta_1|_{[1-2.2\sigma, \infty)} \equiv 0, \quad 0 \leq -\eta_1' \leq 100/\sigma \\
\eta_2|_{(-\infty, 1-0.8\sigma]} \equiv 1, \quad \eta_2|_{[1-0.7\sigma, \infty)} \equiv 0, \quad 0 \leq -\eta_2' \leq 100/\sigma .
\end{gather*}
Define $f : U \to \R$ by
\[
f(x', t) = f_{ \{u_j^G(x')\}_j, \{ u_j^T(x')\}_j}(t)
\]
where $f_{\{a_j\}_j, \{b_j\}_j}(t)$ is the function from Lemma \ref{lem:monof} with $\eps/9$ in place of $\eps$, and $m$ in place of $k$.  Define $s : U \to \R$ by
\begin{gather*}
s(x', t) = \left\{ \begin{array}{l l}
 \eta_1(|x'|) t_2(x', t) + (1-\eta_1(|x'|)) f(x', t) &  1-3\sigma < |x'| < 1-2.2\sigma \\
 f(x', t) & 1-2.2\sigma \leq |x'| < 1-0.8\sigma \\
 \eta_2(|x'|) f(x', t) + (1-\eta_2(|x'|)) t_1(x', t) & 1-0.8\sigma \leq |x'| < 1 \end{array} \right. .
\end{gather*}
Then $s$ is $C^1$ (or $C^k$ if the $\phi_l$, $T_l$, $u_{ij}$ are $C^k$), and $s(x', \cdot)$ is a diffeomorphism of $(-\eps, \eps)$ for every $x'$, and $s(x',t)$ enjoys the properties:
\begin{gather*}
s = t_1 \text{ if } |x'| > 1-0.7\sigma, \quad s = t_2 \text{ if } |x'| < 1-2.3\sigma ,  \\
s(x', t) = t \text{ if } |t| \geq \eps/4 \\
s(x', u_j^G(x')) = u_j^T(x') \text{ for every } j = 1, \ldots, m.
\end{gather*}
If we let
\[
D = \inf \{ u_j^G(x') - u_{j'}^G(x') : j < j', x' \in \bC \cap A_{1-0.5\sigma, 1-2.5\sigma} \} > 0,
\]
then by \eqref{eqn:glue8}, \eqref{eqn:glue9}, \eqref{eqn:glue10} and Lemma \ref{lem:monof} we get
\[
|\nabla s| + |\del_t s|  \leq c(\Lambda, \sigma, D, |\phi_1|_{C^1(O_1)}, |\phi_2|_{C^1(O_2)}) < \infty.
\]

Let us define $g : U \to U$ by $g(x', t) = (x', s(x', t))$.  Trivially $g$ is a $C^1$ (or $C^k$) diffeomorphism of $U$.  We then define $h : O \to O'$ by 
\begin{align*}
h(x) = \left\{ \begin{array}{l l}
(T\circ G \circ g \circ G^{-1})(x) & x \in B_{0.9\eps}(\bC) \cap A_{1-0.1\sigma, 1-2.9\sigma} \\
T(x) & x \in A_{1-0.1\sigma, 1-2.9\sigma} \setminus B_{0.9\eps}(\bC) \\
T_1\phi_1(x) & x \in O \setminus B_{1-0.2\sigma} \\
T_2\phi_2(x) & x \in O \cap B_{1-2.8\sigma}  \end{array} \right. .
\end{align*}
It follows by \eqref{eqn:glue2}-\eqref{eqn:glue5} and our definition of $g$ that $\phi$ is well-defined, and satisfies the requirements of the Lemma.
\end{proof}

We are now set up to prove Theorem \ref{thm:param}.
\begin{proof}[Proof of Theorem \ref{thm:param}]
We prove Theorem \ref{thm:param} by induction on $N$ and by contradiction.  Suppose either $N = 1$, or by inductive hypothesis the Lemma holds with $N-1$ in place of $N$.  Suppose, towards a contradiction, the Lemma fails, for $\beta, \gamma, \eps$ for the moment arbitrarily fixed, but to be chosen later (we note that there is no loss in assuming $\eps(\Lambda, \sigma)$ is as small as we like).  Then there is a finite collection $\cS$ of $(\Lambda, \sigma, \gamma)$-smooth models, a set $\cG$ of $C^3$ metrics on $B_1$, sequences $g_i \in \cG$, $M_i \in \cM_7(B_1, g_i)$, $\bC_i \in \cC_\Lambda$, and an $r < 1$, so that each $M_i \cap B_1$ admits a $(\Lambda, \beta, \cS, N)$-strong-cone decomposition, but for every $j > i$ there is no local bi-Lipschitz map $\phi_{ij} : B_1 \to B_1$ satisfying $Lip(\phi_{ij}|_{B_r}) \leq j$, and making Theorem \ref{thm:param}(\ref{item:param1}),(\ref{item:param2}),(\ref{item:param4}),(\ref{item:param5}) true with $\phi_{ij}$, $M_i$, $\bC_i$, $M_j$ in place of $\phi$, $M_v$, $\bC_v$, $M$ (respectively).  Note that if we replace $i$ with any subsequence of $i$, then this new sequence satisfies the same contradiction hypotheses.

By hypothesis, for every $i$ there is a collection of at most $N$ smooth regions and strong-cone regions which fit together as per Definition \ref{def:decomp}.  After passing to a subsequence (and apply Theorem \ref{thm:cones}), we can assume that for all $i$ we have either: $M_i \cap (B_1, g_i)$ is a $(S, \beta)$-smooth region for some fixed smooth model $(S, \bC, \{(\bC_\alpha, B_{r_\alpha}(y_\alpha))\}_\alpha) \in \cS$; or $M_i \cap (A_{1, \rho_i}, g_i)$ is a $(\bC, m, 2\beta)$-strong-cone region for some $\bC \in \cC_\Lambda$, $m \in \N$, $\rho_i \leq 1/2$.

We treat each case separately, but observe first that by definition of cone-/smooth-region, we have for each $i$:
\begin{gather*}
d_H(M_i \cap A_{1, 1-\sigma}, \bC \cap A_{1, 1-\sigma}) \leq c(\Lambda)(\beta + \gamma) \leq \eps, \\
\sing M_i \subset B_{1-\sigma},
\end{gather*}
provided we take $\beta(\Lambda, \sigma, \eps), \gamma(\Lambda, \sigma, \eps)$ sufficiently small.  So Theorem \ref{thm:param}(\ref{item:param5}) holds for $M_i$, $\bC$ in place of $M$, $\bC_v$ (resp.) for all $i$.

\textbf{Case 1:} Every $M_i \cap (B_1, g_i)$ is a $(S, \beta)$-smooth region.  Let $O = B_1 \setminus \cup_\alpha B_{r_\alpha/4}(y_\alpha)$.  Taking $\gamma(\Lambda, \sigma, \eps)$, $\beta(\Lambda, \eps)$ sufficiently small, we can apply Lemma \ref{lem:smooth-map} to obtain a $C^2$ diffeomorphism $\psi_i : O \to O$ satisfying
\begin{gather*}
\psi_i(\overline{M_1} \cap O) = \overline{M_i} \cap O, \quad |D\psi_i|_{C^0(O \cap B_r)} \leq c(\Lambda, \sigma, M_1, r) \text{ for any } r < 1 \\
\psi_i|_{A_{1,1-\sigma}} \text{ is a $(\bC, \eps)$-map}, \\
\psi_i|_{A_{r_\alpha, r_\alpha/4}(y_\alpha)} \text{ is a $(\bC_\alpha, \eps)$-map, for each $\alpha$.}
\end{gather*}
If $\{\alpha\} = \emptyset$, then $O = B_1$, and we deduce a contradiction.  This proves Case 1 when $N = 1$.

By our inductive hypothesis, passing to a further subsequence as necessary, we can find for every $\alpha$ a $g_\alpha \in \cG$, $\hat M_\alpha \in \cM_7(B_1, g_\alpha)$, $\hat \bC_\alpha \in \cC_\Lambda$, $x_{\alpha i} \in B_{\beta r_\alpha}(y_\alpha)$, $r_{\alpha i} \in (\frac{1}{2}r_\alpha, (1+\beta)r_\alpha)$, and local bi-Lipschitz maps $\hat\phi_{\alpha i} : B_1 \to B_1$, so that Theorem \ref{thm:param}(\ref{item:param1})-(\ref{item:param5}) hold with $\hat \phi_{\alpha i}$, $\hat M_\alpha$, $\hat \bC_\alpha$, $\eta_{x_{\alpha i}, r_{\alpha i}}(M_i)$, $C_{\cS, N-1, \cG}$ in place of $\phi$, $M_v$, $\bC_v$, $M$, $C_{\cS, N, \cG}$ (respectively).  We can moreover assume that the $\{r_{\alpha i}\}_i$ converge, and hence assume $1-\beta \leq \frac{r_{\alpha i}}{r_{\alpha 1}} \leq 1 + \beta$ for for all $i$.

By virtue of $M_i \cap (B_1, g_i)$ being a smooth region, if we ensure $\gamma(\Lambda)$, $\beta(\Lambda)$ are sufficiently small and recall that $\eps_S \leq \min_\alpha r_\alpha$, then we can write
\begin{equation}\label{eqn:param1}
M_i \cap A_{2r_\alpha, r_\alpha/4}(y_\alpha) = \cup_{j=1}^m G_{y_\alpha + \bC_\alpha}(u_{\alpha i j}) \cap A_{2r_\alpha, r_\alpha/8}(y_\alpha)
\end{equation}
for $C^2$ functions $u_{\alpha i j} : (y_\alpha + \bC_\alpha) \cap A_{2r_\alpha, r_\alpha/8}(y_\alpha) \to \bC_\alpha^\perp$ satisfying
\begin{equation}\label{eqn:param2}
r_\alpha^{-1} |u_{\alpha i j}| + |\nabla u_{\alpha i j}| + r_\alpha |\nabla^2 u_{\alpha i j}| \leq c(\Lambda)(\gamma + \beta) \leq \eps_1(\Lambda)/100
\end{equation}
where $\eps_1$ as in Lemma \ref{lem:glue}.

Ensure $\beta \leq 10^{-3} \sigma$, and let $\sigma' = 10^{-2} \sigma$, $r_{\alpha}' = \frac{1-2\sigma'}{1+5\sigma'} \min\{ r_\alpha, r_{\alpha 1} \}$.  Then we have
\[
A_\alpha := A_{(1+\sigma')r_\alpha', (1-4\sigma')r_\alpha'}(y_\alpha) \subset A_{r, (1-\sigma)r}(z),
\]
for any $z \in B_{\beta r_\alpha}(y_\alpha)$ and any $r \in ((1-10\beta) r_{\alpha 1}, (1+10\beta) r_{\alpha 1})$.  Note in particular this holds for $z = x_{\alpha i}$ and $r = r_{\alpha i}$.

From \eqref{eqn:param1},\eqref{eqn:param2}, Theorem \ref{thm:param}(\ref{item:param5}), and the bounds $|y_\alpha - x_{\alpha i}| \leq \beta r_\alpha$, $\frac{1}{2} \leq \frac{r_{\alpha i}}{r_\alpha} \leq \frac{3}{2}$, we get
\[
d_H(\bC_\alpha \cap B_1, \hat \bC_\alpha \cap B_1) \leq c(\Lambda)(\gamma + \beta + \eps).
\]
Recalling that $\eps_1 \leq \eps_0/4$ (for $\eps_1(\Lambda)$ as in Lemma \ref{lem:glue}, $\eps_0(\Lambda)$ as in the earlier discussion in this Section), and ensuring $\gamma, \beta, \eps$ are sufficiently small (depending only on $\Lambda, \sigma$), we deduce that both $X_{\hat \bC_\alpha}(x - x_{\alpha 1})$ and $X_{\bC_\alpha}(x - y_\alpha)$ are defined for $x \in B_{2\eps_1 r'_\alpha}(y_\alpha + \bC_\alpha) \cap A_\alpha$ and satisfy
\begin{align*}
&| X_{\hat \bC_\alpha}(x - x_{\alpha 1}) - X_{\bC_\alpha}(x - y_\alpha)| \label{eqn:param3} \\
&\quad + r_\alpha' |D X_{\hat \bC_\alpha}(x - x_{\alpha 1}) - D X_{\bC_\alpha}(x - y_\alpha)| \leq \delta_1(\Lambda, \sigma'). \nonumber
\end{align*}
Here $\delta_1(\Lambda, \sigma')$ are the constant from Lemma \ref{lem:glue}.  Finally, if we set $T_{\alpha i} = \eta_{x_{\alpha i}, r_{\alpha i}}^{-1} \circ \eta_{x_{\alpha 1}, r_{\alpha 1}}$, then provided we take $\beta(\Lambda, \sigma)$ small we get
\[
r_\alpha'^{-1} |T_{\alpha i} - id| + |DT_{\alpha i} - Id| \leq \delta_1'(\Lambda, \eps_1(\Lambda), \sigma') \text{ on } B_{2r'_\alpha}(y_\alpha),
\]
where $\delta_1'(\Lambda, \eps_1(\Lambda), \sigma')$ as in Lemma \ref{lem:glue}.

Now define the local bi-Lipschitz maps $\phi_{\alpha i} : B_{r_{\alpha 1}}(x_{\alpha 1}) \to B_{r_{\alpha i}}(x_{\alpha i})$ by
\[
\phi_{\alpha i}(x) = (\eta_{x_{\alpha i}, r_{\alpha i}})^{-1} \circ \hat \phi_{\alpha i} \circ \hat \phi_{\alpha 1}^{-1} \circ \eta_{x_{\alpha 1}, r_{\alpha i}}.
\]
The $\phi_{\alpha i}$ restrict to $C^2$ diffeomorphisms $A_{r_{\alpha 1}, (1-\sigma)r_{\alpha 1}}(x_{\alpha 1}) \to A_{r_{\alpha i}, (1-\sigma)r_{\alpha i}}(x_{\alpha i})$, and map $\spt M_1 \cap B_{r_{\alpha 1}}(x_{\alpha 1})$ to $\spt M_i \cap B_{r_{\alpha i}}(x_{\alpha i})$, and admit the bound $\Lip(\phi_{\alpha i}|_{B_{(1+\sigma')r_\alpha'}(y_\alpha)}) \leq C_{\cS, N-1, \cG}(\frac{1+\sigma'}{1+5\sigma'})$.  Moreover, from Theorem \ref{thm:param}(\ref{item:param4}) and our definition of $A_\alpha$ we have that
\begin{gather*}
T_{\alpha i}^{-1} \phi_{\alpha i} |_{A_\alpha} \text{ preserves the flow lines of $X_{\hat \bC_\alpha}(\cdot - x_{\alpha 1})$} , 
\end{gather*}
and
\begin{align*}
T_{\alpha i}^{-1} \phi_{\alpha i} |_{A_\alpha} = id \text{ outside } &B_{\eps r_\alpha}(x_{\alpha 1} + \hat \bC_\alpha) \\
&\subset B_{c(\Lambda)(\gamma + \beta + \eps) r_\alpha}(y_\alpha + \bC_\alpha) \\
&\subset B_{\eps_1(\Lambda) r_\alpha'/10}(y_\alpha + \bC_\alpha)
\end{align*}
provided $\gamma, \beta, \eps$ are sufficiently small (depending only on $\Lambda, \sigma$).

On the other hand, since $A_\alpha \subset A_{r_\alpha, r_\alpha/4}(y_\alpha)$, by construction $\psi_i|_{A_\alpha}$ preserves the flowlines of $X_{\bC_\alpha}(\cdot - y_\alpha)$, and $\psi_i|_{A_\alpha} = id$ outside $B_{\eps r_\alpha}(y_\alpha + \bC_\alpha) \subset B_{\eps_1(\Lambda) r_\alpha'/10}(y_\alpha + \bC_\alpha)$.

The above discussion implies we can apply Lemma \ref{lem:glue} to the maps $\psi_i$, $\phi_{\alpha i}$ at scale $B_{r_\alpha'}(y_\alpha)$ (with $\sigma'$ in place of $\sigma$), to obtain local bi-Lipschitz maps $\phi_i : B_1 \to B_1$ satisfying:
\begin{gather*}
\phi_i|_{B_1 \setminus \cup_\alpha B_{r_\alpha'}(y_\alpha)} = \psi_i , \quad \phi_i|_{B_{(1-3\sigma')r'_\alpha}(y_\alpha)} = \phi_{\alpha i} \\
\phi_i(\overline{M_1} \cap B_1) = \overline{M_i} \cap B_1, \quad \phi_i(\sing M_1) = \sing M_i \\
\phi_i|_{B_1 \setminus \sing M_1} \text{ is a $C^2$ diffeomorphism} \\
\quad \Lip(\phi_i|_{B_r}) \leq c(\Lambda, \sigma, M_1, C_{\cS, N-1, \cG}, r) \text{ for } r < 1 .
\end{gather*}
This is a contradiction, and therefore completes the proof of Case 1.

\textbf{Case 2:} Every $M_i \cap (A_{1,\rho_i}, g_i)$ is a $(\bC, m, 2\beta)$-strong-cone region.  Passing to a subsequence, we can assume that either $\rho_i = 0$ for all $i$, or $\rho_i > 0$ for all $i$.  By the maximum principle, we can write
\begin{gather}\label{eqn:param6}
M_i \cap A_{1, \rho_i/8} = \cup_{j=1}^m G_\bC(u_{ij} \nu_\bC) \cap A_{1,\rho_i}
\end{gather}
for $C^2$ functions $u_{ij} : \bC \cap A_{1,\rho_i/8} \to \R$ satisfying
\begin{gather}\label{eqn:param7}
|x|^{-1}|u_{ij}| + |\nabla u_{ij}| + |x| |\nabla^2 u_{ij}| \leq 2\beta, \quad u_{i1} < u_{i2} < \ldots < u_{im} .
\end{gather}
Let us ensure that $2\beta \leq \min\{ \eps / c(\Lambda), \eps_0(\Lambda)/8\}$ for $c(\Lambda) \geq 4$ any fixed constant sufficiently large so that Remark \ref{rem:extend-lip} applies.

Suppose first that $\rho_i = 0$ for all $i$.  By definition of (strong-)cone-decomposition and Theorem \ref{thm:cones}, necessarily $\theta_\bC(0) > 1$, and so $\sing M_i = \{0\}$.  We can use Lemma \ref{lem:cone-map}, Remark \ref{rem:extend-lip} to obtain bi-Lipschitz maps $\psi_i : B_1 \to B_1$, with $\psi_i = 0$, which restrict to $C^2$ diffeomorphisms $A_{1,0} \to A_{1,0}$, and satisfy
\begin{gather*}
\psi_i(\overline{M_1} \cap B_1) = \overline{M_i} \cap B_1, \quad \Lip(\psi_i) \leq c(\Lambda, M_1) , \\
\psi_i|_{A_{1, 1/8}} \text{ is a $(\bC, \eps)$-map.}
\end{gather*}
This is a contradiction, and proves Case 2 when $N = 1$.

Let us assume now $\rho_i > 0$ for all $i$.  Passing to a further subsequence, by our inductive hypothesis we can find a $\hat g \in \cG$, $\hat M \in \cM_7(B_1, \hat g)$, $\hat \bC \in \cC_\Lambda$, and local bi-Lipschitz maps $\hat \phi_i : B_1 \to B_1$ such that Theorem \ref{thm:param}(\ref{item:param1})-(\ref{item:param5}) hold with $\hat \phi_i$, $\hat M$, $\hat \bC$, $\eta_{0, \rho_i}(M_i)$, $C_{\cS, N-1, \cG}$ in place of $\phi_i$, $M_v$, $\bC_v$, $M$, $C_{\cS, N,\cG}$ (respectively).

Let $\phi_i = \eta_{0, \rho_i}^{-1} \circ \hat \phi_i \circ \eta_{0, \rho_1}$, so that $\phi_i$ are local bi-Lipschitz maps $B_{\rho_1} \to B_{\rho_i}$ satisfying
\begin{gather*}
\phi_i(\overline{M_1} \cap B_{\rho_1}) = \overline{M_i} \cap B_{\rho_i}, \quad \phi_i(\sing M_1 \cap B_{\rho_1}) = \sing M_i \cap B_{\rho_i} \\
\phi_i|_{B_{\rho_1} \setminus \sing M_1} \text{ is a $C^2$ diffeomorphism,} \\
\Lip(\phi_i|_{B_{r\rho_1}}) \leq \rho_1^{-1} C_{\cS, N-1,\cG}(r) \text{ for every $r < 1$, independent of $i$}, \\
\frac{\rho_1}{\rho_i} \phi_i|_{A_{\rho_1, (1-\sigma)\rho_1}} \text{ is a $(\hat\bC, \eps)$-map.}
\end{gather*}

Define $M_{m\bC}$ by
\[
M_{m\bC} = \cup_{j=1}^m G_\bC( \frac{\eps_0 j}{2m} |x| \nu_\bC)\cap A_{\infty, 0},
\]
so that $M_{m\bC}$ looks very close to a multiplicity-$m$ $\bC$.  From Lemma \ref{lem:cone-map}, we can find $C^2$ diffeomorphisms $\hat \psi_i : A_{1, \rho_i/8} \to A_{1, \rho_i/8} $ such that
\begin{gather*}
\hat \psi_i(M_{m\bC} \cap A_{1, \rho_i/8}) = M_i \cap A_{1, \rho_i/8}, \quad |D\hat\psi_i| \leq c(\bC, m) \\
\hat\psi_i|_{A_{1, \rho_i}} \text{ is a  $(\bC, \eps)$-map.}
\end{gather*}
Using the $\hat f$ from Lemma \ref{lem:monof}, we can find smooth increasing functions $f_i : [0, 1] \to [0, 1]$ satisfying
\begin{gather*}
f_i(t) = t \text{ if } t \in [1-2\sigma, 1] , \quad f_i(t) = \frac{\rho_i}{\rho_1} t \text{ if } t \in [0, (1+\sigma)\rho_1], \quad 0 <  f_i' \leq c(\rho_1) .
\end{gather*}
Let $F_i(x) = f_i(|x|) x/|x|$, and then each $F_i$ is a smooth diffeomorphism $B_1 \to B_1$ which fixes $M_{m\bC} \cap B_1$, coincides with the identity in $A_{1, 1-\sigma}$, coincides with $\frac{\rho_i}{\rho_1} id$ in $B_{(1+\sigma)\rho_1}$, and admits the uniform bound $|DF_i| \leq c(\rho_1)$ independent of $i$.

Define $\psi_i = \hat \psi_i \circ F_i \circ \hat \psi_1^{-1}$.  The $\psi_i$ are $C^2$ diffeomorphisms $A_{1, \rho_1/8} \to A_{1, \rho_i/8}$ satisfying:
\begin{gather*}
\psi_i(M_1 \cap A_{1,\rho_1/8}) = M_i \cap A_{1,\rho_i/8}, \quad |D\psi_i|_{C^0(A_{r, \rho_1/4})} \leq c(M_1, r) \text{ for } r < 1, \\
\psi_i|_{A_{1, 1-\sigma}} \text{ is a $(\bC, \eps)$-map},\quad F_i^{-1} \psi_i|_{A_{\rho_1, \rho_1/4}} \text{ is a $(\bC, \eps)$-map.}
\end{gather*}

Let $\sigma' = 10^{-2}\sigma$ and $r' = \frac{1}{1+2\sigma'} \rho_1$, so that $A := A_{(1+\sigma')r', (1-4\sigma') r'} \subset A_{\rho_1, (1-\sigma)\rho_1}$.  By \eqref{eqn:param6}, \eqref{eqn:param7}, and Theorem \ref{thm:param}(\ref{item:param5}), we have
\[
d_H(\hat \bC \cap A, \bC \cap A) \leq 2 (\eps + 2\beta) r',
\]
and hence by Theorem \ref{thm:cones} we can take $\eps(\Lambda, \sigma), \beta(\Lambda, \sigma)$ small to deduce
\begin{gather}\label{eqn:glue11}
|X_\bC - X_{\hat\bC}|_{C^1(B_{2\eps_1}(\bC) \cap A_{1, 1/2})} \leq \delta_1(\Lambda, \sigma') ,
\end{gather}
with $\eps_1$, $\delta_1$ being the constants from Lemma \ref{lem:glue}.

Write $\tilde M_i = \frac{\rho_1}{\rho_i} M_i$.  On $A$, we have that $F_i^{-1} \psi_i = \frac{\rho_1}{\rho_i} \circ \hat \psi_i \circ \frac{\rho_i}{\rho_1} \circ \hat\psi_i^{-1}$, and therefore
\begin{gather}
F_i^{-1} \psi_i(\tilde M_1 \cap A) = \tilde M_i \cap A, \quad |D (F_i^{-1} \psi_i)|_{C^0(A)} \leq c(M_1), \label{eqn:param8} \\
F_i^{-1} \psi_i|_A \text{ preserves the flowlines of $X_\bC$}, \label{eqn:param9} \\
F_i^{-1} \psi_i|_A \equiv id \text{ outside $B_{\eps r'}(\bC) \subset B_{\eps_1 r'/10}(\bC)$}, \label{eqn:param10}
\end{gather}
ensuring $\eps(\Lambda)$ is small.  Likewise, on $A$ we have $F_i^{-1} \phi_i = \rho_1 \circ \hat \phi_i \circ \frac{1}{\rho_1}$, and so \eqref{eqn:param8}, \eqref{eqn:param9}, \eqref{eqn:param10} hold for $F_i^{-1} \phi_i$ in place of $F_i^{-1} \psi_i$ and $C_{\cS, N-1,\cG}(\frac{1+\sigma}{1+2\sigma})$ in place of $c(M_1)$.

In light of \eqref{eqn:param6}, \eqref{eqn:param7} (being scale-invariant), \eqref{eqn:glue11}, and the previous paragraph, we can therefore apply Lemma \ref{lem:glue} to the maps $F_i^{-1} \psi_i$, $F_i^{-1} \phi_i$ at scale $B_{r'}$ (with $\sigma'$ in place of $\sigma$) to obtain a $C^2$ diffeomorphism $h$ on $A_{\rho_1, (1-\sigma)\rho_1}$ taking $\tilde M_1 = M_1$ to $\tilde M_i$, and satisfying
\[
h|_{A_{\rho_1, r'}} = F_i^{-1} \psi_i, \quad h|_{A_{(1-3\sigma')r', (1-\sigma)\rho_1}} = F_i^{-1} \phi_i, \quad |Dh| \leq c(\Lambda, \sigma, M_1, C_1).
\]

If we now define the local bi-Lipschitz maps $\Psi_i : B_1 \to B_1$ by
\[
\Psi_i = \left\{ \begin{array}{l l} \psi_i(x) & x \in A_{1, \rho_1} \\ F_i h (x) & x \in A_{\rho_1, (1-\sigma)\rho_1} \\ \phi_i(x) & x \in B_{(1-\sigma)\rho_1} \end{array}\right. ,
\]
then it follows that Theorem \ref{thm:param}(\ref{item:param1})-(\ref{item:param4}) hold with $\Psi_i$, $M_1$, $\bC$, $M_i$, $c(\Lambda, \sigma, M_1, C_{\cS, N-1,\cG}, r)$ in place of $\phi$, $M_v$, $\bC_v$, $M$, $C_{\cS, N,\cG}$.  This contradicts our choice of $M_i$, and completes the proof of Case 2. \qedhere

\end{proof}

\section{Global finiteness}\label{sec:global}

In this section we prove Theorem \ref{thm:main2}.  Here we shall work in $(N^8, g)$, for $N^8$ a closed Riemannian $8$-manifold with $C^3$ metric $g$.  We need to extend some of our definitions to be adapted to global geometry.  Recall that if $M$ is an embedded $C^2$ hypersurface of $(N, g)$, then $T^{\perp_g} M$ denotes the normal bundle of $M$ in $(N, g)$, and $\nabla^g$ is the connection induced by $g$ on $T^{\perp_g} M$.  We identity each $T_x^{\perp_g} M$ with a subspace of $T_x N$ in the obvious fashion.  We write $C^k(M, M^\perp)$ for the space of $C^k$ sections of this bundle, and define $|u|_{C^k(N, g)}$ in terms of $\nabla^g$.  Given $u \in C^k(M \cap U, M^{\perp_g})$, write $\graph_{M, g}(u) = \{ \exp_N(u(x)) : x \in U \cap M \}$.  We define convergence of hypersurfaces $M_i \to M$ in the same way as in $\R^8$, except with the notion of $\graph_{M, g}$ defined here.

In direct analogy with the set up of Section \ref{sec:prelim}, a $7$-varifold $V$ in $N$ is now a measure on the Grassmann bundle $G_7(TN)$.  We define $\cISV_7(N, g)$ to be the set of stationary integral $7$-varifolds $V$ in $(N, g)$ with $\sing V$ consisting of a discrete collection of points.  The various special cases and constructions (e.g. $[M]_g$, $f_\sharp V$, $\delta_g V$, etc.) extend in an obvious fashion to general $(N, g)$.  Lemmas \ref{lem:rs-bound}, \ref{lem:isv-decomp}, \ref{lem:index-compact} continue to hold for $(N, g)$ in place of $(U, g)$, but if one prefers one could break $N$ into finitely-many coordinate neighborhoods, and apply the Lemmas in each one.  Let us write $\theta_{V, g}(x, r) = \omega_7^{-1} r^{-7}\mu_V(B_r^g(x))$ for the density ratio of $V$ w.r.t. $g$-geodesic balls, and if $M$ is a $C^1$ embedded hypersurface, $\theta_{M, g}(x, r) = \theta_{[M]_g, g}(x, r)$.

\begin{proof}[Proof of Theorem \ref{thm:main2}]
One could prove this Theorem by first constructing a closed analogue of a cone decomposition, as we did with Theorem \ref{thm:main3}, but since we are only interested in the parameterization it will be less cumbersome to prove it directly by contradiction.

We first make a general note about minimal surfaces in $(N, g)$.  Since $N$ is closed, we can find numbers $R(N, g), \eta(N, g) \in (0, 1]$ so that if $g'$ is any $C^2$ metric on $N$ satisfying $|g - g'|_{C^2(N, g)} \leq \eta$, then every ball $(B_R(x), g')$ is isometric to some $(B_R(0) \subset \R^8, g'_x)$, for $g'_x$ a $C^2$ metric on $B_R(0)$ satisfying $|g'_x - \geucl|_{C^2(B_R(0))} \leq \min\{ c_0^{-1}, 10^{-2} \}$, $c_0$ as in \eqref{eqn:monotonicity}.  In particular, for $V \in \cISV_7(N, g')$ we have the bounds
\[
\theta_{V, g'}(x, r) \leq 2 \theta_{V, g'}(x, R) \leq c_2(N, g) \mu_V(N) \quad \forall x \in N, r \in (0, R).
\]
So by \eqref{eqn:ahlfors}, 
\begin{equation}\label{eqn:closed-3}
V \neq 0 \iff \mu_V(N) \geq \frac{1}{2c_2} \iff \spt V \neq \emptyset,
\end{equation}
and
\begin{equation}\label{eqn:closed-10}
\mu_V(N) \leq \Lambda \implies \theta_{V, g'}(x, r) \leq \Lambda' := c_2 \Lambda \quad \forall x \in N, \forall r < R.
\end{equation}

Let $\sigma = \frac{1}{100(1 + I)}$, and ensure $\eps \leq \min\{ \eps_1(\Lambda'), 10^{-2} \eps_0(\Lambda')\}$, for $\eps_1$ as in Lemma \ref{lem:glue} and $\eps_0$ as in the discussion of Section \ref{sec:param}.  We will choose constants $\delta', \beta, \eps > 0$ as we progress, along with several choices, which a posteriori can all be fixed, but to guide the reader our dependencies look like the following ($\to$ meaning ``depends on''):
\[
i_0 \to \eps_M \to \{B_{r_\alpha}(y_\alpha)\}_\alpha \to M, \delta', \beta, \eps \quad\text{and}\quad \delta' \to \beta \to \eps \to \sigma, \Lambda', I .
\]

Suppose the theorem fails.  Then for every $i$ large there is a set of $C^3$ metric $\cG_i$ on $N$ satisfying $|g' - g|_{C^3(N, g)} \leq 1/i$ for all $g' \in \cG_i$, such that we can find sequences $\{g_{ij}\}_j \subset \cG_i$, $\{M_{ij} \in \cM_7(N, g_{ij})\}_j$, satisfying $\haus^7_{g_{ij}}(M_{ij}) \leq \Lambda$, $\mindex(M_{ij}, N, g_{ij}) \leq I$, with the property that for every $j' > j$ there does not exist a bi-Lipschitz mapping $\phi : (N, g) \to (N, g)$ satisfying Theorem \ref{thm:main2}(\ref{item:main2-1})-(\ref{item:main2-2})-(\ref{item:main2-3}) with $M_{ij}$, $M_{ij'}$, $j'$ in place of $M_v$, $M$, $C$.  As before, this contradiction hypothesis is preserved if we replace $j$ with any subsequence of $j$.

If $M_{ij} = M_{ij'} = \emptyset$ then trivially such a $\phi$ exists, so by \eqref{eqn:closed-3} after discarding finitely-many terms in each sequence we must have $\haus^7_{g_i}(M_{ij}) \geq \frac{1}{2c_2}$ for all $i$, $j$.  After passing to subsequences in $j$, we can assume there are $C^2$ metrics $g_i$ (satisfying $|g_i - g|_{C^2(N, g)} \leq 1/i$) so that $g_{ij} \to g_i$ in $C^2(N, g)$, and (by Lemma \ref{lem:index-compact}) there are non-zero varifolds $V_i \in \cISV_7(N, g_i)$, and discrete sets $\cI_i$ consisting of at most $I$ points, so that $[M_{ij}]_{g_{ij}} \to V_i$ as varifolds in $(N, g)$, and in $C^2$ on compact subsets of $N \setminus (\sing V_i \cup \cI_i)$.  Moreover, we have $\mu_{V_i}(N) \leq \Lambda$, and $\mindex(V_i, N, g_i) \leq I$.

Trivially $g_i \to g$ in $C^2(N, g)$, so passing to a subsequence in $i$ we find a non-zero $V \in \cISV_7(N, g)$, and at most $I$ points $\cI$, so that $V_i \to V$ as varifolds in $(N, g)$, and in $C^2$ on compact subsets of $N \setminus (\sing V \cup \cI)$.  Moreover, since we cannot have points $x_i \in \sing V_i$ converging to a point in $\reg V \setminus \cI$, and $N$ is compact, after enlarging $\cI$ to consist of possibly $2I$ points, we can assume that for every $\tau > 0$ there is an $i_1$ so that $\sing V_i \cup \cI_i \subset B_\tau^g(\sing V \cup \cI)$ for all $i > i_1$.  $V$ satisfies $\mu_V(N) \leq \Lambda$, $\mindex(V, N, g) \leq I$.

Write $\sing V \cup \cI = \{y_1 , \ldots, y_d\} \subset \spt V$, and take $\delta', \eps, \beta > 0$ (for the moment) arbitrary.  By working in normal coordinates centered at each $y_\alpha$, by Lemmas \ref{lem:rs-bound}, \ref{lem:index-compact}, \eqref{eqn:monotonicity}, \eqref{eqn:closed-10}, Theorem \ref{thm:scone}, Remark \ref{rem:scone-stable}, and the maximum principle \cite{ilmanen}, we can choose radii $r_\alpha \in (0, 1/8)$, integers $m_\alpha$, and cones $\bC_\alpha \in \cC$, satisfying $m_\alpha \theta_{\bC_\alpha}(0) \leq \Lambda'$, so that:
\begin{enumerate}[label=A.\arabic*]
\item \label{item:ralpha-1} the balls $\{B^g_{8r_\alpha}(y_\alpha)\}_\alpha$ are disjoint, and each $(B^g_{8r_\alpha}(y_\alpha), g)$ is isometric via normal coordinates to (and will henceforth be identified with) $(B_{8r_\alpha}(0) \subset \R^8, g_\alpha)$ for some $C^3$ metric $g_\alpha$ satisfying $|g_\alpha - \geucl|_{C^3(B_{8r_\alpha}(0), \geucl)} \leq \delta'/2$;

\item \label{item:ralpha-2} $\spt V \cap (A_{8r_\alpha, 0}(0), g_\alpha)$ is a $(\bC_\alpha, 1, \beta/2)$-strong-cone region;

\item \label{item:ralpha-3} for each $\alpha$, we have
\begin{gather}
d_H(\spt V \cap B_{r_\alpha}, \bC_\alpha \cap B_{r_\alpha}) \leq (\delta'/2) r_\alpha \\
\theta_V(0, r_\alpha/2) \geq (m_\alpha - 1/4) \theta_{\bC_\alpha}(0), \quad \theta_V(0, r_\alpha) \leq  (m_\alpha + 1/4)\theta_{\bC_\alpha} .
\end{gather}
\end{enumerate}

Let $O = N \setminus \cup_\alpha B_{2r_\alpha}(y_\alpha)$, $O' = N \setminus \cup_\alpha B_{r_\alpha}(y_\alpha)$, and then we can write $V \llcorner O' = n_1 [S_1] + \ldots + n_p [S_p]$ for $n_s \in \N$, and the $S_s$ being a $C^2$, disjoint, closed (as sets), embedded minimal hypersurface in $(O', g)$.  For each $S_s$, define the map $G_s : T^{\perp_g} S_s \to N$ by $G_s(x, v) = \exp_N|_x(v)$.  By \ref{item:ralpha-2}, we can choose an $0 < \eps_M \leq \eps \min\{ 1, \min_\alpha r_\alpha\}$ so that each $G_s$ is a $C^2$ diffeomorphism from $\{ (x, v) \in T^{\perp_g} S_s : |v| < 4\eps_M \}$ onto its image $O_s$, and $O_s \supset B^g_{4\eps_M}(\spt V) \cap O$, and $G_s$ admits the bound $|DG_s| \leq 2$, and all the images $O_s$ are disjoint.

Let us ensure $\delta'(I, \Lambda', \beta, \sigma)$ is smaller than the constant $\delta_{l, I}(l, I, \gamma = \beta, \sigma)$ from Theorem \ref{thm:diffeo}, for $l$ chosen so that $\tilde\theta_l \geq \Lambda'$.  We can then fix an $i_0$, depending on $\delta', \eps_M, \beta, V, \{B_{r_\alpha}(y_\alpha)\}_\alpha$, and pass to a subsequence in $j$, so that for each $j$ we have:
\begin{enumerate}[label=B.\arabic*]
\item \label{item:i0-1} $M_{i_0 j} \subset B^g_{\eps_M}(\spt V)$;

\item \label{item:i0-2} in each $(B_{8r_\alpha}(0), g_\alpha)$, the metric $g_{i_0 j}$ satisfies $|g_{i_0 j} - \geucl|_{C^3(B_{8r_\alpha}(0), \geucl)} \leq \delta'$;

\item \label{item:i0-3} for every $\alpha$, $M_{i_0 j} \cap ( A_{8r_\alpha, r_\alpha/8}, g_\alpha)$ is a $(\bC_\alpha, m_\alpha, \beta)$-strong-cone region;

\item \label{item:i0-4} we have for every $\alpha$:
\begin{gather}
d_H(M_{i_0 j} \cap B_{r_\alpha}, \bC_\alpha \cap B_{r_\alpha}) \leq \delta' r_\alpha \\
\theta_{M_{i_0 j}, g_{i_0 j}}(0, r_\alpha/2) \geq (m_\alpha - 1/2) \theta_{\bC_\alpha}(0), \quad \theta_{M_{i_0 j}, g_{i_0 j}}(0, r_\alpha) \leq (m_\alpha + 1/2)\theta_{\bC_\alpha} ;
\end{gather}

\item for each $s = 1, \ldots, p$, there are $C^2$ functions $\{ u_{j s k} : S_s \to S_s^{\perp_g} \}_{k=1}^{n_s}$ so that
\[
M_{i_0 j} \cap O_s = \cup_{k=1}^{n_s} \{ G_s(x, u_{jsk}(x)) : x \in S_s \}, \quad |u_{jsk}|_{C^2(S_s)} \leq \eps_M.
\]
\end{enumerate}
Having fixed this $i_0$, let us for simplicity drop the $i$ notation, think only of the sequence $M_j \equiv M_{i_0 j}$ and $g_j \equiv g_{i_0 j}$.  We can forget all the other sequences.  We shall build bi-Lipschitz maps $\phi_j : (N, g) \to (N, g)$ with uniform Lipschitz constant taking $\overline{M_1} \to \overline{M_j}$, which will contradict our choice of $M_j$.

By \ref{item:i0-2}, \ref{item:i0-4}, and since every $r_\alpha \leq 1$, we can apply Theorem \ref{thm:diffeo} to each $[M_{j}]_{g_j}$ in each ball $(B_{r_\alpha}, g_{j})$, to deduce there is a finite collection of $(\Lambda', \sigma, \beta)$-smooth regions $\cS$ and an integer $K$ so that for every $\alpha$ and $j$, there is a radius $r_{\alpha j} \in (\frac{4}{5} r_\alpha, r_\alpha)$ so that $M_{j} \cap (B_{r_{\alpha j}}, g_j)$ admits a $(\Lambda', \beta, \cS, K)$-strong-cone decomposition.  From \ref{item:i0-3}, we deduce that each $M_{j} \cap (B_{2r_\alpha}, g_j)$ admits a $(\Lambda', \beta, \cS, K+1)$-strong-cone decomposition.

Assuming $\beta(\Lambda', \eps, \sigma)$ is small, we can apply Theorem \ref{thm:param}, pass to a subsequence in $j$, and obtain bi-Lipschitz maps $\phi_{\alpha j} : B_{2r_\alpha} \to B_{2r_\alpha}$, cones $\{\hat \bC_\alpha\}_\alpha \subset \cC_\Lambda'$, and an increasing function $C : [0, 1) \to \R$ (independent of $\alpha, j$) satisfying the properties:
\begin{enumerate}[label=C.\arabic*]
\item \label{item:phi-1} $\phi_{\alpha j}(\overline{M_1} \cap B_{2r_\alpha}) = \overline{M_j} \cap B_{2r_\alpha}$, and $\phi_{\alpha j}(\sing M_{1} \cap B_{2r_\alpha}) = \sing M_{j} \cap B_{2r_\alpha}$;

\item \label{item:phi-2} $\phi_{\alpha j}|_{B_{2r_\alpha} \setminus \sing M_{1}}$ is a $C^2$ diffeomorphism;

\item \label{item:phi-3} $\Lip_{g_\alpha}(\phi_{\alpha j}|_{B_{2r_\alpha \rho}}) \leq C(\rho)$ for every $\rho < 1$;

\item \label{item:phi-4} $\phi_{\alpha j}|_{A_{2r_\alpha, 2r_\alpha(1-\sigma)}}$ is a $(\hat\bC_\alpha, \eps)$-map;

\item \label{item:phi-5} $d_H(\hat \bC_\alpha \cap A_{2r_\alpha, 2r_\alpha (1-\sigma)}, M_{j} \cap A_{2r_\alpha, 2r_\alpha (1-\sigma)}) \leq 2r_\alpha \eps$ and $\sing M_{j} \cap B_{2r_\alpha} \subset B_{2r_\alpha(1-\sigma)}$.
\end{enumerate}
We now work towards gluing these $\phi_{\alpha j}$ together.

We first build our large-scale diffeomorphisms.  Define $g_{sj} : T^{\perp_g} S_s \to T^{\perp_g} S_s$ as follows.  Given any $x \in S_s$, there is a neighborhood $U \subset S_s$ containing $x$, in which the normal bundle $T^{\perp_g} ( U \cap S_s)$ is isometric to a trivialization $(U \cap S_s) \times \R \ni (x, t)$.  Under this identification, we can think of the $u_{jsk}$ as taking values in $(-4\eps_M, 4\eps_M)$, and so by the maximum principle and the smallness of the singular set of $M_j$ we can assume
\[
u_{js1} < \ldots < u_{js m_s} \quad \forall j = 1, 2, \ldots, \forall s = 1, \ldots, p.
\]

Let $f_{\{a_j\}_j, \{b_j\}_j}(t)$ be the function from Lemma \ref{lem:monof} with $m_s$ in place of $k$ and $\eps_M$ in place of $\eps$.  For $(x, t) \in (U \cap S_j ) \times \R \cong T^{\perp_g} (U \cap S_j)$, we define
\[
g_{sj}(x, t) = (x, f_{\{u_{1sk}(x)\}_k, \{u_{jsk}(x)\}_k}(t)),
\]
so that $g_{sj}(x, u_{1sk}(x)) = (x, u_{jsk}(x))$ for every $k = 1, 2, \ldots, m_j$.  Trivially, $g_{sj}$ is a $C^2$ diffeomorphism $(U \cap S_s) \times (-4\eps_M, 4\eps_M)$ which coincides with the identity for $|t| \geq 2\eps_M$.  By Remark \ref{lem:monof}, this definition is independent of our choice of trivialization, and so gives rise a diffeomorphism from $T^{\perp_g} S_s$ to itself.  By Item \ref{item:ralpha-3}, the number
\[
D = \inf \{ |u_{1sk}(x) - u_{1s k'}(x)| :  x \in S_s, k \neq k', s = 1, \ldots, p \} > 0,
\]
and hence by Lemma \ref{lem:monof} we have the uniform $C^1$ bound $|Dg_{sj}| \leq c(N, V, M_1)$.

Now define $\psi_j : (O, g) \to (N, g)$ by setting
\[
\psi_j(x) = \left\{ \begin{array}{l l} G_s^{-1} \circ g_{sj} \circ G_s & x \in O_s \\ x & x \not\in O_s \end{array} \right.  .
\]
By our construction and by \ref{item:i0-1}, $\psi_j$ is a well-defined $C^2$ diffeomorphism onto its image, satisfying
\[
\psi_j(M_1 \cap O) = M_j \cap \psi_j(O), \quad |\psi_j(x) - x| \leq 10 \eps_M, \quad |D\psi_j| \leq c(N, V, M_1).
\]

For each $\alpha$, set $S_\alpha' = \spt V \cap A_{8r_\alpha, r_\alpha}$.  Let $\nu_\alpha$ be a choice of unit normal for $S_\alpha'$.  By \ref{item:ralpha-2}, Theorem \ref{thm:cones} and our choice of $\eps_M$, provided $\delta'(\Lambda'), \beta(\Lambda')$ are sufficiently small, the function $G_\alpha : S_\alpha' \times (-4\eps_M, 4\eps_M) \to \R^8$ given by $G_\alpha(x, v) = \exp_{g_\alpha}|_x(t \nu_\alpha(x))$ is a well-defined $C^2$ diffeomorphism onto its image.  If we let $X_\alpha := DG_\alpha(\del_t)$, then after shrinking $\delta'(\Lambda'), \beta(\Lambda')$ and flipping $\nu_\alpha$ as necessary, $X_\alpha$ is a $C^2$ vector field on $B_{3\eps_M}(\bC_\alpha) \cap A_{4r_\alpha, 2r_\alpha}$ satisfying
\[
|X_\alpha - X_{\bC_\alpha}| + r_\alpha |DX_\alpha - X_{\bC_\alpha}| \leq c(\Lambda')(\beta + \delta').
\] 
Ensuring $\beta(\Lambda', \eps)$, $\delta'(\Lambda', \eps)$ are small, and recalling our restriction $\eps_M \leq \eps r_\alpha$, we can argue as in Lemma \ref{lem:smooth-map} and obtain $C^2$ diffeomorphisms $\psi_{\alpha j} : (A_{4r_\alpha, r_\alpha},g_\alpha) \to (B_{8r_\alpha}, g_\alpha)$ such that
\begin{gather}
\psi_{\alpha j}( M_1 \cap A_{4r_\alpha, r_\alpha}) = M_j \cap \psi_{\alpha j}(A_{4r_\alpha, r_\alpha}), \quad |D\psi_{\alpha j}| \leq c(N, V, M_1), \label{eqn:closed-1} \\
\psi_{\alpha j}|_{A_{4r_\alpha, 3r_\alpha}} = \psi_j, \quad \psi_{\alpha j}|_{A_{2r_\alpha, r_\alpha}} \text{ is a $(\bC_\alpha, \eps)$-map}. \label{eqn:closed-2}
\end{gather}
We use the $\psi_{\alpha j}$ to transition from our large-scale $\psi_j$, and the small-scale maps $\phi_{\alpha j}$ arising from the cone decompositions of $M_i$ in each ball $B_{2r_\alpha}$.

\vspace{3mm}

Let $\sigma' = 10^{-2} \sigma$, and $r_\alpha' = (1+5\sigma')^{-1} r_\alpha$.  Then we have
\[
A_\alpha := A_{2r_\alpha' (1+\sigma'), 2r_\alpha'(1-4\sigma')} \subset A_{2r_\alpha, 2r_\alpha(1-\sigma)}.
\]
\ref{item:ralpha-3}, \ref{item:phi-4}, \ref{item:phi-5} imply that $d_H(\bC_\alpha \cap B_1, \hat \bC_\alpha \cap B_1) \leq c(\Lambda')(\eps + \beta)$, and so by ensuring $\beta(\Lambda',\sigma), \eps(\Lambda',\sigma)$ are sufficiently small Theorem \ref{thm:cones} implies that both $X_{\bC_\alpha}$ and $X_{\hat \bC_\alpha}$ are defined on $B_{2\eps_1 r_\alpha'}(\bC_\alpha) \cap A_\alpha$ and satisfy
\[
|X_{\bC_\alpha} - X_{\hat \bC_\alpha}| + r_\alpha' |DX_{\bC_\alpha} - X_{\hat\bC_\alpha}| \leq \delta_1(\Lambda', \sigma') .
\]
Here $\eps_1(\Lambda')$, $\delta_1(\Lambda', \sigma')$ are the constants from Lemma \ref{lem:glue}.

Combined with \eqref{eqn:closed-1}, \eqref{eqn:closed-2}, \ref{item:phi-3}, \ref{item:phi-4}, \ref{item:i0-3}, and our choice of $A_\alpha$, we can therefore apply Lemma \ref{lem:glue} (at scale $B_{2r_\alpha'}$ and with $\sigma', \eps_1$ in place of $\sigma, \eps$) to obtain $C^2$ diffeomorphisms $\hat \phi_{\alpha j} : (A_\alpha, g_\alpha) \to (A_\alpha, g_\alpha)$ satisfying
\begin{gather*}
\hat \phi_{\alpha j}( M_1\cap A_\alpha) = M_j \cap A_\alpha, \quad |D\hat \phi_{\alpha j}| \leq c(N, V, M_1), \\
\hat \phi_{\alpha j}|_{A_{2r_\alpha'(1+\sigma'), 2r_\alpha'}} = \psi_{\alpha j}, \quad \hat \phi_{\alpha j}|_{A_{2r_\alpha'(1-3\sigma'), 2r_\alpha'(1-4\sigma')}} = \phi_{\alpha j} .
\end{gather*}

We now define $\phi_j$ by setting
\[
\phi_j(x) = \left\{\begin{array}{l l} 
\psi_j(x) & x \in N \setminus \cup_\alpha B_{3r_\alpha}(y_\alpha) \\ 
\psi_{\alpha j}(x) & x \in A_{4r_\alpha, 2r_\alpha'} \\
\hat \phi_{\alpha j}(x) & x \in A_\alpha \\
\phi_{\alpha j}(x) & x \in B_{2r_\alpha'(1-3\sigma')} \end{array} \right.
\]
By \ref{item:i0-1} and our construction, the $\phi_j$ are bi-Lipschitz maps $(N, g) \to (N, g)$ satisfying
\begin{gather*}
\phi_j(\overline{M_1}) = \overline{M_j}, \quad \phi_j(\sing M_1) = \sing M_j, \\
\phi_j|_{N \setminus \sing M_1} \text{ is a $C^2$ diffeomorphism}, \quad \Lip_g(\phi_j) \leq c(N, V, M_1).
\end{gather*}
This is a contradiction, and finishes the proof of Theorem \ref{thm:main2}.
\end{proof}


\section{Other corollaries}

Here we prove Theorems \ref{thm:main3} and \ref{thm:main4}.  Theorem \ref{thm:main3} is a trivial consequence of Theorems \ref{thm:diffeo}, \ref{thm:param}, but we include a proof for the sake of precision.  We work in $\R^8$.

\begin{proof}[Proof of Theorem \ref{thm:main3}]
We first note that the existence of $V$, $\cI$, and convergence $M_i \to V$, follows directly from Lemma \ref{lem:index-compact}.

A straightforward contradiction argument, using Theorem \ref{thm:cones}, Lemma \ref{lem:index-compact}, and the constancy theorem implies that provided $\delta(\Lambda, I, \sigma)$ is sufficiently small, we can find integers $m_i \in \{0, 1, \ldots\}$ so that $m_i \theta_\bC(0) \leq \Lambda$, and
\[
\theta_{[M_i]_{g_i}}(0, 1-\sigma) \leq (m_i + 1/2) \theta_\bC(0), \quad \theta_{[M_i]_{g_i}}(0, (1-\sigma)/2) \geq (m_i - 1/2) \theta_\bC(0).
\]
Passing to a subsequence, there no loss in assuming $m_i \equiv m$.  If $m = 0$, then provided we take $\delta'(\Lambda, \sigma)$ small \eqref{eqn:ahlfors} implies $M_i \cap B_{1-2\sigma} = \emptyset$ for all $i$ (and hence $\mu_V(B_{1-2\sigma}) = \emptyset$), so we can just take $r = 1-2\sigma$, $\phi_i \equiv id$.

By Theorem \ref{thm:diffeo}, choosing $\delta'(\Lambda, I, \beta, \sigma)$ small, we can find a finite collection $\cS$ of $(\Lambda, \beta, 5\sigma)$-smooth models, and a constant $K(\Lambda, I, \beta, \sigma)$, so that for all $i >> 1$ there are radii $r_i \in (1-200 \sigma(I+1), 1)$ so that $M_i \cap (B_{r_i}(0), g_i)$ admits a $(\Lambda, \beta, \cS, K)$-cone decomposition.  Passing to a subsequence, we can assume $r_i \to (1+\sigma)r$.

Trivially $\cS$ are also $(\Lambda, \beta, 3\sigma)$-smooth models, and so for every $i >> 1$ and $r' \in ((1-\sigma) r, (1+\sigma)r)$, each $M_i \cap (B_{r'}(0), g_i)$ admits a $(\Lambda, \beta, \cS, K)$-cone decomposition (though not explicitly stated, the outer-most region in the decomposition of Theorem \ref{thm:diffeo} is always a smooth region).  This proves the first assertion.

To prove existence of the $\phi_i$, we apply Theorem \ref{thm:param} at scale $B_{(1+\sigma)r}(0)$ (with $3\sigma$ in place of $\sigma$, $\eps > 0$ arbitrary), and deduce that after passing to a further subsequence, we can find a constant $C$, a cone $\bC \in \cC$, and local bi-Lipschitz maps $\phi_i : B_{(1+\sigma)r} \to B_{(1+\sigma)r}$ satisfying
\begin{gather}
\phi_i(\overline{M_1} \cap B_{(1+\sigma)r}) = \overline{M_i}, \quad \phi_i(\sing M_1 \cap B_{(1+\sigma)r}) = \sing M_i , \label{eqn:main3-1} \\
\Lip(\phi_i|_{B_r}) \leq C,  \label{eqn:main3-2} \\
\phi_i|_{A_{(1+\sigma)r, (1-2\sigma)r}} \text{ is a $(\bC, \eps)$-map}, \quad \sing M_i \subset B_{(1-2\sigma)r}. \label{eqn:main3-3}
\end{gather}
\eqref{eqn:main3-3} implies $|\phi_i(x)| = |x|$ for $x \in A_{(1+\sigma)r, (1-2\sigma)r}$, and so we can restrict $\phi_i|_{B_r}$ to obtain our required maps.

Passing to one last subsequence, we can assume there is a Lipschitz $\phi_\infty : B_r \to B_r$ so that $\phi_i \to \phi_\infty$ in $C^\alpha(B_r)$ for all $\alpha \in (0, 1)$.  Since $|\phi_i(x)| = |x|$ for $|x| \in ((1-2\sigma)r, r)$, we have for all $i$:
\begin{equation}\label{eqn:main3-4}
\max\{ |\phi_i^{-1}(x)|, |\phi_i(x)| \} \leq \max\{ (1-2\sigma) r, |x| \},  \quad |\phi_\infty(x)| \leq \max\{ (1-2\sigma)r, |x| \}.
\end{equation}
Varifold convergence $[M_i]_{g_i} \to V$ in $B_1$, \eqref{eqn:ahlfors}, and \eqref{eqn:main3-4} imply $\phi_\infty(\overline{M_1} \cap B_r) \subset \spt V \cap B_r$.  On the other hand, again from \eqref{eqn:ahlfors} and varifold convergence, if $x \in \spt V \cap B_r$, then there is a sequence $x_i \in \overline{M_1} \cap B_r$ so that $\phi_i(x_i) \to x$.  By \eqref{eqn:main3-4} and since $\overline{M_1}$ is closed, there is no loss in assuming that $x_i \to x' \in \overline{M_1} \cap B_r$.  The uniform Lipschitz bound \eqref{eqn:main3-2} then implies $\phi_\infty(x') = x$.  This proves $\phi_\infty(\overline{M_1} \cap B_r) = \spt V \cap B_r$.
\end{proof}

\begin{proof}[Proof of Theorem \ref{thm:main4}]
I first claim that given $\beta > 0$ there is an $\eta(\Lambda, \beta)$ so that if $V \in \cISV_7(B_{10}, \geucl)$ satisfies
\begin{equation}\label{eqn:eucl-1}
\theta_V(0, 10) \leq \Lambda, \quad \theta_V(0, 10) - \theta_V(0, \eta) \leq \eta, \quad \mindex(V, A_{10, \eta}, \geucl) = 0,
\end{equation}
then we can find $\bC \in \cC_\Lambda$, $m \in \N$ with $m \theta_\bC(0) \leq \Lambda$, and $C^2$ functions $\{ u_1, \ldots, u_m : \bC \cap A_{8,1} \to \bC^\perp \}$, so that
\begin{equation}\label{eqn:eucl-2}
V \llcorner A_{8,1} = \sum_{j=1}^m [G_\bC(u_j) \cap A_{8,1}], \quad |u_j|_{C^2} \leq \beta.
\end{equation}

Suppose otherwise: then there are sequences $\eta_i \to 0$, and $V_i \in \cISV_7(B_{10}, \geucl)$, satisfying \eqref{eqn:eucl-1}, but failing \eqref{eqn:eucl-2} for any $\bC \in \cC_\Lambda$, $m \in \N$ with $m \theta_\bC(0) \leq \Lambda$.  By Lemma \ref{lem:index-compact}, we can find a $V \in \cISV_7(B_{10}, \geucl)$ with $\theta_V(0, 10) \leq \Lambda$, $\mindex(V, B_{10}, \geucl) = 0$, so that $M_i \to M$ as varifolds in $B_{10}$, and $V_i \to V$ smoothly on compact subsets of $A_{10, 0}$.  By \eqref{eqn:sharp-mono}, $V$ is dilation-invariant and stable, and hence $V = m [\bC] \llcorner B_{10}$ for some $\bC \in \cC$, and $m \theta_\bC(0) = \theta_V(0, 10) \leq \Lambda$.  Smooth convergence then implies \eqref{eqn:eucl-2} holds for $i >> 1$, which is a contradiction.  This proves my first claim.

I next claim that given any $\beta , \delta > 0$, and $M \in \cM_7(\R^8, \geucl)$ satisfying \eqref{eqn:main4-hyp}, then we can find a positive $R_M \in \R$, a $\bC \in \cC_\Lambda$, and an $m \in \N$, satisfying $m \theta_\bC(0) \leq \Lambda$, so that the following is true:
\begin{gather}
d_H( \spt M \cap B_{R_M/2}, \bC \cap B_{R_M/2}) \leq \delta R_M/2 \label{eqn:eucl-3} \\
(m-1/2) \theta_\bC(0) \leq \theta_M(0, R_M/4)  \quad \theta_M(0, R_M/2) \leq (m + 1/2) \theta_\bC(0)  \label{eqn:eucl-4} \\
M \cap (A_{\infty, R_M}, \geucl) \text{ is a $(\bC, m, \beta)$-strong-cone region.} \label{eqn:eucl-5}
\end{gather}

By Lemma \ref{lem:index-compact} and \eqref{eqn:sharp-mono}, we can take $m[\bC]$ to be any (the) tangent cone at infinity for $M$.  More precisely, there are radii $R_i \to \infty$ so that $[\eta_{0, R_i,\sharp}(M)] \to m[\bC]$ as varifolds and in the local Hausdorff distance, and smoothly with multiplicity-$m$ on compact subsets of $\R^8 \setminus \{0\}$.  In particular, \eqref{eqn:eucl-3}, \eqref{eqn:eucl-4} hold for $R = R_i$ for any choice of $R_i$ large.

Choose $\beta'(\Lambda) \leq \beta(\Lambda)$ sufficiently small so that if $\bC', \bC'' \in \cC_\Lambda$ satisfy $d_H(\bC' \cap B_1, \bC'' \cap B_1) \leq 2\beta'$, then $\bC' \in \cC(\bC'')$ (as guaranteed by Theorem \ref{thm:cones}).  Choose $\eta(\Lambda, \beta')$ as in Claim 1, and now taking $R$ sufficiently large so that $\theta_M(0, \infty) - \theta_M(0, R) \leq \eta$, $\mindex(M, A_{\infty, R}, \geucl) = 0$, it follows from Claim 1 that $M \cap (A_{\infty, 8 R/\eta}, \geucl)$ is a $(\bC, m, \beta')$-cone region.  It then follows by Theorem \ref{thm:scone}, provided $\beta'(\Lambda, \beta)$ is sufficiently small, that $M \cap (A_{\infty, R_i}, \geucl)$ is a $(\bC, m, \beta)$-strong-cone region for all $i >> 1$.  This proves my second claim.

Let $\sigma = \frac{1}{100(I+1)}$, and let $2 \delta = \delta_{l, I}(l, I, \gamma = \beta, \sigma)$ be the constant from Theorem \ref{thm:diffeo}, with $l$ chosen so that $\tilde \theta_l \geq \Lambda$.  Let $R_M$ be the radius from Claim 2 with this choice of $\delta$, and $\beta > 0$ to be determined later.  It will suffice to prove Theorem \ref{thm:main4} with $\lambda = R_M^{-1}$.

Suppose, towards a contradiction, Theorem \ref{thm:main4} fails.  Then there is a sequence of $M_i \in \cM_7(\R^8, \geucl)$ satisfying \eqref{eqn:main4-hyp} and normalized so that $R_{M_i} = 1$, and a radius $\rho$, with the property that for any $i' > i$, there does not exist a local bi-Lipschitz $\phi : \R^8 \to \R^8$ satisfying
\begin{gather*}
\phi(\overline{M_i}) = \overline{M_{i'}}, \quad \phi(\sing M_i) = \sing M_{i'}\\
\phi|_{\R^8 \setminus \sing M_i} \text{ is a $C^2$ diffeomorphism}, \quad \Lip(\phi|_{B_\rho}) \leq i'.
\end{gather*}

By our normalization, we have for each $i$ an $m_i \in \N$, $\bC_i \in \cC_\Lambda$ satisfying $m_i \theta_{\bC_i}(0) \leq \Lambda$, so that \eqref{eqn:eucl-3}, \eqref{eqn:eucl-4}, \eqref{eqn:eucl-5} hold with $M_i$, $1$, $m_i$, $\bC_i$ in place of $M$, $R_M$, $m$, $\bC$.  By Theorem \ref{thm:cones}, after passing to a subsequence, we can assume there is a fixed $\bC \in \cC_\Lambda$, $m \in \N$ (with $m \theta_\bC(0) \leq \Lambda$), so that \eqref{eqn:eucl-3}, \eqref{eqn:eucl-4}, \eqref{eqn:eucl-5} hold with $M_i$, $1$, $2\delta$, $2\beta$ in place of $M$, $R_m$, $\delta$, $\beta$.

By our choice of $\delta$, we can therefore apply Theorem \ref{thm:diffeo} to each $M_i$ to deduce there is a finite set of $(\Lambda, \sigma, \beta)$-smooth models $\cS$, and a constant $K$, and radii $r_i \in (4/10, 1/2)$, so that each $M_i \cap (B_{r_i}, \geucl)$ admits a $(\Lambda, \beta, \cS, K)$-strong-cone decomposition.  From \eqref{eqn:eucl-5}, we get that each $M_i \cap (B_1, \geucl)$ admits a $(\Lambda, \beta, \cS, K+1)$-strong-cone decomposition.  Ensuring $\beta(\Lambda, \sigma, \eps)$ is sufficiently small and passing to a subsequence, from Theorem \ref{thm:param} we can find a cone $\hat \bC \in \cC_\Lambda$, and local bi-Lipschitz maps $\hat\phi_i : B_1 \to B_1$, satisfying Theorem \ref{thm:param}(\ref{item:param1})-(\ref{item:param5}) with $M_1, M_i, \hat\bC$ in place of $M_v, M, \bC_v$.

On the other hand, ensuring $\eps \leq \eps_0(\Lambda)/8$, we can apply Lemma \ref{lem:cone-map} to obtain $C^2$ diffeomorphisms $\psi_i : A_{\infty, 1/4} \to A_{\infty, 1/4}$ satisfying
\begin{gather*}
\psi_i(M_1 \cap A_{\infty, 1/4}) = M_i \cap A_{\infty, 1/4}, \quad \psi_i|_{A_{1, 1/4}} \text{ is a $(\bC, \eps)$-map }, \\
|D\psi_i|_{C^0(A_{R, 1/2})} \leq c(R, M_1)  \text{ for every $R \in (1/2, \infty)$}.
\end{gather*}
If $m = 1$, then Lemma \ref{lem:cone-map} implies $|D\psi_i| \leq C$ on all of $A_{\infty, 1/4}$.

Theorem \ref{thm:param}(\ref{item:param5}) and \eqref{eqn:eucl-5} imply that $d_H(\bC \cap B_1, \hat \bC \cap B_1) \leq c(\Lambda)( \beta + \eps)$.  Therefore, arguing like we did in the proof of Theorem \ref{thm:main2}, provided $\beta(\Lambda, \sigma), \eps(\Lambda, \sigma)$ are sufficiently small, we can apply Lemma \ref{lem:glue} to the maps $\hat\phi_i$, $\psi_i$, to obtain local bi-Lipschitz $\phi_i : \R^8 \to \R^8$ satisfying
\begin{gather*}
\phi_i(\overline{M_1}) = \overline{M_i}, \quad \phi_i(\sing M_1) = \sing M_i , \quad \phi_i|_{\R^8 \setminus \sing M_1} \text{ is a $C^2$ diffeomorphism}, \\
\Lip(\phi_i|_{B_R}) \leq C(R) \text{ for every $R \in \R$},  \\
\phi_i|_{B_{1-10\sigma'}} = \hat \phi_i, \quad \phi_i|_{\R^8 \setminus B_{1-\sigma'}} = \psi_i,
\end{gather*}
where $\sigma' = \sigma/100$, and $C : \R \to \R$ is an increasing function.  This is a contradiction.
\end{proof}

\end{comment}

\section{Appendix 1: a smallness estimate}\label{sec:decay}

In this section we prove an estimate which is a slight variation on Simon's asymptotic decay Theorem \cite{simon}.   Following \cite{simon}, we take $\Sigma$ a smooth closed $n$-dimensional surface, and $\cE$ a functional defined on $C^1$ functions $u$ given by
\[
\cE(u) = \int_\Sigma E(x, u, \nabla u),
\]
where $E$ satisfies the convexity and analyticity hypotheses of \cite[(1.2), (1.3)]{simon}.  We write $\cM(u)$ for the $-\mathrm{grad} \cE$ in the $L^2$ sense (i.e. so that
\[
\left. \frac{d}{dt} \right|_{t = 0} \cE(u + t \zeta) = - \int_\Sigma \cM(u) \zeta \quad \forall \zeta \in C^1(\Sigma)
\]
holds), and $L$ for the linearization of $\cM$ at $u = 0$.  As in \cite{simon}, we allow $u$ to take values in some vector bundle $V$ over $\Sigma$.

Let us recall some notation from \cite{simon}.  Let $u$ be a $V$-valued function on $\Sigma \times [0, T)$.  Write $\nabla$ for the connection derivative on $V$, and $\dot u \equiv \del_t u$.  Write $|u(t)|_0 = |u(t, \cdot)|_{C^0(\Sigma)}$, $||u(t)|| = ||u(t, \cdot)||_{L^2(\Sigma)}$, and 
\begin{align*}
|u|_1^*(t) &= |u(t)|_0 + |\nabla u(t)|_0 + |\dot u|_0 , \\
|u|_2^*(t) &= |u|_1^*(t) + |\nabla^2 u(t)|_0 + |\nabla \dot u(t)|_0 + |\ddot u(t)|_0 .
\end{align*}

We consider here $C^2$ $V$-valued functions $u$ on $\Sigma \times [0, T)$ solving the PDE
\begin{equation}\label{eqn:mainpde}
-\ddot u + m\dot u - \cM(u) - \cR_1(u) - \cR_2(u) = f(x, t),
\end{equation}
where $m > 0$, and:
\begin{enumerate}
\item \label{item:decay-1} $f(x, t)$ is a $C^2$ function satisfying $|f|_2^*(t) \leq \delta e^{-\eps t}$; 
\item \label{item:decay-2} $\cR_1$ has the form
\[
\cR_1(u) = (a_1 \cdot \nabla^2 u + a_2) \dot u + a_3 \cdot \nabla \dot u + a_4 \ddot u
\]
where $a_i$ are $C^2$ functions of $(x, t, u, \nabla u, \dot u)$ such that $a_i(x, t, 0, 0, 0) = 0$ for $i = 2, 3, 4$; 
\item \label{item:decay-3} $\cR_2$ has the form
\[
\cR_2(u) = b_1 \cdot \nabla^2 u + b_2 \cdot \nabla \dot u + b_3 \ddot u + b_4 \cdot \nabla u + b_5 \dot u + b_6 u,
\]
where $b_i$ are $C^2$ functions of $(x, t)$ such that $|b_i|_2^*(t) \leq \delta e^{-\eps t}$.
\end{enumerate}

The main content of our Theorem \ref{thm:decay} is that one can replace Simon's bounded exponential growth condition \cite[Definition 2]{simon} with a bounded linear growth condition (Definition \ref{def:linear}), at the expense of replacing a decay estimate with a smallness estimate.  Though the name suggests otherwise, our linear growth condition is not a strictly weaker notion than Simon's exponential growth condition -- only when $|u|$ is more-or-less larger than $\delta$ are they comparable.  Fortuitously, this is enough for our purposes.

The advantage is that the linear growth condition is effectively equivalent to smallness of density drop of a minimal surface $M$, \emph{without} any assumptions about the actual value of the density.  For example, the linear growth condition results directly from an assumption like $\sup_r |\theta_M(0, r) - \theta_\bC(0)| \leq \delta^2$.  In contrast, the exponential growth condition of \cite{simon} requires \emph{additionally} the inequality $\sup_r \theta_M(0, r) \geq \theta_\bC(0)$.  This second condition is too strong for our purposes, because in our annular cone regions we only know $M$ is varifold close to $\bC$, and $\theta_M(0, r)$ is close to $\theta_\bC(0)$.  In particular, $\theta_M(0, r)$ may be $> \theta_\bC(0)$ at some radii, but $< \theta_\bC(0)$ at others.

On the other hand, the stronger assumption of Simon also gives a stronger estimate.  The exponential growth condition allows \cite{simon} to rule out exponential growth of $u$ at any time, to deduce that $u$ exhibits only decay.  This also illustrates the asymmetry of the lower bound on $\theta_M(0, r)$: if one instead assumed $\sup_r \theta_M(0, r) \leq \theta_\bC(0)$, one could show $u$ exhibits only growth (corresponding to uniqueness the tangent cone of $M$ at infinity).

With our linear growth condition \eqref{eqn:def-linear}, we can only rule out exponential growth (or decay) when $|u| >> \delta$, and so our Theorem \ref{thm:decay} can only conclude a smallness of $u$.  $u$ may still exhibit both exponential growth and decay.  The linear growth condition is symmetric: smallness of $u$ propogates both forward and backward in time.  Of course, like \cite{simon}, we can also give a summability estimate \eqref{eqn:decay-concl2} for $||\dot u(t)||$, as the decay/growth behavior of $u$ is quantifiable.

\begin{definition}\label{def:linear}
Given $K \geq 1$, $\delta > 0$, we say a solution $u$ of \eqref{eqn:mainpde} on $\Sigma \times [0, T)$ has $(K, \delta)$-linear growth if
\begin{equation}\label{eqn:def-linear}
|u|^*_2(t) \leq \delta \max\{ 1, |t - s| \} + K |u|^*_2(s) 
\end{equation}
for every $t, s, \in [0, T)$.
\end{definition}

Our modification of \cite[Theorem 1]{simon} is the following:
\begin{theorem}\label{thm:decay}
Take $T_* \geq T_0 > 0$, $K > 0$.  There are $\alpha(E), \eps(E) \in (0, 1/2)$, $\delta_0(\eps, K, T_0, E, \cR_1) > 0$ so that the following holds.  Let $\delta \leq \delta_0$, and let $u$ be a $C^2$ solution to \eqref{eqn:mainpde} on $\Sigma \times [0, T_*)$ with $(K, \delta)$-linear growth, such that
\begin{gather}
|u|_2^*(t) \leq \delta \text{ for } t \in [0, T_0], \label{eqn:decay-hyp1} \\
\cE(u(t)) \geq \cE(0) - \delta \text{ for } t \in [0, T_*), \label{eqn:decay-hyp2}
\end{gather}
and Items 1-3 above hold for $f$, $\cR_1$, $\cR_2$ and all $(x, t) \in \Sigma \times [0, T_*)$.  Then
\begin{equation}\label{eqn:decay-concl}
|u|_2^*(t) < \delta^\alpha \quad \forall t \in [0, T_*),
\end{equation}
and
\begin{equation}\label{eqn:decay-concl2}
\int_0^{T_*} ||\dot u(t)|| dt \leq \delta^{\alpha/2}.
\end{equation}
\end{theorem}

\begin{remark}\label{rem:decay}
If $m < 0$, then Theorem \ref{thm:decay} continues to hold if instead of \eqref{eqn:decay-hyp2} we have $\cE(u(t)) \leq \cE(0) + \delta$ for all $t \in [0, T_*)$.
\end{remark}

\begin{proof}
The proof is essentially the same as Simon's.  First note that by changing variables $\tilde t = m t$ and replacing $E$ with $\frac{1}{m} E$, it suffices to consider the case when $m = 1$.  We shall highlight: first, where Simon's growth condition is used, and how it can be replaced with our linear growth assumption; and second, where $\cR_2$ can be handled.

The strategy is to assume (after shrinking $T_*$ if necessary), that
\begin{equation}\label{eqn:decay1}
|u|_2^*(t) < \delta^\alpha \quad \forall t \in [0, T_*), \quad \limsup_{t \to T_*} |u|_2^*(t) = \delta^\alpha.
\end{equation}
and then use this to show that in fact
\begin{equation}\label{eqn:decay2}
\limsup_{t \to T_*} |u|_2^*(t) < \delta^\alpha,
\end{equation}
thereby establishing a contradiction.  The choice of constants $\alpha$, $\eps$, $\eps_1$, $\eps_2$, $\eta$, $R$, $\delta$ will satisfy the same restrictions as in Simon, except we may take $R(\eta, \eps, E, K)$ larger than Simon, and $\delta(\eps, R, K, T_0, E, \cR_1)$ smaller.

As noted in \cite{simon}, if $\zeta$ is a solution to $\cM(\zeta) = 0$ with $|\zeta|_{C^2(\Sigma)} \leq \delta^{2\alpha}$, then by differentiating \eqref{eqn:mainpde} the functions $u$, $w = \dot u$, $v = u - \zeta$ all solve equations of the form
\begin{align}
-\ddot u + \dot u - Lu = e_1 \cdot \nabla^2 u + e_2 \cdot \nabla \dot u + e_3 \ddot u + e_4 \cdot \nabla u + e_5 \dot u + e_6 u + f' ,
\end{align}
where $L$ is the linearization of $\cM$ at $0$, and $e_i(t, x)$, $f'(t, x)$ are $C^1$ functions (different for each choice of $u, v, w$) satisfying $|e_j|_1^*(t) \leq c(E, \cR_1) \delta^\alpha$, $|f'|_1^*(t) \leq 100 \delta e^{-\eps t}$.  Therefore, ensuring $\delta(\eps, \eta, E, R, \cR_1)$ is small, we can apply the growth theorem \cite[Theorem 4]{simon} to each of $u$, $\dot u$, $u - \zeta$ to obtain integers
\[
1 \leq j_1 \leq j_2 \leq k-1, \quad 1 \leq k_1 \leq k_2 \leq k-1, \quad 1 \leq l_1 \leq l_2 \leq k-1 ,
\]
so that \cite[(5.9), (5.10), (5.11)]{simon} hold with $r_1, r_2$ replaced by $j_1, j_2$ (in the case of $u$), $k_1, k_2$ (in the case of $\dot u$), and $l_1, l_2$ (in the case of $u - \zeta$).

The first point where Simon uses his growth assumption is to show that $j_2 = l_2 = k-1$.  To show this using our linear growth assumption we argue as follows.  First note that provided $\delta(R, K)$ is sufficiently small, \eqref{eqn:decay1} and \eqref{eqn:def-linear} imply
\begin{equation}\label{eqn:decay3}
|u|_2^*(t) \geq \delta^\alpha/(2K) \quad \text{ whenever } (k-3)R \leq t < T_*.
\end{equation}
And then, if we had $j_2 < k-1$, using \eqref{eqn:decay3}, \cite[(5.4), (5.11)]{simon}, \eqref{eqn:decay1}, we would obtain the following contradiction for sufficiently large $R(E, K)$ and small $\delta(E, K)$:
\begin{align}
\delta^\alpha/(2K) 
&\leq |u|_2^*( (k-3/2) R) \\
&\leq c(E) ( \sup_{[(k-2)R, (k-1)R]} ||u(t)|| + \delta ) \\
&\leq c(E) ( e^{-(\eps_1 - \eps) R} \sup_{[(k-1)R, kR]} ||u(t)|| + \delta ) \\
&\leq c(E) ( e^{-(\eps_1 - \eps) R} \delta^\alpha + \delta).
\end{align}
So $j_2 = k-1$.  A similar argument shows we must also have $l_2 = k-1$, since $|\zeta|_{C^2(\Sigma)} \leq \delta^{2\alpha} << \delta^\alpha$.

The second (and only other) way in which exponential growth is used is to show that bounds of $||u(t)||$ or $|u|_2^*(t)$ in terms of $\delta^p$, for any particular $p \leq 1$, persist for nearby $t$.  This also follows from our linear growth assumption: if we have
\[
|u|_2^*(t) \leq C \delta^p \quad \text{or} \quad \sup_{[jR, (j+1)R]} ||u(t)|| \leq C \delta^p
\]
for some $C > 0$, $p \leq 1$, $t \in [0, T_*)$ or $0 \leq j < k-1$, then \eqref{eqn:def-linear} and \cite[(5.4)]{simon} imply
\[
|u|_2^*(s) \leq c(C, K, R) \delta^p, \quad \text{or} \quad \sup_{[j'R, (j'+1)R']} ||u(t)|| \leq c(C, K, R) \delta^p
\]
(respectively) for some constant $c(C, K, R)$, provided $|s - t| \leq 10R$, $|j - j'| \leq 6$, $0 \leq j' \leq k-1$, and $\delta(R, K, p)$ is sufficiently small.

The second remainder term $\cR_2$ can be dealt with in the same fashion as $f$.  Specifically, we have by \cite[(6.26)]{simon} and our bound $|u|_2^* \leq \delta^\alpha$ that
\[
||\cR_2(t)|| + ||f(t)|| \leq c(E) \delta^{1/2} ||\dot u(t)|| \quad \forall t \in (k_1 R, (k_2 - 1)R),
\]
which is used in applying \cite[Lemma 1]{simon} to obtain \cite[(6.34)]{simon}; and by \cite[(6.46)]{simon},
\[
||\cR_2((k_2+2)R)|| + ||f((k_2+2)R)|| \leq c(E) \delta^{1/2} \sup_{[(k_2+1)R, (k_2+3)R]} ||\dot u(t)||,
\]
which is used in the estimate for $||v(t)||$ on \cite[Page 557]{simon}.  The rest of the proof of estimate \eqref{eqn:decay-concl} proceeds as in \cite{simon}.

To prove \eqref{eqn:decay-concl2}, first note that we have by \eqref{eqn:decay-concl} the bound $||\dot u(t)|| \leq c(\Sigma) \delta^\alpha$ for all $t$.  Now by \cite[(6.31), (6.34)]{simon} we get
\[
\int_0^{(k_2 - 1)R} ||\dot u|| dt \leq c(\eps) \delta^{1/2} + c(E) \delta^\theta + c(E)\delta^\alpha.
\]
For $k_2 - 1 \leq j \leq k-2$, note that from \cite[(6.41)]{simon} we get $\sup_{[jR, (j+1)R]} ||\dot u(t)|| \leq c(\Sigma) e^{-(\eps_1 - \eps)(j - k + 1) R} \delta^\alpha$, and therefore
\[
\int_{(k_2 - 1)R}^{T_*} ||\dot u|| dt \leq c(E, R) \delta^\alpha \sum_{j=0}^{k-k_2} e^{-(\eps_1 - \eps)j} \leq c(E, R, \eps) \delta^\alpha .
\]
Ensuring $\delta(E, R, \eps)$ is sufficiently small, we obtain \eqref{eqn:decay-concl2}.
\end{proof}


\section{Appendix 2: a function}

In constructing our parameterization we used the following $1$-dimensional Lemma.
\begin{lemma}\label{lem:monof}
Given $\eps \in (0, 1]$, $k \in \N$, let $\cA_{\eps, k} = \{ (a_1, \ldots, a_k) \in \R^k : -\eps < a_1 < a_2 < \ldots < a_k < \eps \}$.  Then there is a smooth function $f \equiv f_{\{a_i\}, \{b_i\}}(t) : \cA_{\eps, k} \times \cA_{\eps, k} \times \R \to \R$ satisfying:
\[
f(a_i) = b_i, \quad f|_{\{ |t| \geq 2\eps\}} = t, \quad  \frac{9}{10 } \underline{M} \leq \del_t f \leq 11 \overline{M}
\]
and
\[
|\del_{a_j} f| \leq  c \overline{M}, \quad |\del_{b_j} f| \leq c ,
\]
where 
\[
\overline{M} = \max\left\{ \max_i \frac{|b_{i+1} - b_i|}{|a_{i+1} - a_i|}, 3 \right\}, \quad \underline{M} = \min \left\{ \min_i \frac{|b_{i+1} - b_i|}{|a_{i+1} - a_i|}, \frac{1}{3} \right\} ,
\]
and $c$ is an absolute constant.
\end{lemma}

\begin{remark}\label{rem:monof}
We highlight two basic facts about $f$.  First, we have $f_{\{a_i\}, \{a_i\}}(t) = t$.  Second, if $\tilde a_i = -a_{k-i}$ and $\tilde b_i = -b_{k-i}$, and $\tilde f = f_{\{ \tilde a_i\}, \{ \tilde b_i\}}$, then we have $f_{\{a_i\}, \{b_i\}}(t) = -\tilde f(-t)$.
\end{remark}

\begin{proof}
Fix $\underline{m}, \overline{m} : \R^2 \to \R$ to be $1$-homogenous functions, smooth away from $0$, such that
\begin{align}
&\min\{ s_1, s_2 \} \leq \underline{m} (s_1, s_2) \leq (101/100) \min\{ s_1, s_2 \}, \quad \underline{m}(s, s) = s, \\
&\quad \overline{m}(s_1, s_2) \geq \max\{ s_1, s_2 \}.
\end{align}
Fix $\eta : \R \to \R$ a smooth function satisfying
\[
\eta|_{(-\infty, -1] \cup [1, \infty)} = 0, \quad \eta|_{[-1/4, 1/4]} = 1, \quad t \eta'(t) \leq 0, \quad |\eta'| \leq 10.
\]

Define $m : \R^3 \to \R$ by setting
\[
m(s_1, s_2, t) = \left\{ \begin{array}{l l} s_1 + \eta(t) ( \underline{m}(s_1, s_2) - s_1) & t < 0 \\ s_2 + \eta(t) (\underline{m}(s_1, s_2) - s_2) & t \geq 0 \end{array} \right. .
\]
Trivially $m$ is smooth on $\{ s_1, s_2 > 0 \} \times \R$, and in this domain $m$ satisfies
\[
m + |\del_t m| \leq 20 \max\{ s_1, s_2\}, \quad |\del_{s_i} m| \leq c, 
\]
where $c = \max |D \underline{m}|$ is an absolute constant.  $m$ is our transition gradient.

Define $\hat f(s_1, s_2, t) = m(s_1, s_2, t) t$.  Trivially $\hat f$ is smooth on $\{ s_1 , s_2 > 0\} \times \R$, and $\hat f$ satisfies
\[
\hat f|_{(-\infty, -1]} = s_1 t, \quad \hat f|_{[1, \infty)} = s_2 t, \quad \hat f(s_1, s_2, 0) = 0.
\]
A straightforward computation shows that
\[
(9/10) \min\{ s_1, s_2 \} \leq \del_t \hat f \leq 11 \max\{ s_1, s_2\}, \quad |\del_{s_i} \hat f| \leq c |t|,
\]
with $c$ an absolute constant.  $\hat f$ is our model transition function, between two linear functions of different positive gradients.

Take $\{a_i\}_{i=1}^k, \{b_i\}_{i=1}^k \in \cA_{\eps, k}$.  Let us define $a_0 = b_0 = -\frac{3}{2} \eps$, $a_{k+1} = b_{k+1} = \frac{3}{2} \eps$, and
\begin{align}
r_i &= \overline{m}\left( \frac{2}{a_{i+1} - a_i}, \frac{2}{a_i - a_{i-1}} \right) \quad i = 1, \ldots, k, \\
m_i &= \frac{b_{i+1} - b_i}{a_{i+1} - a_i}, \quad i = 0, \ldots, k.
\end{align}
Now we let
\[
f(\{a_i\}, \{b_i\}, t) = \left\{ \begin{array}{l l} 
\frac{\eps}{4} \hat f(1, m_0, \frac{4}{\eps} (t + \frac{3}{2}\eps)) - \frac{3}{2} \eps & t \leq a_1 - \frac{1}{r_1} \\
\frac{1}{r_i} \hat f(m_{i-1}, m_i, r_i (t - a_i) + b_i & a_i - \frac{1}{r_i} \leq t \leq a_i + \frac{1}{r_i} \\
m_i (t - a_i) + b_i & a_i + \frac{1}{r_i} \leq t \leq a_{i+1} - \frac{1}{r_{i+1}} \\
\frac{\eps}{4} \hat f(m_{k}, 1, \frac{4}{\eps}(t - \frac{3}{2}\eps)) + \frac{3}{2} \eps & t \geq a_k + \frac{1}{r_k} \end{array} \right. .
\]

Then one can readily check $f$ is smooth, well-defined, and satisfies
\[
(9/10) \min_i \{ m_i \} \leq \del_t f \leq 11 \max_i \{ m_i \}, \quad |\del_{a_j} f| \leq c \max_i \{ m_i\}, \quad |\del_{b_j} f| \leq c ,
\]
for any $j = 1, \ldots, k$, and for $c$ being an absolute constant.  This proves the Lemma. \qedhere
\end{proof}

\section{Appendix 3: quadratic cones} \label{sec:simons}

In this section we prove Theorem \ref{thm:main5}.  We shall work in $\R^{n+1}$, for $n \geq 7$.  We will want to use a slight variant of Theorem \ref{thm:scone}, the proof being virtually verbatim (in fact easier).  Let us state that first.

\begin{theorem}\label{thm:scone-1}
Let $\bC^n$ be a smooth (away from $0$), stationary cone in $\R^{n+k}$.  Given any $\eps > 0$, there are constants $\delta(\bC, k, \eps)$, $\beta(\bC, k, \eps)$ so that the following occurs.  Let $\rho \geq 0$, $g$ be a $C^3$ metric on $B_1$ satisfying $|g - \geucl|_{C^3(B_1)} \leq \beta$, and let $V$ be a stationary integral $n$-varifold in $(B_1, g)$.  Suppose there is a $C^2$ function $u : \bC \cap A_{1, 1/2} \to \bC^\perp$ so that
\begin{gather}
\spt V \cap A_{1, 1/2} = G_\bC(u), \quad |u|_{C^2} \leq \delta \\
\theta_V(0, 1) \leq \theta_\bC(0) + \beta, \quad \theta_V(0, \rho/2) \geq \theta_\bC(0) - \beta.
\end{gather}
Then $u$ can be extended to a $C^2$ function on $\bC \cap A_{1, \rho}$, so that
\[
\spt V \cap A_{1, \rho} = G_\bC(u) \cap A_{1,\rho}, \quad |x|^{-1} |u| + |\nabla u| + |x| |\nabla^2 u| \leq \eps.
\]
\end{theorem}

\begin{proof}
Same as Theorem \ref{thm:scone}, except we use Lemma \ref{lem:cone-extend-1} in place of Lemma \ref{lem:cone-extend}.
\end{proof}

\begin{lemma}\label{lem:cone-extend-1}
Let $\bC^n$ be a smooth (away from $0$), stationary cone in $\R^{n+k}$.  Given $\eps > 0$, there is a $\beta(\bC, k, \eps)$ so that the following occurs.  Let $g$ be a $C^3$ metric on $B_1$ satisfying $|g - \geucl|_{C^2(B_1)} \leq \beta$, and let $V$ be a stationary integral $n$-varifold in $(B_1, g)$ satisfying
\begin{gather}
\spt V \cap A_{1, 1/2} = G_\bC(u), \quad |u|_{C^2} \leq \beta, \label{eqn:cone-ext-1-hyp1} \\
\theta_V(0, 1) \leq \theta_\bC(0) + \beta, \quad \theta_V(0, 1/8) \geq \theta_\bC(0) - \beta. \label{eqn:cone-ext-1-hyp2}
\end{gather}
Then
\begin{equation}\label{eqn:cone-ext-1-concl}
\spt V \cap A_{1, 1/4} = G_\bC(u) \cap A_{1, 1/4}, \quad |u|_{C^2} \leq \eps.
\end{equation}
\end{lemma}

\begin{proof}
Proof by contradiction: otherwise, there are $\beta_i \to 0$, metrics $g_i$ such that $|g_i - \geucl|_{C^3(B_1)} \leq \beta_i$, and $V_i$ being stationary integral varifolds in $(B_1, g_i)$ such that \eqref{eqn:cone-ext-1-hyp1}, \eqref{eqn:cone-ext-1-hyp2} hold with $\beta_i$ in place of $\beta$, but \eqref{eqn:cone-ext-1-concl} fails.  Passing to a subsequence, we can assume $V_i \to V$ as varifolds for some stationary integral varifold $V$ in $(B_1, \geucl)$.  Since $\theta_V(0, 1) \leq \theta_\bC(0)$ and $\lim_{r \downarrow 1/8} \theta_V(0, r) \geq \theta_\bC(0)$, monotonicity implies $V$ is dilation-invariant in $A_{1, 1/8}$.  On the other hand, by \eqref{eqn:cone-ext-1-hyp1} we can assume that $V_i \to [\bC]$ in $C^2$ with multiplicity-one in $A_{1, 1/2}$.  So we must have $V \llcorner A_{1, 1/8} = [\bC] \llcorner A_{1, 1/8}$.  Allard's theorem implies \eqref{eqn:cone-ext-1-concl} holds for $i >> 1$, which is a contradiction.
\end{proof}

\vspace{3mm}

We could use the cone regions from Section \ref{sec:cone-region}, but to avoid clashes of notation and because everything in this proof is vastly easier, let us use instead the following simpler notion of cone regions.
\begin{definition}\label{def:cone-star}
Let $V$ be a varifold in $B_R(a)$.  We say $V \llcorner A_{R, \rho}(a)$ is a $(\bC, \beta, \tau)$-weak-cone$*$ region if for every $r \in (\rho/4, (1-10^{-2}) R]$, there is an $a_r \in B_R(a)$ with $B_r(a_r) \subset B_R(a)$, so that
\[
\theta_\bC(0) - \beta \leq \theta_V(a_r, r) \leq \theta_\bC(0) + \beta,
\]
and if $s \in (\rho/4, (1-10^{-2})R] \cap [r/2, 2r]$ then
\[
|a_r - a_s| \leq \tau \max\{ r, s \} .
\]

We say $V \llcorner A_{R, \rho}(a)$ is a $(\bC, \beta)$-cone$*$ region if for every $r \in (\rho/4, R]$ we have
\[
\theta_\bC(0) - \beta \leq \theta_V(a, r) \leq \theta_\bC(0) + \beta.
\]
\end{definition}

If we think of $\sigma = 10^{-2}$, $m = 1$, and replace $c(\Lambda, m)$ with $c(\bC)$, then Lemma \ref{lem:recenter} and its proof continue to hold for cone$*$ regions.  We are now prepared to prove Theorem \ref{thm:main5}.
\begin{proof}[Proof of Theorem \ref{thm:main5}]
We prove Theorem \ref{thm:main5} by contradiction.  We will choose our constants as we go along, and a posteriori they can all be fixed, but the reader should think of them like:
\[
\beta'' << \tau << \beta' << \eps' << \eps < 1.
\]

Suppose the theorem fails: then there is a sequence of numbers $\delta_i \to 0$, metrics $g_i$ such that $|g_i - \geucl|_{C^3(B_1)} \leq \delta_i$, and stationary integral varifolds $V_i$ in $(B_1, g_i)$, so that \eqref{eqn:main5-hyp} holds with $\delta_i$, $g_i$, $V_i$ in place of $\delta$, $g$, $V$, but \eqref{eqn:main5-concl2} fails for any $a, \lambda, q$ satisfying \eqref{eqn:main5-concl1}.  Passing to a subsequence, standard compactness for integral varifolds implies we can find a stationary integral varifold $V$ in $(B_1, \geucl)$, so that $V_i \to V$ as varifolds.  By our hypotheses and the constancy theorem we must have $V = [\bC]$, and so by Allard's theorem there is no loss in assuming $V_i \to [\bC]$ as varifolds in $B_1$ and in $C^2$ on compact subsets of $B_1 \setminus \{0\}$.

Let $\rho_i$ be the least number such that $V_i \llcorner A_{1, \rho_i}(0)$ is a $(\bC, \beta'', \tau)$-weak-cone$*$ region.  Let $a_i = a_{\rho_i}(V_i)$ as in Definition \ref{def:cone-star}, and then by our convergence $V_i \to [\bC]$ we have $a_i \to 0$, $\rho_i \to 0$.  By Lemma \ref{lem:recenter}, ensuring $\beta''(\bC, \beta')$, $\tau(\bC, \beta')$ are sufficiently small, each $V_i \llcorner A_{9/10, \rho_i}$ is a $(\bC, \beta')$-cone$*$ region, and hence by our $C^2$ convergence and by Theorem \ref{thm:scone-1}, ensuring $\beta'(\bC, \eps')$ is small and $i >> 1$, we can find $C^2$ functions $u_i : \bC \cap A_{9/10, \rho_i/2} \to \bC^\perp$ so that
\begin{equation}\label{eqn:simons-1}
\spt V_i \cap A_{9/10, \rho_i/2}(a_i) = a_i + G_\bC(u_i) \cap A_{9/10, \rho_i/2}, \quad |x|^{-1} |u_i| + |\nabla u_i| + |x| |\nabla^2 u_i| \leq \eps' \leq \eps.
\end{equation}

If $\rho_i = 0$ for infinitely-many $i$ then by \eqref{eqn:simons-1} we obtain a contradiction for $i$ large.  Suppose now every $\rho_i > 0$.  Define the rescaled varifolds $V_i' = (\eta_{a_i, \rho_i})_\sharp V_i$, and metrics $g_i' = g_i \circ \eta_{a_i, \rho_i}^{-1}$.  Then for every $R > 2$ and $i >> 1$, we have
\begin{equation}\label{eqn:simons-2}
1 = \inf \{ \rho : V_i' \llcorner A_{R, \rho}(0) \text{ is a $(\bC, \beta'', \tau)$-weak-cone$*$ region} \},
\end{equation}
and $V_i'$ is stationary in $(B_R, g_i')$, and $g_i' \to \geucl$ in $C^2$ on compact subsets of $\R^{n+1}$.

Monotonicity \eqref{eqn:monotonicity} implies that $\theta_{V_i'}(0, R) \leq \theta_\bC(0) + o(1)$, and so we can pass to a subsequence and obtain a stationary integral varifold $V'$ in $(\R^{n+1}, \geucl)$ satisying $\theta_{V'}(0, \infty) \leq \theta_\bC(0)$, so that $V_i' \to V'$ as varifolds.  From \eqref{eqn:simons-1}, Arzela-Ascoli, and standard elliptic estimates, provided $\eps'(\bC)$ is small, we can assume that
\begin{equation}\label{eqn:simons-3}
\spt V' \cap A_{\infty, 1/2} = G_\bC(u')\cap A_{\infty, 1/2}, \quad |x|^{-1} |u'| + |\nabla u'| + |x| |\nabla^2 u'| \leq \eps',
\end{equation}
and $V_i' \to V'$ in $C^2$ on compact subsets of $A_{\infty, 1/2}$.  We can also assume that $V' \llcorner A_{\infty, 1}$ is a $(\bC, \beta')$-cone$*$ region.

Together the previous two sentences imply that if $V''$ is a tangent cone of $V'$ at infinity, then $d_H(\spt V'' \cap B_1, \bC \cap B_1) \leq \eps'$, and $|\theta_{V''}(0) - \theta_\bC(0)| \leq \beta'$.  Ensuring $\eps'(\bC)$, $\beta'(\bC)$ is small, we deduce by integrability of $\bC$ that $V'' = [q(\bC)]$ for some rotation $q \in SO(n+1)$ satisfying $|q - Id| \leq c(\bC) \eps'$.  \cite{SiSo} or \cite[Corollary 3.7]{me-luca} then implies that $V' = [b + q(S_\lambda)]$ for some leaf $S_\lambda$ of the Hardt-Simon foliation, where we allow $S_0 = \bC$.  We aim to show $\lambda \neq 0$.

I claim there is an $\eta(\bC, \tau)$ so that if $a \in B_1$, $r \in [1/2, 4]$, and
\[
|\theta_\bC(a, r) - \theta_\bC(0)| \leq \eta,
\]
then $|a| \leq 2^{-10}\tau$.  Suppose otherwise: there are sequences $\eta_i \to 0$, $a_i \in B_2$, $r_i \in [1/2, 4]$ so that $|\theta_\bC(a_i, r_i) - \theta_\bC(0)| \leq \eta_i$, but $|a_i| \geq 2^{-10}\tau$.  Passing to a subsequence, we can assume $a_i \to a \in \overline{B_1}$, $r_i \to r \in [1/2, 2]$, and so by monotonicity we get $\theta_\bC(a, 2r) = \theta_\bC(0)$, and hence $\bC \cap A_{\infty, 2r}(a) = (a + \bC') \cap A_{\infty, 2r}(a)$ for some cone $\bC'$.  Since $\bC$ has no symmetries, we must have $\bC' = \bC$ and $a = 0$.  This is a contradiction.

Let us ensure $\beta'' \leq \eta(\bC, \tau)$, and $\tau \leq 2^{-4}$.  Assume that $V' = [b + q(\bC)]$.  By \eqref{eqn:simons-3}, $V'$ is regular outside $\overline{B_{1/2}}$, so we must have $b \in \overline{B_{1/2}}$.  I next claim that if $a_{i, r} = a_r(V_i')$, then for all $i >> 1$ and $r \in [1/2, 2]$ we must have
\[
|a_{i, r} - b| \leq 2^{-9} \tau.
\]
Note that by construction we have $a_{i, 1} = 0$, and (hence) $a_{i, r} \in B_{1/4}$.  If my second claim failed, then we could find a sequence $r_i \in [1/2, 2]$ so that $|a_{i, r_i}| \geq 2^{-8}\tau$.  Passing to a subsequence, we can assume $a_{i, r_i} \to a \in \overline{B_{1/4}}$ and $r_i \to r \in [1/2, 2]$, and hence $|\theta_{b + q(\bC)}(a, s) - \theta_\bC(0)| \leq \beta''$ for all $s > r$.  From my first claim (applied in the ball $B_1(b)$), we deduce that $|a - b| \leq 2^{-10}\tau$.  This proves my second claim.

Varifold convergence $V_i' \to V'$ implies $|\theta_{V_i'}(b, r) - \theta_\bC(0)| \leq \beta''$ for $i >> 1$ and $r \in [1/16, 2]$, and therefore combined with the previous paragraph we get that $V_i' \llcorner B_4(0) \setminus B_{1/2}$ is a $(\bC, \beta'', \tau)$-cone$*$ region for $i$ large.  This contradicts \eqref{eqn:simons-2}, so we must have $\lambda \neq 0$.

Now $V' = [b + q(S_\lambda)]$ is smooth, multiplicity-one, and hence by Allard's theorem $V_i' \to V'$ in $C^2$ with multiplicity-one.  Since $|b| \leq 1$ and $\spt V'' \cap B_1 \neq \emptyset$, we have an upper bound $|\lambda| \leq \lambda_0(\bC)$, and hence there is an $R(\bC, \eps')$ so that
\begin{equation}\label{eqn:simons-5}
(b + q(S_\lambda)) \cap A_{\infty, R} = G_{q(\bC)}(v'), \quad |x|^{-1} |v'| + |\nabla v'| + |x| |\nabla^2 v'| \leq \eps'.
\end{equation}
For $i >> 1$, by $C^2$ convergence we can find $u_i' : (b + q(S_\lambda)) \cap B_{2R} \to S_\lambda^\perp$ so that
\begin{equation}\label{eqn:simons-6}
\spt V_i' \cap B_{2R} = \graph_{b + q(S_\lambda)}(u_i'), \quad |x|^{-1} |u_i'| + |\nabla u_i'| + |x| |\nabla^2 u_i'| \leq \eps'.
\end{equation}

If we let $b_i + q(S_{\lambda_i}) = \eta_{a_i, \rho_i}^{-1}(b + q(S_\lambda))$, and ensure $\eps'(\bC, \eps)$ is sufficiently small, then \eqref{eqn:simons-1}, \eqref{eqn:simons-5}, \eqref{eqn:simons-6} imply \eqref{eqn:main5-concl1}, \eqref{eqn:main5-concl2} hold with $b_i, q, \lambda_i$ in place of $a_i, q, \lambda_i$.  This is a contradiction, and finishes the proof of Theorem \ref{thm:main5} for quadratic cones.

For general strictly-minimizing and strictly-stable $\bC^n \subset \R^{n+1}$, and $\spt V$ lying to one side of $\bC^n$, the proof proceeds in a similar but even simpler fashion.  Take $V_i$, $g_i$ the counter-example sequence as before, and then let $\rho_i$ be the least radius so that $V_i \llcorner A_{9/10, \rho_i}(0)$ is a $(\bC, \beta')$-cone$*$ region.  Our convergence $V_i \to [\bC]$ implies $\rho_i \to 0$, and ensuring $\beta'(\bC, \eps)$ is small, we can apply Theorem \ref{thm:scone-1} to deduce \eqref{eqn:simons-1} holds with $a_i = 0$.

Let $V_i' = (\eta_{0, \rho_i})_\sharp V_i$, and $g_i' = g \circ \eta_{0, \rho_i}^{-1}$.  Then for every $R > 2$ and $i >> 1$ we have
\begin{equation}\label{eqn:simons-12}
1 = \inf \{ \rho : V_i' \llcorner A_{R, \rho}(0) \text{ is a $(\bC, \beta')$-cone$*$ region} \} ,
\end{equation}
$V_i'$ is stationary in $(B_R(0), g_i')$, $\spt V_i'$ lies to one-side of $\bC^n$, and $g_i' \to \geucl$ in $C^2(B_R(0))$.

As before, by monotonicity and standard compactness we can pass to a subsequence and assume that $V_i' \to V'$, for $V'$ a stationary integral varifold in $(\R^{n+1}, \geucl)$ such that $\theta_{V'}(0, \infty) \leq \theta_\bC(0)$, and $\spt V'$ lies to one side of $\bC^n$.  \cite[Lemma 7.6]{simon:liousville} implies that $V' = [S_\lambda]$ for some $\lambda$.  If we had $\lambda = 0$, then we would have $\theta_{V_i'}(0, r) \to \theta_{V'}(0, r) = \theta_\bC(0)$ for every $r > 0$, and hence by monotonicity $V_i' \llcorner A_{4, 1/2}(0)$ would be a $(\bC, \beta')$-cone$*$ region for $i >> 1$, contradicting \eqref{eqn:simons-12}.  So we must have $\lambda \neq 0$, and then the proof proceed in the same way as above.
\end{proof}

\bibliographystyle{alpha}
\bibliography{refs}

\end{document}